\documentclass[11pt, reqno]{amsart} 
\usepackage{amssymb,amscd,amsfonts,amsbsy}
\usepackage{latexsym}
\usepackage{exscale}
\usepackage{amsmath,amsthm,amsfonts}
\usepackage{mathrsfs}
\usepackage{xcolor} 
\usepackage[colorlinks=true,linkcolor=blue,citecolor=red,urlcolor=red, 
]{hyperref} 
\usepackage{esint} 

\usepackage[utf8]{inputenc}
\usepackage{lineno}

\makeatletter
\@namedef{subjclassname@2020}{%
  \textup{2020} Mathematics Subject Classification}
\makeatother

\calclayout

\allowdisplaybreaks[3] 

\newtheorem{theorem}{Theorem}[section]

\newtheorem{lemma}[theorem]{Lemma}
\newtheorem{corollary}[theorem]{Corollary}
\newtheorem{example}[theorem]{Example}

\theoremstyle{definition}
\newtheorem{definition}[theorem]{Definition}
\theoremstyle{remark}
\newtheorem{remark}[theorem]{Remark}

\numberwithin{equation}{section} 

\def\dist{\operatorname{dist}}
\def\Lip{\operatorname{Lip}}
\def\supp{\operatorname{supp}}
\def\BMO{\operatorname{BMO}}

\def\loc{\operatorname{loc}}
\def\Re{\operatorname{Re}}

\def\C{\mathbb{C}}
\def\D{\mathcal{D}}
\def\F{\mathcal{F}}
\def\M{\mathbb{M}}
\def\N{\mathbb{N}}
\def\A{\mathcal{A}}

\def\S{\mathcal{S}}
\def\T{\mathcal{T}}
\def\R{\mathbb{R}}
\def\Rn{\mathbb{R}^n}
\def\Sn{\mathbb{S}^{n-1}}
\def\X{\mathbb{X}}
\def\Y{\mathbb{Y}}
\def\Z{\mathbb{Z}} 
\def\V{\mathcal{V}}
\def\b{\mathbf{b}}
\def\w{\omega} 
\def\m{\mathfrak{m}} 

\newcommand{\norm}[1]{\left\lVert#1\right\rVert}

\DeclareMathOperator*{\esssup}{ess\,sup}

\begin{document}

\author{Mingming Cao}
\address{Mingming Cao\\
Instituto de Ciencias Matem\'aticas CSIC-UAM-UC3M-UCM\\
Con\-se\-jo Superior de Investigaciones Cient{\'\i}ficas\\
C/ Nicol\'as Cabrera, 13-15\\
E-28049 Ma\-drid, Spain} \email{mingming.cao@icmat.es}

\author{Andrea Olivo}
\address{Andrea Olivo\\
Departamento de Matem\'atica\\ 
Facultad de Ciencias Exactas y Naturales\\ 
Universidad de  Buenos Aires and IMAS-CONICET\\
Pabell\'on I (C1428EGA), Ciudad de Buenos Aires, Argentina} \email{aolivo@dm.uba.ar}

\thanks{The first author is supported by Spanish Ministry of Science and Innovation through the Ram\'{o}n y Cajal  2021 (RYC2021-032600-I), through the ``Severo Ochoa Programme for Centres of Excellence in R\&D'' (CEX2019-000904-S), and through PID2019-107914GB-I00, and by the Spanish National Research Council through the ``Ayuda extraordinaria a Centros de Excelencia Severo Ochoa'' (20205CEX001). The second author was supported by Grants UBACyT 20020170100430BA (University of Buenos Aires), PIP 11220150100355 (CONICET) and PICT 2018-03399}

\date{November 30, 2022}

\title{Two-weight extrapolation on function spaces and applications}

\subjclass[2020]{42B20, 42B25, 42B35, 46E30}


\keywords{Rubio de Francia extrapolation, 
Endpoint estimates, 
Banach function spaces, 
Local decay estimates, 
Muckenhoupt-Wheeden conjecture, 
Two-weight inequalities, 
Vector-valued inequalities}

\begin{abstract} 
This paper is devoted to studying the extrapolation theory of Rubio de Francia on general function spaces. We present endpoint extrapolation results including $A_1$, $A_p$, and $A_\infty$ extrapolation in the context of Banach function spaces, and also on modular spaces. We also include several applications that can be easily obtained using extrapolation: local decay estimates for various operators, Coifman--Fefferman inequalities that can be used to show some known sharp $A_1$ inequalities, Muckenhoupt--Wheeden and Sawyer's conjectures are also presented for many operators, which go beyond Calder\'{o}n--Zygmund operators. Finally, we obtain two-weight inequalities for Littlewood--Paley operators and Fourier integral operators on weighted Banach function spaces.
\end{abstract}

\maketitle

\section{Introduction}\label{sec:intro}
One of the most useful and powerful tools in harmonic analysis is the Rubio de Francia extrapolation theorem \cite{Rub}, which states that given an operator $T$, 
\begin{equation}
\begin{array}{c}
\text{if $T$ is bounded on $L^{p_0}(w_0)$ for some $p_0 \in [1, \infty)$ and for all $w_0 \in A_{p_0}$}, 
\\[4pt]
\text{then $T$ is bounded in $L^p(w)$ for all $p \in (1, \infty)$ and for all $w \in A_p$}. 
\end{array}
\end{equation} 
Over the years, Rubio de Francia's result has been extended and complemented in different manners, see \cite{CMP11} and the references therein. These results only involve the $L^p(w)$ boundedness of an operator $T$, which can be generally formulated by the pair of functions.  In addition to the extrapolation theorems aforementioned, there are several other kinds of extrapolation. For example, B\'enyi et al. \cite{BMMST} obtained comprehensive extrapolation results from $T$ to its commutator $[b,T]$ on weighted Lebesgue spaces. Recently, Hyt\"onen and Lappas \cite{HL-1, HL-2} established a ``compact version'' of Rubio de Francia's extrapolation theorem, which allows one to extrapolate the compactness of an operator from just one space to the full range of weighted spaces, provided that the operator is bounded. This result has been extended to the multilinear case in \cite{COY}. Another kind of extrapolation, related with second order elliptic operators, was presented by Shen \cite{Shen}, where it was proved that one can extrapolate the solvability of the $L^p$-Dirichlet problem on Lipschitz domains from on some exponent $p_0$ to a large range of $p$'s. Beyond that, Hofmann and Martell \cite{HM12, HM14} built the extrapolation of Carleson measure in order to investigate the uniform rectifiability and $A_{\infty}$ property of elliptic measure. To sum up, a common point in all extrapolation theorems is that one can obtain the global or local information from an estimate at certain single point. 

The Rubio de Francia extrapolation theorem and its variants have been proved to be extremely advantageous and the key to solving many problems in harmonic analysis. Indeed, extrapolation theorems allow us to reduce the general $L^p$ estimates for certain operators to a suitable case $p=p_0$, for example, see \cite{CXY} for the Coifman-Fefferman's inequality for $p_0 = 1$, \cite{Hy12} for the Calder\'{o}n-Zygmund operators for $p_0 = 2$, and \cite{L14} for square functions for $p_0=3$. Even more, the technique of extrapolation can refine some estimates, see \cite{CMP05} for the Sawyer conjecture, \cite{LOP, LOP1} for the weak Muckenhoupt-Wheeden conjecture and \cite{OPR} for the local decay estimates. Also, using extrapolation theorems, one can obtain sharp weighted inequalities for several operators, see \cite{BMMST} and \cite{LMPT}. Another interesting point is that by  means of extrapolation, the vector-valued inequalities immediately follow from the corresponding scalar-valued inequalities, see Theorems \ref{thm:BFSAi}, \ref {thm:wMw}, \ref{thm:FS-X} and \ref{thm:ModAi} below.

An important feature of the proof of Rubio de Francia extrapolation theorems in \cite{CMP11} is that it makes clear exactly what the essential ingredients are. They are three-fold: norm inequalities for the Hardy-Littlewood maximal operator $M$, duality, and the reverse factorization property of $A_p$ weights. More precisely, to extrapolate, we only need the following: 
\begin{enumerate}
\item[(a)] $M$ is bounded on $L^p(w)$, 

\item[(b)] $M'_w$ is bounded on $L^{p'}(w)$, 

\item[(c)] If $w_1, w_2 \in A_1$, then $w_1 w_2^{1-p} \in A_p$,  

\item[(d)] H\"{o}lder's inequality,
\end{enumerate}
where $M'_w f := M(fw)/w$. Observing these, recently, Mar\'{i}n, Martell and the first author \cite{CMM} generalized the Rubio de Francia extrapolation theory to the context of Banach Function spaces over measure spaces, which include Lorentz spaces, variable Lebesgue spaces, Orlicz spaces, their weighted version, etc. To do this, they formulated the weighted boundedness of Hardy-Littlewood maximal operators on weighted Banach function spaces, which is the natural substitute of Muckenhoupt weights in the general setting.

In this paper, we are interested in two kinds of weak type estimates. The first one is called local decay estimates, which concerns the following quantity: 
\begin{align}\label{def:local}
\psi_t(\mathfrak{T}, \mathfrak{M}) 
:= \sup_{Q:\text{ cube } \subset \Rn} \sup_{\substack{f \in L_c^{\infty}(\Rn) \\ \supp(f) \subset Q}} 
|Q|^{-1} |\{x \in Q: |\mathfrak{T}f(x)| > t\, \mathfrak{M}f(x)\}|, \quad t>0, 
\end{align} 
where $\mathfrak{T}$ is a singular operator and $\mathfrak{M}$ is an appropriate maximal operator. In general, one would like to obtain an exponential/sub-exponential decay with respect with $t$. Local decay estimates accurately reflect the extent that an operator is locally controlled by certain maximal operator, which refines the corresponding good-lambda inequality. We here mention that the latter was used to pursue the sharp dependence on  the weight's norm (cf. \cite{Buc} for the maximal singular integrals). These kinds of estimates in \eqref{def:local} appeared for the first time in the work of Karagulyan \cite{Kar} and was further investigated by Ortiz-Caraballo, P\'{e}rez and Rela \cite{OPR} using a novel strategy.

The second one is the Muckenhoupt-Wheeden conjecture, which states that for every standard Calder\'{o}n-Zygmund operator $T$, 
\begin{align}\label{eq:MW-11}
\|Tf\|_{L^{1, \infty}(w)} \lesssim \|f\|_{L^1(Mw)} \quad\text{ for any weight } w, 
\end{align}
where the implicit constant is independent of $w$. Evidently, \eqref{eq:MW-11} implies the weak  Muckenhoupt-Wheeden conjecture: 
\begin{align}\label{eq:MW-12}
\|Tf\|_{L^{1,\infty}(w)} \lesssim [w]_{A_1} \|f\|_{L^1(w)}, \quad\text{ for any } w \in A_1.  
\end{align}
Unfortunately, \eqref{eq:MW-11} and \eqref{eq:MW-12} have been disproved by Reguera and Thiele \cite{RT} for the Hilbert transform, and by Nazarov et al. \cite{NRVV} obtaining a lower bound $[w]_{A_1} \log^{\frac13}(e + [w]_{A_1})$, respectively.  This leads to the question whether one can determine the optimal exponents $\alpha$ and $\beta$ such that the following holds:
\begin{align}\label{eq:wwAA} 
\|Tf\|_{L^{1,\infty}(w)} 
\lesssim  [w]_{A_1}^{\alpha} \log^{\beta}(e+[w]_{A_1}) \|f\|_{L^1(w)}, \quad\text{ for any } w \in A_1, 
\end{align}
where the operator $T$ is not necessarily a Calder\'{o}n-Zygmund operator. For the Hilbert transform, in  \cite{LNO} was proved that that $\alpha=\beta=1$.  

On the other hand, a somehow ``dual" version of \eqref{eq:MW-11} can be formulated by 
\begin{align}\label{eq:mwdd}
\bigg\|\frac{Tf}{Mw}\bigg\|_{L^{1, \infty}(w)} 
\lesssim \|f\|_{L^1(\Rn)} \quad\text{for any weight } w. 
\end{align}
A substitute of \eqref{eq:mwdd} with $M^3w$ in place of $Mw$ was proved in \cite{LOP2} for the standard Calder\'{o}n-Zygmund operators. As argued above, it is interesting to establish \eqref{eq:mwdd} with $M^{1+\varepsilon}w$ instead of $Mw$. Additionally, the inequalities of the sort \eqref{eq:mwdd} are called sometimes in the literature mixed weak type estimates, for which a more general form is given by the following: 
\begin{align}\label{eq:mwsaw}
\norm{\frac{T(fv)}{v}}_{L^{1,\infty}(uv)} \lesssim \|f\|_{L^1(uv)}.  
\end{align}
Such estimates originated in the work of Muckenhoupt and Wheeden \cite{MW}, where \eqref{eq:mwsaw} for $v^{-1} \in A_1$ and $uv \equiv 1$ was established for the Hardy-Littlewood maximal operator $M$ and the Hilbert transform $H$ on the real line $\R$. Later on it was developed by Sawyer \cite{Saw} to the case $u, v \in A_1$ but only for $M$ on $\R$. Also, Sawyer conjectured that \eqref{eq:mwsaw} is true for $H$ and for M but in higher dimensions. A positive answer to both conjectures was given in \cite{CMP05} by using extrapolation arguments.  Recently, Sawyer's conjecture for $M$ was extended to the setting of $u \in A_1$ and $v \in A_{\infty}$ in \cite{LiOP} by means of some delicate techniques. 

Motivated by the work above, the current paper will build extrapolation theorems to study weak type estimates aforementioned. Our extrapolation will be given in the endpoint case, which means the weights $u$ and/or $v$ belong to $A_1$ and/or $A_{\infty}$ classes. In view of the general framework used in \cite{CMM}, our work is carried out on Banach function spaces. 
Considering the weights in $A_1$ or $A_{\infty}$, we construct only one Rubio de Francia iteration algorithm in each proof, which is quite different from that in \cite{CMM}. 

Hereafter, a family of extrapolation pairs $\F$ is a collection of pairs $(f, g)$ of nonnegative measurable functions. More definitions and notation are given in Section \ref{sec:weights}. We begin with two results related to the so-called $A_p$ extrapolation, which will be used to the aforementioned weak-type estimates obtain local decay estimates in \eqref{def:local} for plenty of operators.

\begin{theorem}\label{thm:Aq}
Let $\F$ be a family of extrapolation pairs. Suppose that $\X$ is a rearrangement invariant Banach function space over $(\Rn, dx)$ with $q_{\X}<\infty$. If for some $p_0 \in [1, \infty)$ and for every $w \in A_{p_0}$, 
\begin{align}\label{eq:Aq-1}
\|f\|_{L^1(w)} \leq \Psi \big([w]_{A_{p_0}}\big) \|g\|_{L^1(w)}, \quad (f, g) \in \F, 
\end{align}
where $\Psi:[1, \infty) \to [1, \infty)$ is an increasing function, then for every $u^{1-{p_0}'} \in A_1$ and for every $v \in A_1$, 
\begin{align}\label{eq:Aq-2}
\|fu\|_{\X_v} \leq 4 \Psi \big(2\|M_v\|_{\X'_v \to \X'_v} 
[u^{1-{p_0}'}]_{A_1}^{p_0-1} [v]_{A_1} \big) \|gu\|_{\X_v}, \quad (f, g) \in \F. 
\end{align}
Here, $u$ is understood as the constant $1$ in the case $p_0=1$.
\end{theorem}

\begin{corollary}\label{cor:Ap}
Let $\F$ be a family of extrapolation pairs. Suppose that $\X$ be a rearrangement invariant Banach function space over $(\Rn, dx)$ such that $q_{\X}<\infty$. If for some $p_0 \in (1, \infty)$ and for every $w \in A_{p_0}$, 
\begin{equation}\label{eq:Ap-1}
\|f\|_{L^{p_0}(w)} \leq C \|g\|_{L^{p_0}(w)}, \quad(f, g) \in \F.
\end{equation}
then for every $p \in (1, \infty)$, for every $u^{1-p'} \in A_1$, and for every $v \in A_1$,   
\begin{equation}\label{eq:Ap-2}  
\|f^p u\|_{\X_v} \leq C \|g^p u\|_{\X_v}, \quad (f,g) \in \F. 
\end{equation}
\end{corollary}

We present an $A_1$ extrapolation result in the language of the boundedness of the Hardy-Littlewood maximal operators, which should be compared with \cite[Theorems 3.1]{CMM} in the case $p_0=1$.

\begin{theorem}\label{thm:BFSA1}
Let $\F$ be a family of extrapolation pairs. Suppose that $u$ and $v$ are weights on $\Rn$ such that $v \in A_1$, and $\X_v$ is a Banach function space over $(\Rn, v\,dx)$. If for some $p_0 \in (0,\infty)$ and for every $w \in A_1$, 
\begin{equation}\label{eq:BFSA1-2}
\|f\|_{L^{p_0}(w)} \leq \Psi([w]_{A_1}) \|g\|_{L^{p_0}(w)}, \quad (f,g) \in \F, 
\end{equation} 
where $\Psi:[1,\infty) \to [1,\infty)$ is a non-decreasing function, then 
\begin{equation}\label{eq:BFSA1-3}
\|fu\|_{\X_v^{p_0}} \leq 4^{1/p_0} \Psi(2 K_0 [v]_{A_1}) \|gu\|_{\X_v^{p_0}},\quad (f,g) \in \F, 
\end{equation}
provided that 
\begin{align}\label{eq:BFSA1-1} 
\|(M_v f) u^{-p_0}\|_{\X'_v} \le K_0 \|f u^{-p_0}\|_{\X'_v},\quad\forall f \in \M. 
\end{align} 
\end{theorem}

Next, we formulate a series of $A_\infty$ extrapolation below.

\begin{theorem}\label{thm:BFSAi}
Let $\F$ be a family of extrapolation pairs. Suppose that $u$ and $v$ are weights on $\Rn$ such that $v \in A_{\infty}$, and $\X_v$ is a Banach function space over $(\Rn, v\,dx)$. If for some $p_0 \in (0,\infty)$ and for every $w \in A_{\infty}$, 
\begin{equation}\label{eq:BFSAi-2}
\|f\|_{L^{p_0}(w)} \leq C \|g\|_{L^{p_0}(w)}, \quad (f,g) \in \F, 
\end{equation} 
then for every $p \in (0, \infty)$, 
\begin{equation}\label{eq:BFSAi-3}
\|f^p u\|_{\X_v} \leq C \|g^p u\|_{\X_v},\quad (f,g) \in \F, 
\end{equation}
provided that 
\begin{align}\label{eq:BFSAi-1} 
\|(M_v f) u^{-1}\|_{\X'_v} \le K_0 \|f u^{-1}\|_{\X'_v},\quad\forall f \in \M. 
\end{align} 
Moreover, under the hypothesis \eqref{eq:BFSAi-1}, \eqref{eq:BFSAi-2} implies that for every $p, q \in (0, \infty)$, 
\begin{equation}\label{eq:BFSAi-vec}
\bigg\|\Big(\sum_j f_j^q\Big)^{\frac{p}{q}} u\bigg\|_{\X_v} 
\leq C \bigg\|\Big(\sum_j g_j^q\Big)^{\frac{p}{q}}\bigg\|_{\X_v}, \qquad \{(f_j, g_j)\}_j \subset \F. 
\end{equation}
\end{theorem}

Compared with Theorem \ref{thm:BFSAi}, the following result directly asserts which weight classes $u$ and $v$ belong to, instead of the boundedness of Hardy-Littlewood maximal operators in \eqref{eq:BFSAi-1}. 

\begin{theorem}\label{thm:AiAi} 
Let $\F$ be a family of extrapolation pairs. Suppose that $\X$ is a rearrangement invariant Banach function space over $(\Rn, dx)$ such that $q_{\X}<\infty$. If for some $p_0 \in(0, \infty)$ and for every $w \in A_{\infty}$, 
\begin{equation}\label{eq:AiAi-1}
\|f\|_{L^{p_0}(w)} \leq C \|g\|_{L^{p_0}(w)}, \quad(f, g) \in \F, 
\end{equation}
then for every $u \in RH_{\infty}$ and for every $v \in A_{\infty}$,  
\begin{equation}\label{eq:AiAi-2}  
\|fu\|_{\X_v} \leq C \|gu\|_{\X_v}, \quad (f,g) \in \F. 
\end{equation}
\end{theorem}

A variant of Theorem \ref{thm:AiAi} is the following extrapolation, which considers the type of estimates $\|f/u\|_{\X_v}$ in place of $\|fu\|_{\X_v}$. Correspondingly, the weight classes of $u$ and $v$ are given in terms of $A_1$ and $A_{\infty}$. It is worth pointing out that such estimates are key ingredients to establish Coifman-Fefferman inequalities, which together with duality in turn implies sharp $A_1$ inequalities similar to \eqref{eq:wwAA} for the Calder\'{o}n-Zygmund operator and its maximal truncation, rough singular integrals, Bochner-Riesz means, and so on (cf. Section \ref{sec:app}). 

\begin{theorem}\label{thm:AA} 
Let $\F$ be a family of extrapolation pairs. Suppose that $\X$ is a rearrangement invariant Banach function space over $(\Rn, dx)$ such that $q_{\X}<\infty$. If for some $p_0 \in(0, \infty)$ and for every $w \in A_{\infty}$, 
\begin{equation}\label{eq:AA-1}
\|f\|_{L^{p_0}(w)} \leq C \|g\|_{L^{p_0}(w)}, \quad(f, g) \in \F, 
\end{equation}
then for every $u \in A_1$ and for every $v \in A_{\infty}$,  or for every $u \in A_{\infty}$ and for every $v \in A_1$,  
\begin{equation}\label{eq:AA-2} 
\Big\|\frac{f}{u}\Big\|_{\X_v} \leq C \Big\|\frac{g}{u}\Big\|_{\X_v}, \quad (f,g) \in \F. 
\end{equation}
\end{theorem}

\begin{theorem}\label{thm:Mvr} 
Let $\F$ be a family of extrapolation pairs. Suppose that $u$ and $v$ are weights such that $v \in L^1_{\loc}(\Rn)$, and $\X_v$ is a Banach function space over $(\Rn, v\,dx)$. Assume that there exist $r>1$ and $s>0$ such that $u^{\frac1s} v^{-\frac{1}{sr'}} \in A_1$ and 
\begin{align}\label{eq:Mvr-1}
\|M'_{v^{1/r}} f\|_{\X'_v} \le K_0 \|f\|_{\X'_v},\quad \forall f \in \M.  
\end{align}
If for some $p_0 \in(0, \infty)$ and for every $w \in A_{\infty}$, 
\begin{equation*}
\|f\|_{L^{p_0}(w)} \leq C \|g\|_{L^{p_0}(w)}, \quad(f, g) \in \F, 
\end{equation*} 
then 
\begin{equation*}
\Big\|\frac{f}{u}\Big\|_{\X_v} \leq C \Big\|\frac{g}{u}\Big\|_{\X_v}, \quad (f,g) \in \F. 
\end{equation*}
\end{theorem}

Furthermore, considering the type of estimates in \eqref{eq:mwsaw} and Sawyer's conjecture for operators beyond the Calder\'{o}n-Zygmund theory, we present an extrapolation on weighted Banach function spaces $\X_{uv}$. In Section \ref{sec:app}, one can see that it is a refinement of \cite[Theorem 1.7]{CMP05} on Banach function spaces. 

\begin{theorem}\label{thm:Saw} 
Let $\F$ be a family of extrapolation pairs. Suppose that $u$ and $v$ are weights on $\Rn$ such that $v \in A_{\infty}$ and $uv \in L^1_{\loc}(\Rn)$, and $\X_{uv}$ is a Banach function space over $(\Rn, uv\,dx)$.  Assume that there exist $q_0>1$ and $K_0>0$ such that 
\begin{align}\label{eq:Saw-1}
\|M'_u f\|_{(\X^{q'}_{uv})'} \le K_0 \|f\|_{(\X^{q'}_{uv})'},\quad \forall q \ge q_0. 
\end{align}
If for some $p_0 \in(0, \infty)$ and for every $w \in A_{\infty}$, 
\begin{equation}\label{eq:Saw-2}
\|f\|_{L^{p_0}(w)} \leq C \|g\|_{L^{p_0}(w)}, \quad(f, g) \in \F.
\end{equation}
then 
\begin{equation}\label{eq:Saw-3}
\Big\|\frac{f}{v}\Big\|_{\X_{uv}} \leq C \Big\|\frac{g}{v}\Big\|_{\X_{uv}}, \quad (f,g) \in \F, 
\end{equation}
\end{theorem}

\begin{remark}
Under the same hypotheses, one can also obtain that there exists a small constant $\epsilon>0$ depending on only $q_0$ and $[v]_{A_{\infty}}$ such that 
\begin{equation}\label{eq:Saw-4}
\Big\|\frac{f}{v}\Big\|_{\X^{1+\epsilon}_{uv}} \leq C \Big\|\frac{g}{v}\Big\|_{\X^{1+\epsilon}_{uv}}, \quad (f,g) \in \F. 
\end{equation}
In fact, in the proof below, it suffices to pick $r>1$ and $\epsilon>0$ such that $0<\frac{1}{r'}<\min\{\varepsilon_0, \frac{1}{q_0}\}$ and $r(1+\epsilon)<q'_0$. On the other hand, if the assumption $v \in A_{\infty}$ is replaced by $v \in RH_{\infty}$, then the conclusion \eqref{eq:Saw-3} still holds provided  a weaker condition than \eqref{eq:Saw-1}: 
\begin{align*}
\|M'_u f\|_{(\X^r_{uv})'} \le K_0 \|f\|_{(\X^r_{uv})'} \quad\text{for some } r>1. 
\end{align*} 
\end{remark}

In view of  the dual version of Muckenhoupt-Wheeden conjecture in \eqref{eq:mwdd}, we establish an exotic extrapolation below, where the hypothesis is given in terms of $A_{\infty}$ weights but the conclusion  is asserted for any weight. It should be compared with Theorem \ref{thm:AA}, since Theorem \ref{thm:M3w} with $w \in A_1$ implies Theorem \ref{thm:AA} for $u=v \in A_1$. 

\begin{theorem}\label{thm:M3w}
Let $\F$ be a family of extrapolation pairs. Suppose that $\X$ is a rearrangement invariant Banach function space over $(\Rn, dx)$ with $q_{\X}<\infty$. If for some $p_0 \in (0, \infty)$ and every $w_0 \in A_{\infty}$, 
\begin{align}\label{eq:M3w-1}
\|f\|_{L^{p_0}(w_0)} \le C \|Mg\|_{L^{p_0}(w_0)},\quad (f, g) \in \F, 
\end{align}
then for every weight $w$, 
\begin{align}\label{eq:M3w-2}
\norm{\frac{f}{M^3w}}_{\X(Mw)} \leq C \norm{\frac{g}{Mw}}_{\X(Mw)}, \quad (f, g) \in \F. 
\end{align}
\end{theorem}

Beyond that, we study the extrapolation with arbitrary weights on Banach function spaces, which extends \cite[Theorem 1.3]{CP} to the weighted Banach function spaces.  
\begin{theorem}\label{thm:wMw}
Let $\F$ be a family of extrapolation pairs. Suppose that $\X$ is a rearrangement invariant Banach function space over $(\Rn, dx)$ with $q_{\X}<\infty$ such that $\X^{\frac{1}{p_0}}$ is a Banach function space for some $p_0 \in (0,\infty)$. If for every weight $w_0$,  
\begin{equation}\label{eq:hyp-M}
\norm{f}_{L^{p_0}(w_0)} \leq C \norm{g}_{L^{p_0}(Mw_0)}, \quad (f,g) \in \F, 
\end{equation} 
then for every weight $w$,  
\begin{align}\label{eq:con-M}
\|f\|_{\X_w} \leq C \|g (Mw/w)^{\frac{1}{p_0}}\|_{\X_w},\quad (f, g) \in \F. 
\end{align}
Moreover, if $\X^{\frac1q}$ is a BFS for some $q \in (0,\infty)$, then 
\begin{align}\label{eq:vec-M}
\bigg\|\bigg(\sum_{j=1}^{\infty}|f_j|^q\bigg)^{\frac1q}\bigg\|_{\X_w} 
\leq C \bigg\|\bigg(\sum_{j=1}^{\infty} |g_j|^q \bigg)^{\frac1q} \bigg(\frac{Mw}{w}\bigg)
^{\frac1q}\bigg\|_{\X_w}, \quad \{(f_j, g_j)\}_{j} \subset \F. 
\end{align}
\end{theorem}

\begin{theorem}\label{thm:uvw}
Let $u$, $v$, $w_1$ and $w_2$ be weights on $\Rn$. Let $\Phi$ be a Young function or $\Phi(t)=t$. Suppose that $\X_u$ and $\X_v$ are respectively Banach function spaces over $(\Rn, u\, dx)$ and $(\Rn, v\,dx)$ such that $\Y_u=\X_u^{\frac{1}{p_0}}$ and $\Y_v=\X_v^{\frac{1}{p_0}}$ are BFS for some $p_0 \in (0,\infty)$. Assume that 
\begin{align}\label{eq:uvw-1}
\|(M_{\Phi}f) w_2^{-p_0}v^{-1}\|_{\Y'_v} \le C \|fw_1^{-p_0}u^{-1}\|_{\Y_u'}, \quad\forall f \in \M.  
\end{align}
If for every weight $w_0$,  
\begin{equation}\label{eq:uvw-2}
\|f\|_{L^{p_0}(w_0)} \leq C \|g\|_{L^{p_0}(M_{\Phi} w_0)}, \quad (f,g) \in \F. 
\end{equation} 
then 
\begin{align}\label{eq:uvw-3}
\|fw_1\|_{\X_u} \leq C \|g w_2\|_{\X_v},\quad (f, g) \in \F. 
\end{align}
\end{theorem}

\begin{remark}
Theorem $\ref{thm:uvw}$ extends \cite[Theorem~8.2]{CMP11} to the weighted Banach function spaces. Indeed, let $A$, $B$ and $\Phi$ be Young functions such that $A \in B_p$ and $A^{-1}(t)B^{-1}(t) \lesssim \Phi^{-1}(t)$ for any $t>0$. Let $u=v$, $\X_u=\X_v=L^p(v)$, $r=p/p_0>1$ and $(w_1, w_2)=(U^{\frac1p}, V^{\frac1p})$. With these notation in hand, we see that $\Y'_u=\Y'_v=L^{r'}(v)$ and \eqref{eq:uvw-1} is equivalent to 
\begin{align}\label{eq:MUV}
\|M_{\Phi}f\|_{L^{r'}((vV)^{1-r'})} \lesssim \|f\|_{L^{r'}((vU)^{1-r'})}. 
\end{align}
Moreover, \eqref{eq:uvw-3} is equivalent to 
\begin{align*}
\|f\|_{L^p(vU)} \le C \|g\|_{L^p(vV)}, \quad (f, g) \in \F. 
\end{align*}
In view of Lemma $\ref{lem:M-uv}$, the inequality \eqref{eq:MUV} holds if 
\begin{align*}
\sup_{Q} \|(vU)^{\frac1r}\|_{A,Q} \|(vV)^{-\frac1r}\|_{r',Q} = 
\sup_{Q} \|((vV)^{1-r'})^{\frac{1}{r'}}\|_{r', Q} \|((vU)^{1-r'})^{-\frac{1}{r'}}\|_{A, Q} < \infty. 
\end{align*}
\end{remark}

\begin{theorem}\label{thm:ABC}
Let $\Phi$ be a Young function or $\Phi(t)=t$, $A$ and $B$ be Young functions such that $A^{-1}(t) B^{-1}(t) \lesssim \Phi^{-1}(t)$. Suppose that $u$ and $v$ are weights on $\Rn$, and $\X_v$ and $\X_{M_A v}$ are Banach function spaces over $(\Rn, v\,dx)$ and over $(\Rn, M_Av\,dx)$ respectively. Assume that for for some $p_0 \in (0,\infty)$ and for every weight $w_0$,  
\begin{equation}\label{eq:ABC-1}
\|f\|_{L^{p_0}(w_0)} \leq C \|g\|_{L^{p_0}(M_{\Phi} w_0)}, \quad (f,g) \in \F. 
\end{equation} 
\begin{enumerate}
\item[{\rm (i)}] If $\Y_v=\X_v^{\frac{1}{p_0}}$ is a BFS such that 
\begin{align}\label{eq:ABC-2}
\|M'_{B, v}f\|_{\Y'_v} \le C \|f\|_{\Y'_v}, \quad\forall f \in \M, 
\end{align}
then 
\begin{align}\label{eq:ABC-3}
\|f u\|_{\X_v} \leq C \|g M_A(u^{p_0})^{\frac{1}{p_0}}\|_{\X_v},\quad (f, g) \in \F. 
\end{align}
\item[{\rm (ii)}] If both $\Y_v=\X_v^{\frac{1}{p_0}}$ and $\Y_{M_A v}=\X_{M_A v}^{\frac{1}{p_0}}$ are BFS such that 
\begin{align}\label{eq:ABC-4}
\|M'_{B, u^{p_0}}f\|_{\Y'_{M_A v}} \le C \|f\|_{\Y'_v}, \quad\forall f \in \M, 
\end{align}  
then 
\begin{align}\label{eq:ABC-5} 
\|f u\|_{\X_v} \leq C \|gu\|_{\X_{M_A v}}, \quad (f, g) \in \F. 
\end{align}
\end{enumerate}
\end{theorem}

Finally, let us give vector-valued Fefferman-Stein inequalities on the general Banach function spaces with arbitrary weights.  We also use it to recover a sharp weighted inequality in \cite[Theorem 1.4]{CP} and \cite[Theorem 1.1]{P00}.

\begin{theorem}\label{thm:FS-X}
Let $1<q<\infty$ and $v$ be a weight on $\Rn$. Suppose that $\X_v$ is a Banach function space over $(\Rn, v\,dx)$ such that $\X_v^{\frac1q}$ is also a Banach function space. Let $A$ be a Young function. If $M'_{\bar{A}, v}$ is bounded on $(\X_v^{\frac1q})'$, then for every weight $u$,  
\begin{align}\label{eq:MMA-1}
\bigg\|\Big(\sum_{j=1}^{\infty} |\mathfrak{M}f_j|^q\Big)^{\frac1q} u\bigg\|_{\X_v} 
\leq C \bigg\|\Big(\sum_{j=1}^{\infty} |f_j|^q\Big)^{\frac1q} M_{A}(u^q)^{\frac1q}\bigg\|_{\X_v}, 
\end{align}
where the maximal operator $\mathfrak{M}$ is one of the following: 
\begin{align*}
& (1)\,\, \mathfrak{M}=M; \quad (2)\,\, \mathfrak{M}=M_{\Phi} \text{ with $\Phi \in B_p$ for some $p \in (1, \infty)$}; 
\vspace{0.3cm}\\
& (3)\,\, \mathfrak{M}=M_{r,s} \text{ with $1<r<\infty$ and $1 \le s<\infty$}.  
\end{align*}
\end{theorem}

\begin{corollary}\label{cor:FS-Lp}
Let $1<q<p<\infty$. There exists a constant $C=C(n, p, q)$ such that for every weight $w$,  
\begin{align}\label{eq:vec-FS}
\bigg\|\Big(\sum_{j=1}^{\infty} |Mf_j|^q\Big)^{\frac1q}\bigg\|_{L^p(w)} 
\leq C \bigg\|\Big(\sum_{j=1}^{\infty} |f_j|^q\Big)^{\frac1q}\bigg\|_{L^p(M^{[p/q]+1}w)}. 
\end{align}
Moreover, the inequality \eqref{eq:vec-FS} is sharp in the sense that the exponent $[p/q]+1$ cannot be replaced by $[p/q]$. 
\end{corollary}

The paper is organized as follows. In Section \ref{sec:weights}, we present some preliminaries, definitions, and some auxiliary results that will be used throughout the paper. We also introduce Banach function spaces and present some examples under the framework of the general function spaces.  Section \ref{sec:BFS} is devoted to proving our main results, Theorem \ref{thm:Aq}--Corollary \ref{cor:FS-Lp}. In Section \ref{sec:modular}, for a Young function $\Phi$ and a weight $w$, we build some extrapolation results on modular spaces $\rho_w^{\Phi}$. Theorem \ref{thm:ModAi} can be viewed as a two-weight version of \cite[Theorem 3.1]{CGMP}. Then Corollary \ref{cor:BFSMod} is a hybrid of Theorem \ref{thm:BFSAi} and Theorem \ref{thm:ModAi}. As a consequence, Corollary \ref{cor:TM} gives a control of an operator $T$ by the maximal operator $M_A$ locally.  In addition, compared to \cite[Theorem 4.36]{CMM}, Theorems \ref{thm:PhiA1} and \ref{thm:PhiAi} are respectively formulated by the boundedness of $M_v$ instead of $M'_v$. In Section \ref{sec:comm}, following the ideas in \cite{BMMST}, we establish sharp two-weight extrapolation for commutators in the context of weighted Banach function spaces (cf. Theorem \ref{thm:TTb}). Then Bloom type estimates for commutators originated in \cite{Blo85} are given in Theorems \ref{thm:TaTb} and \ref{thm:CbIa}. Finally, Section \ref{sec:app} contains a variety of applications of extrapolation theorems, such as local decay estimates,  Coifman-Fefferman type's inequalities, Muckenhoupt-Wheeden and Saywer's conjectures for many operators, and two-weight inequalities for certain Littlewood-Paley operators and Fourier integral operators.

\section{Preliminaries}\label{sec:weights}

\subsection{Muckenhoupt weights}
We briefly recall the notion of Muckenhoupt weights and the relevant properties in this section. Given a locally integrable function $f$ on $\Rn$, the Hardy-Littlewood maximal operator $M$ is defined by 
\begin{equation}\label{eq:HL}
Mf(x):=\sup_{Q \ni x} \fint_{Q} |f(y)| \, dy, \qquad x \in \Rn,
\end{equation}
where the supremum is taken over all cubes $Q \subset \Rn$ containing $x$.

We say that a measurable function $w$ defined on $\Rn$ is a {\tt weight} if $0<w(x)<\infty$ a.e. $x\in\Rn$. Given $p\in(1,\infty)$, we define the {\tt Muckenhoupt class} $A_p$ as the collection of all weights $w$ satisfying 
\begin{equation}\label{eq:Def:A_p}
[w]_{A_p}:=\sup_{Q \subset \Rn} \bigg(\fint_{Q}w(x)\,dx \bigg) 
\bigg(\fint_{Q} w(x)^{1-p\,'} dx \bigg)^{p-1}<\infty,
\end{equation} 
where the supremum is taken over all cubes $Q\subset\Rn$ and $p'$ is the H\"older conjugate exponent of $p$, i.e., $1/p+1/p'=1$.
In the case $p=1$, we say that $w\in A_1$ if
\begin{equation*}
[w]_{A_1}:=\|Mw/w\|_{L^{\infty}(\Rn)} <\infty.
\end{equation*}
Then, we define
\begin{equation*}
A_{\infty}=\bigcup_{p\geq 1}A_p.
\end{equation*} 
For every $p\in(1,\infty)$ and every weight $w$, we define the associated weighted Lebesgue space $L^p(w):=L^p(\Rn,\,w\,dx)$ as the set of measurable functions $f$ with $\int_{\Rn} |f|^p\,w\,dx<\infty$. 

Let $1\le p\le q<\infty$. We say that $w \in A_{p,q}$ if it satisfies 
\begin{align*}
[w]_{A_{p,q}} 
:= \sup_{Q} \bigg(\fint_{Q} w^q \, dx \bigg)  \bigg(\fint_{Q} w^{-p'} dx \bigg)^{\frac{q}{p'}}<\infty.  
\end{align*}
Observe that $w \in A_{p,q}$ if and only if $w^q \in A_{1+\frac{q}{p'}}$ if and only if $w^{-p'} \in A_{1+\frac{p'}{q}}$. Moreover, 
\begin{align}\label{eq:wApq}
[w]_{A_{p,q}} = [w^q]_{A_{1+\frac{q}{p'}}} = [w^{-p'}]_{A_{1+\frac{p'}{q}}}^{\frac{q}{p'}}. 
\end{align} 
Together with $p\le q$, the equation \eqref{eq:wApq} implies that 
\begin{align}\label{eq:ppqq}
w \in A_{p,q} \quad\Longleftrightarrow\quad w^p \in A_p \quad\text{and}\quad w^q \in A_q.  
\end{align}
 
For $s\in(1,\infty]$, we define the {\tt reverse H\"{o}lder class} $RH_s$ as the collection of all weights $w$ such that 
\begin{equation*}
[w]_{RH_s} := \sup_{Q} \bigg(\fint_Q w^s\,dx \bigg)^{\frac1s} \bigg(\fint_Q w\,dx \bigg)^{-1}<\infty. 
\end{equation*}
When $s=\infty$, $(\fint_Q w^s\,dx)^{1/s}$ is understood as $(\esssup_{Q}w)$. 

Given $p\in(1,\infty)$ and a weight $w$, the Muckenhoupt's theorem \cite{Muc} states that the Hardy-Littlewood maximal operator $M$ is bounded on $L^p(w)$ if and only if $w\in A_p$. More precisely, there is a constant $C=C(n,p)$ such that
\begin{equation}\label{eq:M-sharp}
\|M\|_{L^p(w) \to L^p(w)} \leq C\,[w]_{A_p}^{\frac{1}{p-1}}.
\end{equation}

The properties of Muckenhoupt weights  are well-known (cf., e.g \cite{GR}). We mention some of them here for completeness.
\begin{list}{\textup{(\theenumi)}}{\usecounter{enumi}\leftmargin=1cm \labelwidth=1cm \itemsep=0.2cm 
			\topsep=.2cm \renewcommand{\theenumi}{\alph{enumi}}}
\item For all $1\leq p\leq q\leq\infty$, $A_p \subseteq A_q$ with $[w]_{A_q} \leq [w]_{A_p}$, $\forall w\in A_p$.  
\item Let $w \in A_p$ with $p\in(1,\infty)$. Then there are $\varepsilon\in(0,p-1)$ and $\tau\in(1,\infty)$ such that $w \in A_{p-\varepsilon}$ and $w^{\tau} \in A_p$. 
\item For every $p\in(1,\infty)$, $w \in A_p$ if and only if $w^{1-p'}\in A_{p'}$, and $[w^{1-p'}]_{A_{p'}}=[w]_{A_p}^{p'-1}$. 
\item For every $p\in(1,\infty)$ and $w_1,w_2\in A_1$,  
\begin{equation}\label{eq:rev-fac}
w_1\,w_2^{1-p}\in A_p \quad\text{with}\quad [w_1\,w_2^{1-p}\,]_{A_p}\leq[w_1]_{A_1}\,[w_2]^{p-1}_{A_1}. 
\end{equation}
\item $w\in A_{\infty}$ if and only if $w \in RH_s$ for some $s\in(1,\infty)$.  
\item For any positive Borel measure $\mu$, 
\begin{align}\label{eq:CR}
[(M \mu)^{\delta}] \le \frac{c_n}{1-\delta}, \quad\forall \delta \in (0, 1). 
\end{align}
As a consequence, 
\begin{equation}\label{eq:MRH}
1 \le [(M\mu)^{-\lambda}]_{RH_{\infty}} \le c_{n,\lambda},\quad\forall \lambda>0. 
\end{equation}
\end{list}

By $w \in A_p(u)$, we mean that $w$ satisfies the $A_p$ condition defined with respect to the measure $u\,dx$. 
The properties below considering the endpoint case were given in \cite{CMP05}. 
\begin{lemma}\label{lem:A1Ap} 
The following statements hold: 
\begin{list}{\textup{(\theenumi)}}{\usecounter{enumi}\leftmargin=1cm \labelwidth=1cm \itemsep=0.2cm 
			\topsep=.2cm \renewcommand{\theenumi}{\arabic{enumi}}}
\item\label{eq:A1RH} If $u \in A_1$, then $u^{-1} \in RH_{\infty}$. 
\item\label{eq:AiRH} If $u \in A_{\infty}$ and $v \in RH_{\infty}$, then $uv \in A_{\infty}$. 
\item\label{eq:RHs} If $u \in RH_{\infty}$, then $u^s \in RH_{\infty}$ for any $s>0$.  
\item\label{eq:uu} $u \in A_{\infty}$ if and only if $u=u_1 u_2$, where $u_1 \in A_1$ and $u_2 \in RH_{\infty}$. 
\item\label{eq:Apu}If $u \in A_1$ and $v \in A_p(u)$ with $1 \leq p < \infty$, then $uv \in A_p$ with 
$[uv]_{A_p} \leq [u]_{A_1}^p [v]_{A_p(u)}$. 
\item\label{eq:vA1u} If $u \in A_p$ with $1\le p \le \infty$ and $v \in A_1(u)$, then $uv \in A_p$ with 
$[uv]_{A_p} \leq [u]_{A_p} [v]_{A_1(u)}$. 
\item\label{eq:A1A1} If $u \in A_1$ and $v \in A_p$ with $1 \leq p < \infty$, then there exists $0<\epsilon_0<1$ depending only 
on $[u]_{A_1}$ such that $uv^{\epsilon} \in A_p$ for all $0<\epsilon<\epsilon_0$. 
\end{list}
\end{lemma}

\subsection{Orlicz maximal operators}  

A function $\Phi:[0,\infty) \to [0,\infty)$ is called a {\tt Young function} if it is continuous, convex, strictly increasing, and satisfies
\begin{equation*}
\lim_{t\to 0^{+}}\frac{\Phi(t)}{t}=0 \quad\text{and}\quad \lim_{t\to\infty}\frac{\Phi(t)}{t}=\infty.
\end{equation*}
Given $p \in[1, \infty)$, we say that a Young function $\Phi$ is a {\tt $p$-Young function}, if $\Psi(t)=\Phi(t^{1/p})$ is a Young function. 

If $A$ and $B$ are Young functions, we write $A(t) \simeq B(t)$ if there are constants $c_1, c_2>0$ such that 
$c_1 A(t) \leq B(t) \leq c_2 A(t)$ for all $t \geq t_0>0$. Also, we denote $A(t) \preceq B(t)$ if there exists $c>0$ such that $A(t) \leq B(ct)$ for all $t \geq t_0>0$. Note that for all Young functions $\phi$, $t \preceq \phi(t)$. Further, if $A(t)\leq cB(t)$ for some $c>1$, then by convexity, $A(t) \leq B(ct)$.

A function $\Phi$ is said to be {\tt doubling}, or $\Phi \in \Delta_2$, if there exists a constant $C>0$ such that $\Phi(2t) \leq C \Phi(t)$ for any $t>0$. Given a Young function $\Phi$, its complementary function $\bar{\Phi}:[0,\infty) \to [0,\infty)$ is defined by 
\[
\bar{\Phi}(t):=\sup_{s>0}\{st-\Phi(s)\}, \quad t>0, 
\]
which clearly implies that
\begin{align}\label{eq:stst}
st \leq \Phi(s) + \bar{\Phi}(t), \quad s, t > 0.
\end{align}
Moreover, one can check that $\bar{\Phi}$ is also a Young function and
\begin{equation}\label{eq:Young-1}
t \leq \Phi^{-1}(t) \bar{\Phi}^{-1}(t) \leq 2t, \qquad t>0.
\end{equation}
In turn, by replacing $t$ by $\Phi(t)$ in first inequality of \eqref{eq:Young-1}, we obtain
\begin{equation}\label{eq:Young-2}
\bar{\Phi} \Big(\frac{\Phi(t)}{t}\Big) \leq \Phi(t), \qquad t>0.
\end{equation}

Let us recall the lower and upper dilation indices of a positive increasing function $\Phi$ on $[0,\infty)$, which are respectively defined by
\begin{equation}\label{eq:Ii}
i_{\Phi}:=\lim_{t\to0^{+}} \frac{\log h_{\Phi}(t)}{\log t} \quad\text{and}\quad 
I_{\Phi}:=\lim_{t\to\infty} \frac{\log h_{\Phi}(t)}{\log t}, 
\end{equation}
where $h_{\Phi}$ is defined as 
\begin{equation*}
h_{\Phi}(t) := \sup_{s>0} \frac{\Phi(st)}{\Phi(s)}, \qquad t>0.
\end{equation*} 
From the definitions, one can show that 
\begin{equation*}
1\leq i_{\Phi} \leq I_{\Phi} \leq \infty,\quad 
(I_{\Phi})'=i_{\bar{\Phi}},\quad\text{and}\quad (i_{\Phi})'=I_{\bar{\Phi}}.
\end{equation*}
Furthermore, it turns out that $\Phi \in \Delta_2$ if and only if $I_{\Phi}<\infty$, and hence 
\begin{equation}\label{eq:Delta2equiv}
\Phi, \bar{\Phi} \in \Delta_2 \quad\text{ if and only if }\quad
1<i_{\Phi}\leq I_{\Phi}<\infty.
\end{equation}
We conclude by giving some examples of the lower and upper dilation indices. 
\begin{itemize}
\item Let $\Phi(t)=t^p$, $1<p<\infty$. Then $\bar{\Phi}(t)=p t^{p'}/p'$ and $i_{\Phi}=I_{\Phi}=p$. 
\vspace*{0.4em}
\item Let $\Phi(t) \simeq t^p \log(e+t)^{\alpha}$ with $1<p<\infty$ and $\alpha \in \R$. Then $\bar{\Phi}(t) \simeq t^{p'}\log(e+
t)^{\alpha(1-p')}$ and $i_{\Phi}=I_{\Phi}=p$. 
\vspace*{0.4em}
\item Given $1<p<\infty$, let $\Phi(t) \simeq t^p$, $0 \leq t \leq 1$, and $\Phi(t) \simeq e^t$, $t\geq 1$. Then $\bar{\Phi}(t) \simeq t^{p'}$, $0 \leq t \leq 1$, and $\bar{\Phi}(t) \simeq t\log(e+t)$, $t \geq 1$. In this case, $i_{\Phi}=p$ and $I_{\Phi} =\infty$, and $\Phi \not\in \Delta_2$. 
\end{itemize}

Given a Young function $\Phi$, we define the Orlicz space $L^{\Phi}(\Omega, \mu)$ to be the function space with Luxemburg norm
\begin{align}\label{eq:Orlicz}
\|f\|_{L^{\Phi}(\Omega, \mu)} := \inf\bigg\{\lambda>0: 
\int_{\Omega} \Phi \Big(\frac{|f(x)|}{\lambda}\Big) d\mu(x) \leq 1 \bigg\}. 
\end{align}
Now we define the Orlicz maximal operator 
\begin{align*}
M_{\Phi}f(x) := \sup_{Q \ni x} \|f\|_{\Phi, Q} := \sup_{Q \ni x} \|f\|_{L^{\Phi}(Q, \frac{dx}{|Q|})}, 
\end{align*}
where the supremum is taken over all cubes $Q$ in $\Rn$. When $\Phi(t)=t^p$, $1\leq p<\infty$, 
\begin{align*}
\|f\|_{\Phi, Q} = \bigg(\fint_{Q} |f(x)|^p dx \bigg)^{\frac1p}=:\|f\|_{p, Q}.  
\end{align*}
In this case, if $p=1$, $M_{\Phi}$ agrees with the classical maximal operator $M$ in \eqref{eq:HL}; if $p>1$, $M_{\Phi}f=M_pf:=M(|f|^p)^{1/p}$. If $\Phi(t) \preceq \Psi(t)$, then $M_{\Phi}f(x) \leq c M_{\Psi}f(x)$ for all $x \in \Rn$. 

The classical H\"{o}lder's inequality can be generalized to Orlicz spaces \cite{ON}. 
\begin{lemma}
Given a Young function $A$, then for all cubes $Q$, 
\begin{equation}\label{eq:Holder-AA}
\fint_{Q} |fg| dx \leq 2 \|f\|_{A, Q} \|g\|_{\bar{A}, Q}. 
\end{equation}
More generally, if $A$, $B$ and $C$ are Young functions such that $A^{-1}(t) B^{-1}(t) \leq c_1 C^{-1}(t), $ for all $t \geq t_0>0$, 
then 
\begin{align}\label{eq:Holder-ABC}
\|fg\|_{C, Q} \leq c_2 \|f\|_{A, Q} \|g\|_{B, Q}. 
\end{align}
\end{lemma}

 The following result is an extension of the well-known Coifman-Rochberg theorem \eqref{eq:CR}. The proof can be found in \cite[Lemma~4.2]{HP}. 
\begin{lemma}
Let $\Phi$ be a Young function and $w$ be a nonnegative function such that $M_{\Phi}w(x)<\infty$ a.e.. Then 
 \begin{align}
 \label{eq:CR-Phi} [(M_{\Phi}w)^{\delta}]_{A_1} &\le c_{n,\delta}, \quad\forall \delta \in (0, 1), 
 \\
\label{eq:MPhiRH} [(M_{\Phi} w)^{-\lambda}]_{RH_{\infty}} &\le c_{n,\lambda},\quad\forall \lambda>0. 
 \end{align}
 \end{lemma}

Given $p \in (1, \infty)$, a Young function $\Phi$ is said to satisfy the {\tt $B_p$ condition} (or, $\Phi \in B_p$) if for some $c>0$, 
\begin{align}\label{def:Bp}
\int_{c}^{\infty} \frac{\Phi(t)}{t^p} \frac{dt}{t} < \infty.  
\end{align}
Observe that if \eqref{def:Bp} is finite for some $c>0$, then it is finite for every $c>0$. Let $[\Phi]_{B_p}$ denote the value if $c=1$ in \eqref{def:Bp}. It was shown in \cite[Proposition~5.10]{CMP11} that if $\Phi$ and $\bar{\Phi}$ are doubling Young functions, then $\Phi \in B_p$ if and only if 
\begin{align*}
\int_{c}^{\infty} \bigg(\frac{t^{p'}}{\bar{\Phi}(t)}\bigg)^{p-1} \frac{dt}{t} < \infty.  
\end{align*}

Let us present two types of $B_p$ bump conditions. An important special case is the ``log-bumps" of the form 
\begin{align}\label{eq:log}
A(t) =t^p \log(e+t)^{p-1+\delta}, \quad  B(t) =t^{p'} \log(e+t)^{p'-1+\delta},\quad \delta>0. 
\end{align}
Another interesting example is the ``loglog-bumps" as follows: 
\begin{align}
\label{eq:loglog-1} &A(t)=t^p \log(e+t)^{p-1} \log\log(e^e+t)^{p-1+\delta}, \quad \delta>0\\  
\label{eq:loglog-2} &B(t)=t^{p'} \log(e+t)^{p'-1} \log\log(e^e+t)^{p'-1+\delta}, \quad \delta>0. 
\end{align}
Then one can verify that in both cases above, $\bar{A} \in B_{p'}$ and $\bar{B} \in B_p$ for any $1<p<\infty$.

The $B_p$ condition can be also characterized by the boundedness of the Orlicz maximal operator $M_{\Phi}$. Indeed, the following result was given in \cite[Theorem~5.13]{CMP11} and \cite[eq. (25)]{HP}.  
\begin{lemma}\label{lem:MBp}
Let $1<p<\infty$. Then $M_{\Phi}$ is bounded on $L^p(\Rn)$ if and only if $\Phi \in B_p$. Moreover, $\|M_{\Phi}\|_{L^p(\Rn) \to L^p(\Rn)} \le C_{n,p} [\Phi]_{B_p}^{\frac1p}$. In particular, if the Young function $A$ is the same as the first one in \eqref{eq:log} or \eqref{eq:loglog-1}, then 
\begin{equation}\label{eq:MAnorm}
\|M_{\bar{A}}\|_{L^{p'}(\Rn) \to L^{p'}(\Rn)} \le c_n p^2 \delta^{-\frac{1}{p'}},\quad\forall \delta \in (0, 1]. 
\end{equation}
\end{lemma}

\begin{definition}\label{def:sepbum}
Given $p \in (1, \infty)$, let $A$ and $B$ be Young functions such that $\bar{A} \in B_{p'}$ and $\bar{B} \in B_p$. We say that the pair of weights $(u, v)$ satisfies the {\tt double bump condition} with respect to $A$ and $B$ if 
\begin{align}\label{eq:uvABp}
[u, v]_{A,B,p}:=\sup_{Q} \|u^{\frac1p}\|_{A,Q} \|v^{-\frac1p}\|_{B,Q} < \infty.  
\end{align}
where the supremum is taken over all cubes $Q$ in $\Rn$. Also, $(u, v)$ is said to satisfy the {\tt separated bump condition} if 
\begin{align}
\label{eq:uvAp} [u, v]_{A,p'} &:= \sup_{Q} \|u^{\frac1p}\|_{A,Q} \|v^{-\frac1p}\|_{p',Q} < \infty, 
\\
\label{eq:uvpB} [u, v]_{p,B} &:= \sup_{Q} \|u^{\frac1p}\|_{p,Q} \|v^{-\frac1p}\|_{B,Q} < \infty. 
\end{align}
\end{definition}

Note that if $A(t)=t^p$ in \eqref{eq:uvAp} or $B(t)=t^p$ in \eqref{eq:uvpB}, each of them actually is two-weight $A_p$ condition and we denote them by $[u, v]_{A_p}:=[u, v]_{p,p'}$.  Also, the separated bump condition is weaker than the double bump condition. Indeed, \eqref{eq:uvABp} implies \eqref{eq:uvAp} and \eqref{eq:uvpB}, but the reverse direction is incorrect. 
The first fact holds since $\bar{A} \in B_{p'}$ and $\bar{B} \in B_p$ respectively indicate $A$ is a $p$-Young function and $B$ is a $p'$-Young function. The second fact was shown in \cite[Section~7]{ACM} by constructing log-bumps.   

\begin{lemma}\label{lem:M-uv}
Let $1<p<\infty$, let $A$, $B$ and $\Phi$ be Young functions such that $A \in B_p$ and $A^{-1}(t)B^{-1}(t) \lesssim \Phi^{-1}(t)$ for any $t>t_0>0$. If a pair of weights $(u, v)$ satisfies $[u, v]_{p, B}<\infty$, then 
\begin{align}\label{eq:MPhi-uv}
\|M_{\Phi}f\|_{L^p(u)} \leq C [u, v]_{p, B} [A]_{B_p}^{\frac1p} \|f\|_{L^p(v)}. 
\end{align}
Moreover, \eqref{eq:MPhi-uv} holds for $\Phi(t)=t$ and $B=\bar{A}$ satisfying the same hypotheses.  In this case, $\bar{A} \in B_p$ is necessary. 
\end{lemma}

The two-weight inequality above was established in \cite[Theorem~5.14]{CMP11} and \cite[Theorem~3.1]{CP99}.  The weak type inequality for $M_{\Phi}$ was also obtained in \cite[Proposition~5.16]{CMP11} as follows.

\begin{lemma}\label{lem:Muv-weak}
Let $1<p<\infty$, let $B$ and $\Phi$ be Young functions such that $t^{\frac1p} B^{-1}(t) \lesssim \Phi^{-1}(t)$ for any $t>t_0>0$. If a pair of weights $(u, v)$ satisfies $[u, v]_{p, B}<\infty$, then 
\begin{align}\label{eq:MPuv}
\|M_{\Phi}f\|_{L^{p,\infty}(u)} \leq C \|f\|_{L^p(v)}. 
\end{align}
Moreover, \eqref{eq:MPuv} holds for $M$ if and only if $[u, v]_{A_p}<\infty$.   
\end{lemma}

\subsection{Banach function spaces}\label{subsec:BFS}
Let $\mathbb{M}$ denote the collection of all (equivalence classes of) Lebesgue measurable functions $f:\Rn \rightarrow \C$. The characteristic function of the set $E$ will be denoted by $\mathbf{1}_E$. 

\begin{definition}\label{def:BFS}
  We say that a mapping 
$\norm{\cdot}:\mathbb{M}\to[0,\infty]$ is a {\tt function norm} if it satisfies the following: 
\begin{enumerate}
\item $\norm{f}=\norm{|f|}$ and $\norm{f}=0$ if and only if $f=0$ a.e.
\vspace*{0.4em}
\item $\norm{f+g} \leq\norm{f} + \norm{g}$.
\vspace*{0.4em}
\item $\norm{\lambda f}=|\lambda|\norm{f}$ for every $\lambda \in\mathbb{R}$.
\vspace*{0.4em}
\item If $|f|\leq|g|$ a.e., then $\norm{f} \leq\norm{g}$.
\vspace*{0.4em}
\item If $\{f_j\}_{j\in\mathbb{N}} \subseteq\mathbb{M}$ is a sequence such that $|f_j|$ increases 
to $|f|$ a.e.~as $j\to\infty$, then $\norm{f_j}$ increases to $\norm{f}$ as $j\to\infty$.
\vspace*{0.4em}
\item If $E\subseteq\Rn$ is a measurable set with $|E|<\infty$, then $\norm{\mathbf{1}_E}<\infty$ and there is a constant $C_E$ such that
$\int_E |f| dx \leq C_E\norm{f}$.  
\end{enumerate}
\end{definition}

Given a function norm $\|\cdot\|$, the set 
\begin{equation}\label{eq:BFS-def}
\X=\{f\in\mathbb{M}:\norm{f}<+\infty\} 
\end{equation} 
is called a {\tt Banach function space} (BFS, for short) over $(\Rn, dx)$.  In such a scenario, we shall write $\norm{\cdot}_{\X}$ in place of $\norm{\cdot}$ in order to emphasize the connection between the function norm $\norm{\cdot}$ and its associated function space $\X$. Then $(\X, \norm{\cdot}_{\X})$ is a Banach space.   

For a Banach function space $\X$ over $(\Rn, dx)$, one can define its {\tt associate space} $\X'$ according to
\begin{equation*}
\X'=\{f\in\mathbb{M}:\norm{f}_{\X'}<+\infty\}, 
\end{equation*}   
where 
\begin{equation*}
\norm{f}_{\X'}=\sup\bigg\{\int_{\Rn}|f(x)g(x)|\,dx: \,g\in\X,\,\, \norm{g}_{\X}\leq 1\bigg\},
\end{equation*}
and with this definition $\X'$ is also a Banach function space.  
 
It follows from the definition of $\X'$ that the following generalized H\"{o}lder's inequality holds:
\begin{equation}\label{eq:Holder}
\int_{\Rn}|f(x)g(x)|\,dx \leq \norm{f}_{\X} \norm{g}_{\mathbb{X^{'}}},
\quad f\in\X \, \text{ and }\, g\in\X^{'}.
\end{equation}
It turns out that $\X=(\X')'=:\X''$ (cf. \cite[Theorem 2.7, p.\,10]{BS}). Therefore, one has 
\begin{equation}\label{eq:ass-X}
\norm{f}_{\X}=\sup\bigg\{\int_{\Rn}|f(x)g(x)|\,dx:
\,g\in\mathbb{X'},\,\,\norm{g}_{\mathbb{X'}}\leq 1\bigg\}.
\end{equation}

\begin{remark}\label{rem:sup}
It is useful to note that the supremum in \eqref{eq:ass-X} does not change if it is taken only over functions $g\in\mathbb{X'}$ with $\norm{g}_{\X'}\leq 1$ which are non-negative and positive on a set of positive measure (that is, non-negative $g\in\X'$ with $0<\norm{g}_{\X'}\leq 1$). Indeed, the fact that we can consider only non-negative functions is direct from \eqref{eq:ass-X}. If $\norm{f}_{\X}>0$ then there is $g\in\X'$ with $\norm{g}_{\X'}\leq 1$ such that $0<\norm{f}_{\X} \leq 2 \int_{\Rn} |fg|\,dx$ and this forces $g$ to be non-zero on a set of positive measure. Finally, the case $\norm{f}_{\X}=0$ is trivial.
\end{remark}

Given $p\in(0,\infty)$ and a Banach function space $\X$, we define the scale of space $\X^p$ by
\begin{equation}\label{eq:Xp-def}
\X^p:=\{f\in\mathbb{M}:|f|^p\in\X\} \quad\text{with}\quad \|f\|_{\X^p}:=\norm{|f|^p}_{\X}^{1/p}, 
\end{equation}
which indicates that $\X^p$ is also a Banach function space whenever $p\in[1,\infty)$.

For every function $f\in\mathbb{M}$, define its distribution function as 
\begin{equation}\label{eq:RIBFS-distribution}
\mu_f(\lambda)= |\{x \in \Rn : |f(x)|>\lambda\}|, \qquad \lambda>0. 
\end{equation}

A Banach function space $\X$ is called {\tt rearrangement invariant} (RI, for short) if $\norm{f}_{\X}=\norm{g}_{\X}$ for every pair of functions $f, g\in\X$ such that $\mu_f=\mu_g$. In particular, if $\X$ is a RIBFS, one can check that its associate space $\X'$ is also a RIBFS. 

For each $f\in\mathbb{M}$, the {\tt decreasing rearrangement} of $f$ with respect to the Lebesgue measure in $\Rn$ is the function $f^*$ defined by 
\begin{equation}\label{eq:rearrangement}
f^{*}(t)=\inf\{\lambda\geq 0:\mu_f(\lambda)\leq t\},\quad t \in [0, \infty).
\end{equation} 
Note that the functions $f$ and $f^{*}$ have the same distribution function. One remarkable consequence of this is the Luxemburg representation theorem: if $\X$ is a RIBFS then there exists a RIBFS $\overline{\X}$ over $[0, \infty)$ such that $f\in\X$ if and only if $f^{*}\in\overline{\X}$ and $\norm{f^{*}}_{\overline{\X}}=\|f\|_{\X}$ (cf.  \cite[Theorem 4.10, p.\,62]{BS}). Using this representation we can define Boyd indices of a RIBFS $\X$. Given $f \in \overline{\X}$, consider the dilatation operator $D_t$, $0<t<\infty$, by setting $D_t f(s):= f(s/t)$ for each $s\ge 0$. Writing 
\begin{equation}\label{eq:h-function}
h_{\X}(t):=\sup \big\{\norm{D_t f}_{\overline{\X}}: f \in \overline{\X} \,\, \text{with} \,\, \norm{f}_{\overline{\X}} \leq 1\big\}, \quad t>0,
\end{equation} 
the lower and upper {\tt Boyd indices} may, respectively, defined as
\begin{equation}\label{eq:Boydindices_p}
p_{\X}:=\lim_{t\to\infty} 
\frac{\log t}{\log h_{\X}(t)} = \sup_{1<t<\infty} \frac{\log t}{\log h_{\X}(t)}
\end{equation}
and
\begin{equation}\label{eq:Boyindices_q}
q_{\X} := \lim_{t\to 0^{+}}\frac{\log t}{\log h_{\X}(t)} = \inf_{0<t<1} \frac{\log t}{\log h_{\X}(t)}.
\end{equation}
By design, 
\begin{align}\label{eq:pq}
1\leq p_{\X} \leq q_{\X} \leq \infty, \quad (p_{\X})'=q_{\X'} \quad\text{and}\quad (q_{\X})'=p_{\X'}.
\end{align}

Given a RIBFS $\X$ over $(\Rn, dx)$, we wish to introduce a weighted version $\X(w)$ of $\X$ via an analogous definition in which the underlying measure in $\Rn$ now is $d\mu:= w dx$. These spaces appeared in \cite{CGMP} as an abstract generalization
of a variety of weighted function spaces. Given $f \in \mathbb{M}$ and a weight $w$, let $w_f$ denote the distribution function of $f$ with respect to the measure $w dx$: 
\begin{equation*}
w_f(\lambda):= w(\{x \in \Rn: |f(x)|> \lambda\}), \quad\lambda\ge 0.
\end{equation*}

We also let $f^*_{w}$ denote the {\tt decreasing rearrangement} of $f$ with respect to the measure $w dx$,
\begin{equation*}
f^*_{w}(t):=\inf\{\lambda \ge 0 : w_f(t)\leq \lambda\}, \quad t>0.
\end{equation*}
Established these, define the weighted space $\X(w)$ by, 
\begin{equation*}\label{eq: Xwspaces}
\X(w) :=\{ f \in \mathbb{M}:  \norm{f^*_{w}}_{\overline{\X}} <\infty\}
\end{equation*} 
and the norm associated to it $\|f\|_{\X(w)}:= \|f^*_{w}\|_{\overline{\X}}$. By construction $\X(w)$ is a Banach function space over $(\Rn, w dx)$. By doing the same procedure with the associate spaces we can see that the associate space $\X(w)'$ coincides with the weighted space $\X'(w)$. 

On the other hand, one can define the general weighted Banach function spaces directly. Given a BFS $\X$ over $(\Rn, dx)$ and a weight $v$ on $\Rn$, we define $\X_v$ as the Banach function space over $(\Rn, v\,dx)$ instead of $(\Rn, dx)$ in Definition \ref{def:BFS}. Then, one has $\X_v=\X(v)$ whenever $\X$ is a RIBFS and $v$ is a weight. Thus, in what follows, given a RIBFS $\X$ over $(\Rn, dx)$, we always think of $\X_v$ as $\X(v)$ for every weight $v$ on $\Rn$.

The following result is a consequence of Boyd's interpolation theorem (cf. \cite[Theorem 5.16, p.\,153]{BS} or \cite[Theorem 8.44]{MMMMM}) in the space $(\Rn, w\,dx)$. 
\begin{lemma}\label{lem:inter}
Let $1<p<q<\infty$ and $w$ be a weight on $\Rn$. Suppose that $\X$ is a RIBFS over $(\Rn, dx)$ with $p<p_{\X} \le q_{\X}<q$. If a sublinear operator $T$ is bounded on $L^p(w)$ and $L^q(w)$, then $T$ is bounded on $\X_w$. 
\end{lemma}

Let us see the boundedness of the Hardy-Littlewood maximal operator on these weighted RI space. Let $M_w$ be the weighted  maximal operator defined by 
\begin{align}\label{eq:Mw-def}
M_w f(x) := \sup_{Q \ni x} \frac{1}{w(Q)} \int_{Q} |f(y)| w(y) dy, 
\end{align}
where the supremum is taken over all cubes $Q$ containing $x$. Moreover, we define $M_w^c$ as the weighted centered maximal operator if the supremum in \eqref{eq:Mw-def} is taken over all cubes $Q$ centered at $x$. 

\begin{lemma}\label{lem:MwcXw} 
Let $\X$ be a RIBFS over $(\Rn, dx)$ with $p_{\X}>1$. Then there exists a constant $C$ such that the following hold: 
\begin{list}{\textup{(\theenumi)}}{\usecounter{enumi}\leftmargin=1cm \labelwidth=1cm \itemsep=0.2cm 
			\topsep=.2cm \renewcommand{\theenumi}{\alph{enumi}}}
\item\label{list-1} $M$ is bounded on $\X$ if and only if $p_{\X}>1$. 
\item\label{list-2} $M$ is bounded on $\X(w)$ for every  $w \in A_{p_{\X}}$. 
\item\label{list-3} For every weight $w$,  $M_w^c$ is bounded on $\X(w)$ with $\|M_w^c\|_{\X(w) \to \X(w)} \leq C$. 
\item\label{list-4} For every $w \in A_{\infty}$, then $M_w$ is bounded on $\X(w)$ with 
\begin{align}\label{eq:3QQ}
\|M_w\|_{\X(w) \to \X(w)} \leq C \Big(1+\sup_{Q \subset \Rn} \frac{w(3Q)}{w(Q)} \Big). 
\end{align}
\end{list} 
\end{lemma}

\begin{proof}
Note that \eqref{list-1} and \eqref{list-2} were respectively obtained in \cite{Mon} and  \cite[Lemma~4.12]{CMP11}. To show \eqref{list-3} and \eqref{list-4}, we will follow the strategy in \cite{CGMP} to pursue the accurate bound. It is shown in \cite[~p.509]{G1} that for every weight $w$, 
\begin{align}\label{eq:M11}
\|M_w^c\|_{L^1(w) \to L^{1,\infty}(w)} \leq c_n, 
\end{align}
where $c_n$ is independent of $w$. Invoking \cite[Theorem~2]{AKMP}, we obtain that $M_w^c$ is of weak-type $(1, 1)$ with 
respect to the measure $w$ if and only if $(M_w^c f)_w^{*}(t) \leq C f_w^{**}(t)$, $t>0$, where $f_w^{**}(t)=\frac{1}{t}\int_{0}^t f_w^*(s) ds$. Moreover, a careful checking of the proof yields that 
\begin{equation}\label{eq:Mw-Mwf}
\|M_w^c\|_{L^1(w) \to L^{1,\infty}(w)} f_w^{**}(t)
\leq (M_w^c f)_w^{*}(t) \leq (1+\|M_w^c\|_{L^1(w) \to L^{1,\infty}(w)}) f_w^{**}(t),   
\end{equation}
for all $t>0$. Combining \eqref{eq:M11} with \eqref{eq:Mw-Mwf}, one has 
\begin{align}\label{eq:Mwfw}
\|M_w^c f\|_{\X(w)} = \|(M_w^c f)^*_w\|_{\overline{\X}} 
\leq (1+c_n) \|f_w^{**}\|_{\overline{\X}}. 
\end{align}
Also, observe that $f_w^{**}$ is the Hardy operator acting over $f_w^*$.  A conclusion from 
\cite{Mon} asserts that the Hardy operator is bounded on $\X$ if and only if $p_{\X}>1$ whenever 
$\X$ is a RIBFS. Thus, together with \eqref{eq:Mwfw}, this implies 
\begin{align*}
\|M_w^c f\|_{\X(w)} \leq C(1+c_n) \|f_w^{*}\|_{\overline{\X}} =C_n \|f\|_{\X(w)}. 
\end{align*}
This shows \eqref{list-3}. The proof of \eqref{list-4} is almost the same, but \eqref{eq:M11} is replaced by 
\begin{align*}
\|M_w\|_{L^1(w) \to L^{1,\infty}(w)} 
\leq \sup_{Q \subset \Rn} \frac{w(3Q)}{w(Q)}, \quad \forall w \in A_{\infty}, 
\end{align*}
which can be found in \cite[Exercise~2.1.1]{G1} since $w \in A_{\infty}$ implies that $wdx$ is a doubling measure.  
\end{proof}

Finally, we present several examples of Banach function spaces. 
\begin{example}[Classical Lorentz spaces]\label{ex:Lpq}
Let $0<p,q\le \infty$. The Lorentz spaces $L^{p,q}(\Rn)$ consist of all measurable functions $f$ on $\Rn$ with the quasi-norm $\|f\|_{L^{p,q}(\Rn)}<\infty$, where 
\begin{equation*}
\|f\|_{L^{p,q}(\Rn)} := 
\begin{cases}
{\displaystyle \bigg(\int_{0}^{\infty} \big(t^{\frac1p} f^*(t) \big)^q \frac{dt}{t}\bigg)^{\frac1q} }, &q<\infty, \\ 
\sup\limits_{t>0} t^{\frac1p} f^{*}(t), &q=\infty.
\end{cases}
\end{equation*}
Observe that $L^{p,p}(\Rn)$ is the Lebesgue space $L^p(\Rn)$ for any $0<p\le \infty$. 

Let $\X=L^{p,q}(\Rn)$. Then $\X$ is RIBFS and $p_{\X}=q_{\X}=p$ for all $1<p<\infty$ and $1\le q\le \infty$, or $p=q=\infty$ (see \cite[Theorem~4.6, p.\,219]{BS}). The associated space $\X'$ is given in \cite[Example~2.4.43]{CRS} by 
\begin{list}{$\bullet$}{\leftmargin=0.8cm  \itemsep=0.2cm}
\item $\X'=L^{p',q'}(\Rn)$, if $1<p,q<\infty$. 
\item $\X'=L^{\infty}(\Rn)$, if $0<q\le p=1$. 
\item $\X'=L^{p',\infty}(\Rn)$, if $0<q\le 1<p<\infty$. 
\item $\X'=\{0\}$, if $0<p<1$ and $0<q\le \infty$, or $1=p<q<\infty$. 
\end{list}
\end{example}

\begin{example}[Lorentz spaces $\Lambda^{p,q}(w)$] 
Let $w$ be a weight on $\R_{+}$, and let $0<p,q \le \infty$. The Lorentz space $\Lambda^{p,q}(w)$ is defined by 
\begin{align*}
\Lambda^{p,q}(w) :=\{f \in \mathbb{M}: \|f\|_{\Lambda^{p,q}(w)}:=\|f^*\|_{L^{p,q}(w)}<\infty\}. 
\end{align*}
In this sequel, we write $\Lambda^p(w):=\Lambda^{p,p}(w)$ for $0<p\le \infty$, and $\widehat{w}(t):=\int_{0}^t w(s) ds$ for every weight $w$ on $\R_{+}$. 

\begin{list}{$\bullet$}{\leftmargin=0.8cm  \itemsep=0.2cm}
\item If $w \equiv 1$, then $\Lambda^{p,q}(w)=L^{p,q}(\Rn)$ for $0<p,q \le \infty$. 
\item If $w(t)=t^{\frac{q}{p}-1}$, then $\Lambda^{q}(w)=L^{p,q}(\Rn)$ for $0<p,q< \infty$. 
\item If $w(t)=t^{\frac{q}{p}-1}(1+\log^{+}\frac1t)^{\alpha}$, then $\Lambda^{q}(w)=L^{p,q}(\log L)^{\alpha}$ is the Lorentz-Zygmund space, where $0<p,q< \infty$ and $\alpha \in \R$, see \cite{BR}. 
\item If $w(t)=t^{\frac{q}{p}-1}(1+\log^{+}\frac1t)^{\alpha}(1+\log^{+}\log^{+}\frac1t)^{\beta}$, then $\Lambda^{q}(w)=L^{p,q}(\log L)^{\alpha}(\log\log L)^{\beta}$ is the generalized Lorentz-Zygmund space, where $0<p,q< \infty$ and $\alpha, \beta \in \R$. In this case, $\Lambda^p(w)$ is a BFS and $\Lambda^p(w)'=L^{p'}(\log L)^{-\alpha}(\log\log L)^{-\beta}$ for $1<p<\infty$ and $\alpha, \beta \in \R$, see \cite{EOP}. 
\item If $w=\mathbf{1}_{(0,1)}$, the space $\Lambda^1(w)$ contains $L^{\infty}(\Rn)$. The functional $\|\cdot\|_{\Lambda^1(w)}$ is a norm and the space is a BFS. 
\end{list}

We collect some facts from \cite[Section~2]{CRS}. If $w$ is decreasing or $w \in B_{p,\infty}$, then $\Lambda^p(w)$ is a BFS for $1 \le p<\infty$. In addition, if $\widehat{w} \in \Delta_2$, then $\Lambda^{p,q}(w)'$ is RIBFS for $0<p<\infty$ and $0<q\le \infty$.  For the associated spaces, it was proved that $\Lambda^p(w)'=\Lambda^{p'}(w^{1-p'})$ for $1<p<\infty$ if $w$ is increasing or $w$ satisfies $\frac1t \int_{0}^t w(s) ds \le Cw(t)$ for all $t>0$. 

Next let us see the equivalence between Boyd indices and the boundedness of $M$. The the Lorentz-Shimogaki theorem \cite{Mon} says that for any $0<p<\infty$, 
\begin{align*}
M: \Lambda^p(w) \to \Lambda^p(w) \quad\Longleftrightarrow\quad I_{\widehat{w}}<p, 
\end{align*}
where $I_{\widehat{w}}$ is given in \eqref{eq:Ii}. If we set $\X=\Lambda^p(w)$ with $0<p<\infty$, then the weighted version of $\X$ is $\X(u)=\Lambda^p_u(w)$ for any weight $u$ on $\Rn$, where 
\begin{align*}
\Lambda^p_u(w) :=\{f \in \mathbb{M}: \|f\|_{\Lambda^p_u(w)}:=\|f_u^*\|_{L^p(w)}<\infty\}. 
\end{align*}
Although $\X(u)$ may not be a BFS, Lerner and P\'{e}rez \cite{LP} obtained that 
\begin{align}\label{eq:MXu}
M: \X(u) \to \X(u) \quad\Longleftrightarrow\quad p_{\X(u)}>1. 
\end{align}
In particular, if we take $w \equiv 1$, then $\X(u)=\Lambda^p_u(w)=L^p(w)$. Then by \eqref{eq:M-sharp} and \eqref{eq:MXu}, we deduce that for every $1<p<\infty$, 
\begin{align*}
M: L^p(u) \to L^p(u) \quad\Longleftrightarrow\quad u \in A_p \quad\Longleftrightarrow\quad p_{L^p(u)}>1. 
\end{align*}
\end{example}

\begin{example}[Grand Lebesgue spaces]\label{ex:gLp}
Let $\mathbb{T}=(0, 1)$ and $w$ be a weight on $\mathbb{T}$. Let $\mathbb{M}_0$ be the set of all Lebesgue measurable real valued functions on $\mathbb{T}$. Suppose that $0<\delta(\cdot) \in L^{\infty}(\mathbb{T})$ with $\|\delta\|_{L^{\infty}(\mathbb{T})} \le 1$, and $p(\cdot) \in \mathbb{M}_0$ with $p(\cdot) \ge 1$ a.e.. Denote 
\begin{align*}
L^{p[\cdot],\delta(\cdot)}(\mathbb{T}, w) :=\{f \in \mathbb{M}_0: \rho_{p[\cdot],\delta(\cdot),w}(f)<\infty\}, 
\end{align*}
where 
\begin{align*}
\rho_{p[\cdot],\delta(\cdot),w}(f)=\esssup_{x \in \mathbb{T}} 
\bigg(\int_{\mathbb{T}} |\delta(x)f(t)|^{p(x)} w(t) dt \bigg)^{\frac{1}{p(x)}}. 
\end{align*}
The collection of functions $L^{p[\cdot],\delta(\cdot)}(\mathbb{T}, w)$ is called the {\tt weighted fully measurable grand Lebesgue space}. It was proved in \cite[Proposition~2]{AF} that 
\begin{align}\label{eq:LpTw}
L^{p[\cdot],\delta(\cdot)}(\mathbb{T}, w) \text{ is a BFS for every weight w on }\mathbb{T}. 
\end{align}
Moreover, 
\begin{align}\label{eq:MTTw}
M_{\mathbb{T}}:L^{p[\cdot],\delta(\cdot)}(\mathbb{T}, w) \to L^{p[\cdot],\delta(\cdot)}(\mathbb{T}, w) 
\text{ if and only if } w \in A_{p_{+}}(\mathbb{T}), 
\end{align}
where $M_{\mathbb{T}}$ is the Hardy-Littlewood maximal operator restricted on $\mathbb{T}$. 

If $p(\cdot) \equiv p$ with $1 \le p<\infty$ and $\delta(\cdot) \equiv 1$, the space $L^{p[\cdot],\delta(\cdot)}(\mathbb{T}, w)$ coincides with the weighted Lebesgue space $L^p(\mathbb{T}, w)$. Let $L^{p),\delta(\cdot)}(\mathbb{T})$ denote $L^{p[\cdot],\delta(\cdot)}(\mathbb{T}, w)$ if $p(x)=p-x$, $\delta(x)=\eta(x)^{\frac{1}{p-x}}$ is increasing and $w \equiv 1$. Let $L^{p)}(\mathbb{T}, w)$ denote $L^{p[\cdot],\delta(\cdot)}(\mathbb{T}, w)$ if $p(x)=p-x$, $\delta(x)=x^{\frac{1}{p-x}}$. The results \eqref{eq:LpTw} and \eqref{eq:MTTw} contains the particular cases in \cite{FGJ} and \cite{FG}. Beyond that, Formica and Giova \cite{FG} obtained the Boyd indices:  
\begin{align*}
\X:=L^{p),\delta(\cdot)}(\mathbb{T}) \text{ is a RIBFS}\quad\text{and}\quad p_{\X}=q_{\X}=p.  
\end{align*} 
\end{example}

\begin{example}[Musielak-Orlicz space $L^{\varphi(\cdot)}(\Rn)$] 
A convex, left-continuous function $\phi: [0, \infty) \to [0, \infty]$ with $\lim\limits_{t \to 0^+} \phi(t)=\phi(0)=0$, and $\lim\limits_{t \to \infty}\phi(t)=\infty$ is called a $\Phi$-function. Let $\Xi(\Rn)$ be the collection of functions $\varphi: \Rn \times [0, \infty) \to [0, \infty]$ such that $\varphi(y, \cdot)$ is a $\Phi$-function for every $y \in \Rn$, and $\varphi(\cdot, t) \in \mathbb{M}$ for every $t \ge 0$. 

Given a function $\varphi \in \Xi(\Rn)$, the {\tt Musielak-Orlicz space} is defined as the set 
\[
L^{\varphi(\cdot)}(\Rn) := \big\{f \in \mathbb{M}: \lim_{\lambda \to 0} \rho_{\varphi(\cdot)}(\lambda f)=0 \big\}
\]
equipped with the (Luxemburg) norm 
\begin{align*}
\|f\|_{L^{\varphi(\cdot)}(\Rn)} := \inf\{\lambda>0: \rho_{\varphi(\cdot)}(x/\lambda) \le 1\}, 
\text{ where }
\rho_{\varphi(\cdot)}(f) = \int_{\Rn} \varphi(x, |f(x)|) dx. 
\end{align*}
Then $L^{\varphi(\cdot)}(\Rn)$ is a Banach space. Many of the classical function spaces can also be viewed as Musielak-Orlicz spaces.  
\begin{list}{$\bullet$}{\leftmargin=0.6cm  \itemsep=0.2cm}
\item If $\varphi(x,t)=t^p$ with $1<p<\infty$, we get the classical Lebesgue space $L^p(\Rn)$. 
\item If $\varphi(x,t)=\infty \mathbf{1}_{(1,\infty)}(t)$, then we get the Lebesgue space $L^{\infty}(\Rn)$. 
\item If $\varphi(x,t)=t^p w(x)$ with $1<p<\infty$ and a weight $w$, then we get the weighted Lebesgue space $L^p(w)$. 
\item If $\varphi(x,t)=t^{p(x)}$, then we get the variable Lebesgue space $L^{p(\cdot)}(\Rn)$ as in \cite[Example~2.10]{CMM}. 
\item If $\varphi(x,t)=\psi(t)$, then we get the Orlicz spaces $L^{\psi}(\Rn)$ defined in \eqref{eq:Orlicz}.  
\item If $\varphi(x,t)=t^p(1+\log^{+}t)^{\alpha}$ with $0<p<\infty$ and $\alpha \in \R$, then we get the Zygmund space $L^p(\log L)^{\alpha}$ see \cite{BR}.  If we write $\X=L^p(\log L)^{\alpha}$, then $p_{\X}=q_{\X}=p$ and $\X^r=L^{pr}(\log  L)^{\alpha}$ for any $0<r<\infty$. Note that $\X^r$ is a Banach space provided $r$ is large enough (say,  $r>1/p$).
\item If $\varphi(x,t)=t^{p(x)}(1+\log^{+}t)^{q(x)}$, we get the generalization of the Zygmund spaces $L^{p(\cdot)}(\log L)^{q(\cdot)}$ see \cite{CF09}. 
\end{list}

Let us see when $L^{\varphi(\cdot)}(\Rn)$ is a BFS. A function $\varphi \in \Xi(\Rn)$ is called proper if the set of simple functions $L_0(\Rn)$ satisfies $L_0(\Rn) \subset L^{\varphi(\cdot)}(\Rn) \cap L^{\varphi(\cdot)}(\Rn)'$. Diening et al. in \cite[Section~2]{DHHR} proved that 
\begin{align*}
L^{\varphi(\cdot)}(\Rn)\text{ is a BFS if and only if } \varphi \text{ is proper}. 
\end{align*}
Additionally, if $\varphi \in \Xi(\Rn)$ is proper, then 
\begin{align*}
L^{\varphi(\cdot)}(\Rn)' = L^{\varphi^{*}(\cdot)}(\Rn), \quad 
L^{\varphi^{*}(\cdot)}(\Rn)'=L^{\varphi(\cdot)}(\Rn)\quad\text{and}\quad 
L^{\varphi(\cdot)}(\Rn)''=L^{\varphi(\cdot)}(\Rn). 
\end{align*} 
For the boundedness of the maximal operator $M$, H\"{a}st\"{o} \cite{Ha} showed that 
\begin{align*}
M: L^{\varphi(\cdot)}(\Rn) \to L^{\varphi(\cdot)}(\Rn), 
\end{align*} 
for all $\varphi \in \Xi(\Rn)$ with some extra assumptions.  
\end{example}

\section{Extrapolation on Banach function spaces}\label{sec:BFS}

This section is devoted to showing extrapolation theorems on Banach function spaces advertised in Section \ref{sec:intro}. Let us first recall the $A_p$ and $A_{\infty}$ extrapolation from Theorem~3.9 and Corollary~3.15 in \cite{CMP11}. 

\noindent{\bf Theorem A.} 
If for some $p_0 \in[1, \infty)$ and for every $w_0 \in A_{p_0}$, 
\begin{equation*}
\|f\|_{L^{p_0}(w_0)} \leq C \|g\|_{L^{p_0}(w_0)}, \quad (f,g) \in \F, 
\end{equation*}
then for every $p \in(1, \infty)$ and for every $w \in A_p$, 
\begin{equation*}
\|f\|_{L^p(w)} \leq C \|g\|_{L^p(w)}, \quad (f,g) \in \F.
\end{equation*}

\noindent{\bf Theorem B.}
If for some $p_0 \in(0, \infty)$ and for every $w_0 \in A_{\infty}$, 
\begin{equation*}
\|f\|_{L^{p_0}(w_0)} \leq C \|g\|_{L^{p_0}(w_0)}, \quad (f,g) \in \F, 
\end{equation*}
then for every $p \in(0, \infty)$ and for every $w \in A_{\infty}$, 
\begin{equation*}
\|f\|_{L^p(w)} \leq C \|g\|_{L^p(w)}, \quad(f, g) \in \F.
\end{equation*}

Given a weight $v$, we denote by $M'_v$ the dual of the Hardy-Littlewood maximal operator $M$. That is, $M'_vf(x)=M(fv)(x)/v(x)$ if $v(x) \neq 0$, $M'_vf(x)=0$ otherwise. Similarly, one can define $M'_{\Phi, v}$ as the dual of the Orlicz maximal operator $M_{\Phi}$.

\begin{proof}[\bf Proof of Theorem \ref{thm:Aq}]
Fix $u^{1-p'_0} \in A_1$ and $v \in A_1$. We claim that for every pair $(f, g) \in \F$ with $\|fu\|_{\X_v}<\infty$ and $\|gu\|_{\X_v}<\infty$, there exists $w=w(f, g) \in A_{p_0}$ with $[w]_{A_{p_0}} \le 2 \|M_v\|_{\X'_v \to \X'_v} [u^{1-{p_0}'}]_{A_1}^{p_0-1} [v]_{A_1}$ such that 
\begin{align}\label{eq:Aq-3}
\|fu\|_{\X_v} \le 2 \|f\|_{L^1(w)} \quad\text{and}\quad \|g\|_{L^1(w)} \le 2\|gu\|_{\X_v}.  
\end{align}
Assuming that our claim holds momentarily, let us see how \eqref{eq:Aq-2} follows from \eqref{eq:Aq-3}. Given $(f, g) \in \F$, we may assume that $\|gu\|_{\X_v}<\infty$, otherwise, there is nothing to prove. We would like to observe that $f<\infty$ a.e.. Otherwise, there exists a measurable set $E \subset \Rn$ with $|E|>0$ such that $f=\infty$ on $E$. In view of \eqref{eq:Aq-1}, this gives that 
\begin{align}\label{eq:ginf}
\|g\|_{L^1(w)} = \infty \quad\text{for every } w \in A_{p_0}. 
\end{align}
On the other hand, applying \eqref{eq:Aq-3} to $f=g$, we find a weight $w_0=w_0(g) \in A_{p_0}$ such that $\|g\|_{L^1(w_0)} \le 2\|gu\|_{\X_v}<\infty$. This contradicts \eqref{eq:ginf}. 

For each $N \ge 1$, we define $f_N:=f\mathbf{1}_{\{B(0, N): f(x) \le N, u(x) \le N\}}$. The fact $v \in A_1$ implies $v(B(0, N))<\infty$. Thanks to the property (6) in Definition \ref{def:BFS}, this gives that $\|f_N u\|_{\X_v} \le N^2 \|\mathbf{1}_{B(0, N)}\|_{\X_v} < \infty$.  Then applying \eqref{eq:Aq-3}, we get a weight $w=w(f_N, g) \in A_{p_0}$ with $[w]_{A_{p_0}} \le 2 \|M_v\|_{\X'_v \to \X'_v} [u^{1-{p_0}'}]_{A_1}^{p_0-1} [v]_{A_1}$ such that 
\begin{align}\label{eq:Aq-4}
\|f_N u\|_{\X_v} \le 2 \|f_N\|_{L^1(w)} \quad\text{and}\quad \|g\|_{L^1(w)} \le 2\|gu\|_{\X_v}.  
\end{align}
Together with \eqref{eq:Aq-1} and \eqref{eq:Aq-4}, it yields 
\begin{align}\label{eq:Aq-5}
\|f_N u\|_{\X_v} &\le 2 \|f_N\|_{L^1(w)} \le 2 \|f\|_{L^1(w)} 
\le 2 \Psi([w]_{A_{p_0}}) \|g\|_{L^1(w)}  
\nonumber \\
&\le 4 \Psi \big(2 \|M_v\|_{\X'_v \to \X'_v} [u^{1-{p_0}'}]_{A_1}^{p_0-1} [v]_{A_1}\big) \|gu\|_{\X_v}.  
\end{align}
Recall that $f<\infty$ a.e.. Consequently, $f_N \nearrow f$ as $N \to \infty$. Therefore, \eqref{eq:Aq-5} and property (5) in Definition \ref{def:BFS} immediately give \eqref{eq:Aq-2} as desired. 

It remains to prove \eqref{eq:Aq-3}. Note that $p_{\X'}=(q_{\X})'>1$. Then Lemma \ref{lem:MwcXw} \eqref{list-4} gives that $M_v$ is bounded on $\X'_v$. Let $h \in \X'_v$ be a nonnegative function such that $h$ is non-zero on a set of positive measure. For the function $h$, we define the Rubio de Francia iteration algorithm as 
\begin{align}\label{eq:RdF}
\mathcal{R}_{v}h := \sum_{j=0}^{\infty} \frac{M^j_{v} h}{2^j \|M_v\|_{\X'_v \to \X'_v}^j}.
\end{align}
Then the following properties are fulfilled: 
\begin{align}\label{eq:RdF-Aq} 
h \le \mathcal{R}_vh,\quad \|\mathcal{R}_vh\|_{\X'_v} \leq 2 \|h\|_{\X'_v} 
\quad\text{and}\quad [\mathcal{R}_vh]_{A_1(v)} \le 2 \|M_v\|_{\X'_v \to \X'_v}. 
\end{align}
Indeed, the first two inequalities are straightforward. Since $h$ is nonnegative and non-zero on a set of positive measure, for every $x \in \Rn$ there exists a cube $Q_x$ containing $x$ such that $\int_{Q_x} hv\, dy>0$. By the fact that $v$ is a weight with $v \in L^1_{\loc}(\Rn)$, one has 
\begin{align}\label{eq:Rvpo}
\mathcal{R}_v h(x) \gtrsim M_vh(x)=\sup_{Q \ni x} \frac{1}{v(Q)} \int_{Q} hv\, dy 
\ge \frac{1}{v(Q_x)} \int_{Q_x} hv\, dy >0. 
\end{align}
Let $r>0$. Then $v(B(0, r))<\infty$, which together with Item (6) in Definition \ref{def:BFS} implies $\|\mathbf{1}_{B(0, r)}\|_{\X_v}<\infty$.  Hence, 
\begin{align*}
\int_{B(0, r)} \mathcal{R}_vh \, v\, dx 
\le \|\mathbf{1}_{B(0, r)}\|_{\X_v} \|\mathcal{R}_v h\|_{\X'_v} 
\le 2 \|\mathbf{1}_{B(0, r)}\|_{\X_v} \|h\|_{\X'_v} < \infty. 
\end{align*}
Thus, $\mathcal{R}_vh \, v<\infty$ a.e. in $B(0, r)$. By the arbitrariness of $r$, $0<v<\infty$ a.e. and \eqref{eq:Rvpo}, this in turn gives that $0<\mathcal{R}_vh <\infty$ a.e., that is, $\mathcal{R}_v h$ is a weight. Moreover, the last one in \eqref{eq:RdF-Aq} follows at once from the definition of $\mathcal{R}_v$. 

To proceed, fix $(f, g) \in \F$ with $\|fu\|_{\X_v}<\infty$. By \eqref{eq:ass-X} and Remark \ref{rem:sup}, there exists a nonnegative function $h \in \X'_v$ with $\|h\|_{\X'_v} \le 1$ such that $h$ is non-zero on a set of positive measure and 
\begin{align}\label{eq:fuh}
\|fu\|_{\X_v} \le 2 \int_{\Rn} f(x) u(x) h(x) v(x) dx.
\end{align}
Pick $w:=u (\mathcal{R}_vh)v = (u^{1-{p_0}'})^{1-p_0} (\mathcal{R}_vh) v$. Since $u^{1-{p_0}'} \in A_1$, $v \in A_1$ and $\mathcal{R}_vh \in A_1(v)$ , by \eqref{eq:rev-fac}, \eqref{eq:RdF-Aq} and Lemma \ref{lem:A1Ap} \eqref{eq:Apu} we have 
\begin{align*}
[w]_{A_{p_0}} \le [u^{1-{p_0}'}]_{A_1}^{p_0-1} [(\mathcal{R}_vh)v]_{A_1} 
&\le [u^{1-{p_0}'}]_{A_1}^{p_0-1} [\mathcal{R}_vh]_{A_1(v)} [v]_{A_1} 
\nonumber \\
&\le 2 \|M_v\|_{\X'_v \to \X'_v} [u^{1-{p_0}'}]_{A_1}^{p_0-1} [v]_{A_1}. 
\end{align*}
In addition, the first one in \eqref{eq:RdF-Aq} and \eqref{eq:fuh} imply 
\begin{align*}
\|fu\|_{\X_v} \le 2 \int_{\Rn} f(x) u(x) \mathcal{R}_vh(x) v(x) dx = 2\|f\|_{L^1(w)}. 
\end{align*}
On the other hand, combining H\"{o}lder's inequality \eqref{eq:Holder} and \eqref{eq:RdF-Aq}, we obtain
\begin{align*}
\|g\|_{L^1(w)} = \int_{\Rn} gu \mathcal{R}_v h\, v\, dx 
\le \|gu\|_{\X_v} \|\mathcal{R}_vh\|_{\X'_v} 
\le 2 \|gu\|_{\X_v} \|h\|_{\X'_v} 
\le 2 \|gu\|_{\X_v}.
\end{align*}
This shows \eqref{eq:Aq-3} and completes the proof. 
\end{proof}


\begin{proof}[\bf Proof of Corollary \ref{cor:Ap}]
Fix $p\in (1,\infty)$. We first observe that the hypothesis \eqref{eq:Ap-1} and Theorem {\bf A} give that for every $w \in A_p$, 
\begin{equation*}
\|f^p\|_{L^1(w)}=\|f\|_{L^p(w)}^p \leq C \|g\|_{L^p(w)}^p=\|g^p\|_{L^1(w)}, \quad (f, g) \in \F. 
\end{equation*}
This says that the hypothesis \eqref{eq:Aq-1} is satisfied for the exponent $p_0=p$ and the pair $(f^p, g^p)$. Accordingly, \eqref{eq:Aq-2} immediately implies \eqref{eq:Ap-2} as desired. 
\end{proof} 

\begin{proof}[\bf Proof of Theorem \ref{thm:BFSA1}]
As we did in the proof of Theorem \ref{thm:Aq}, it only needs to show that if \eqref{eq:BFSA1-1} holds, then for every pair $(f, g)$ with $\|fu\|_{\X_v^{p_0}}<\infty$ and $\|gu\|_{\X_v^{p_0}}<\infty$, there exists a weight $w=w(f, g) \in A_1$ such that $[w]_{A_1} \le 2 K_0 [v]_{A_1}$, 
\begin{align}\label{eq:BFSA1-4}
\|fu\|_{\X_v^{p_0}}  \le 2^{1/p_0} \|f\|_{L^{p_0}(w)} \quad\text{and}\quad 
\|g\|_{L^{p_0}(w)} \le 2^{1/p_0} \|gu\|_{\X_v^{p_0}}. 
\end{align} 

To this end, fix a pair of functions $(f, g) \in \F$ with $\|fu\|_{\X_v^{p_0}}<\infty$ and $\|gu\|_{\X_v^{p_0}}<\infty$. By \eqref{eq:ass-X}, there exists a non-negative function $h \in \X'_v$ with $\|h\|_{\X'_v} \le 1$ such that $h$ is non-zero on a set of positive measure, and 
\begin{align}\label{eq:fpup}
\|fu\|_{\X_v^{p_0}}^{p_0} = \|f^{p_0}u^{p_0}\|_{\X_v} \le 2 \int_{\Rn} f^{p_0}u^{p_0} hv\, dx.
\end{align}
For this function $h$, we define 
\begin{equation*}
\mathcal{R} h :=\sum_{j=0}^{\infty}\frac{M_v^j h}{2^j K_0^j} \quad\text{and}\quad 
H := \mathcal{R}(hu^{p_0}) u^{-p_0}.
\end{equation*} 
Then using \eqref{eq:BFSA1-1}, one can verify that 
\begin{align}\label{eq:RdFA1}
h \leq H, \quad  \|H\|_{\X'_v} \leq 2\|h\|_{\X'_v} \quad\text{and}\quad
[\mathcal{R}h]_{A_1(v)} \leq 2K_0.
\end{align}

If we set $w:=u^{p_0}Hv$, then $v \in A_1$ and Lemma \ref{lem:A1Ap} \eqref{eq:vA1u} imply $[w]_{A_1}=[\mathcal{R}(hu^{p_0}) v]_{A_1} \le [\mathcal{R}(hu^{p_0})]_{A_1(v)} [v]_{A_1} \le 2K_0 [v]_{A_1}$. It follows from \eqref{eq:fpup} and \eqref{eq:RdFA1} that 
\begin{align*}
\|fu\|_{\X_v^{p_0}}^{p_0} \le 2 \int_{\Rn} f^{p_0}u^{p_0} Hv\, dx = 2\|f\|_{L^{p_0}(w)}^{p_0}. 
\end{align*}
Also, invoking \eqref{eq:Holder} and \eqref{eq:RdFA1}, we conclude that 
\begin{align*}
\|g\|_{L^{p_0}(w)}^{p_0} &= \int_{\Rn} g^{p_0} u^{p_0} Hv\,dx 
\le \|g^{p_0} u^{p_0}\|_{\X_v} \|H\|_{\X'_v} \le 2 \|gu\|_{\X_v^{p_0}}^{p_0}
\end{align*}
This shows \eqref{eq:BFSA1-4}. 
\end{proof}

\begin{proof}[\bf Proof of Theorem \ref{thm:BFSAi}]
By Theorem \textbf{B}, we have that for every $p \in (0,\infty)$ and for every $w \in A_{\infty}$, 
\begin{align}\label{eq:BFSAi-2p}
\|f\|_{L^p(w)} \leq C \|g\|_{L^p(w)}, \quad (f,g) \in \F.  
\end{align} 
This can be rewritten as 
\begin{align}\label{eq:BFSAi-p2}
\|f^p\|_{L^1(w)} \le C \|g^p\|_{L^1(w)}, \quad (f,g) \in \F.  
\end{align} 
As before, it is enough to prove that if \eqref{eq:BFSAi-1} holds, then for every pair $(f, g) \in \F$ with $\|fu\|_{\X_v}<\infty$ and $\|gu\|_{\X_v}<\infty$, there exists a weight $w=w(f, g) \in A_{\infty}$ such that $[w]_{A_{\infty}} \le 2 K_0[v]_{A_{\infty}}$, 
\begin{align}\label{eq:BFSAi-4}
\|f u\|_{\X_v}  \le 2 \|f\|_{L^1(w)} \quad\text{and}\quad \|g\|_{L^1(w)} \le 2 \|gu\|_{\X_v}. 
\end{align}
We here point out that \eqref{eq:BFSAi-3} will follow from \eqref{eq:BFSAi-p2} and \eqref{eq:BFSAi-4}. 

Let us turn to the proof of \eqref{eq:BFSAi-4} and it is similar to that of \eqref{eq:BFSA1-4}. In the current setting, the Rubio de Francia iteration algorithm is replaced by 
\begin{equation*}
\mathcal{R} h :=\sum_{j=0}^{\infty}\frac{M_v^j h}{2^j K_0^j} \quad\text{and}\quad 
H := \mathcal{R}(hu) u^{-1}.
\end{equation*} 
Then it follows from \eqref{eq:BFSAi-1} that
\begin{align}\label{eq:BFSAi-5}
h \leq H,\quad  \|H\|_{\X'_v} \leq 2 \|h\|_{\X'_v} \quad\text{and}\quad
[\mathcal{R}h]_{A_1(v)} \leq 2K_0.
\end{align}
Pick $w:=uHv$. Then by the last estimate in \eqref{eq:BFSAi-5}, $v \in A_{\infty}$ and Lemma \ref{lem:A1Ap} \eqref{eq:vA1u}, we obtain that $w=\mathcal{R}(hu)v \in A_{\infty}$ with $[w]_{A_{\infty}} \le [\mathcal{R}(hu)]_{A_1(v)} [v]_{A_{\infty}} \le 2 K_0 [v]_{A_{\infty}}$. The rest of argument is almost the same as in Theorem \ref{thm:BFSA1}. 
 
Additionally, \eqref{eq:BFSAi-vec} is a consequence of \eqref{eq:BFSAi-3} and \eqref{eq:BFSAi-2p}. 
\end{proof}

\begin{proof}[\bf Proof of Theorem \ref{thm:AiAi}]
By Theorem ${\bf B}$, the hypothesis \eqref{eq:AiAi-1} implies that for every $p \in(0, \infty)$ and for every $w \in A_{\infty}$, 
\begin{equation}\label{eq:AiAi-3}
\|f\|_{L^p(w)} \leq C \|g\|_{L^p(w)}, \quad(f, g) \in \F.
\end{equation}
It suffices to show that for every $r \in (1, \infty)$, for every pair $(f, g) \in \F$ with $\|fu\|_{\X_v}<\infty$ and $\|gu\|_{\X_v}<\infty$, there exists a weight $w=w(f, g) \in A_{\infty}$ such that  
\begin{align}\label{eq:AiAi-4}
\|f u\|_{\X_v}  \le 2^r \|f\|_{L^{\frac1r}(w)} \quad\text{and}\quad \|g\|_{L^{\frac1r}(w)} \le 2^r \|gu\|_{\X_v}. 
\end{align}

Let $u \in RH_{\infty}$ and $v \in A_{\infty}$. Denote $\Y=\X^r$. Then $p_{\Y'}=(q_{\Y})'=(r q_{\X})'>1$ provided $q_{\X}<\infty$. Then, it follows from Lemma \ref{lem:MwcXw} \eqref{list-4} that $M_v$ is bounded on $\Y'_v$. For any non-negative function $h \in\Y'_v$ such that $h$ is non-zero on a set of positive measure, we define the Rubio de Francia iteration algorithm as: 
\begin{equation}
\mathcal{R}_{v} h :=\sum_{j=0}^{\infty}\frac{M_v^k h}{2^k \|M_v\|_{\Y'_v \to \Y'_v}^k}.
\end{equation}
Then it is easy to verify that 
\begin{align}\label{eq:RdF-AiAi}
h \leq\mathcal{R}_{v} h, \quad \norm{\mathcal{R}_{v}h}_{\Y'_v}\leq 2\|h\|_{\Y'_v} 
\quad\text{and}\quad [\mathcal{R}_vh]_{A_1(v)} \leq 2 \|M_v\|_{\Y'_v}. 
\end{align}

From \eqref{eq:ass-X}, there exists a nonnegative function $h \in \Y'_v$ with $\|h\|_{\Y'_v} \le 1$ such that $h$ is non-zero on a set of positive measure, and 
\begin{align}\label{eq:fu-AiAi}
\norm{fu}_{\X_v}^{\frac1r} &= \big\|(fu)^{\frac{1}{r}}\big\|_{\Y_v} 
\le 2 \int_{\Rn} f^{\frac{1}{r}} u^{\frac{1}{r}} h v \, dx. 
\end{align} 
Since $v \in A_{\infty}$ and $\mathcal{R}_v h \in A_1(v)$, Lemma \ref{lem:A1Ap} \eqref{eq:vA1u} implies $(\mathcal{R}_v h) v \in A_{\infty}$. Also, by Lemma \ref{lem:A1Ap} \eqref{eq:RHs}, one has $u^{\frac{1}{r}} \in RH_{\infty}$. Together with Lemma \ref{lem:A1Ap} \eqref{eq:AiRH}, these facts give that 
\begin{equation}\label{eq:Rvhv}
w := u^{\frac{1}{r}} (\mathcal{R}_v h) v \in A_{\infty}.
\end{equation}
Moreover, by \eqref{eq:RdF-AiAi} and \eqref{eq:fu-AiAi},  
\begin{align*}
\|fu\|_{\X_v} \le 2^r \|f\|_{L^{\frac1r}(w)},  
\end{align*}
and 
\begin{align*}
\|g\|_{L^{\frac1r}(w)}^{\frac1r}
&=\int_{\Rn} g^{\frac{1}{r}} u^{\frac{1}{r}} \mathcal{R}_{v} h\, v\, dx 
\leq \|(gu)^{\frac{1}{r}}\|_{\Y_v} \|\mathcal{R}_{v}h\|_{\Y'_v} 
\le 2 \|gu\|_{\X_v}^{\frac{1}{r}}. 
\end{align*} 
The proof is complete. 
\end{proof} 

\begin{proof}[\bf Proof of Theorem \ref{thm:AA}]
Let $u \in A_1$ and $v \in A_{\infty}$. Then Lemma \ref{lem:A1Ap} \eqref{eq:A1RH} implies $u^{-1} \in RH_{\infty}$. Thus,  \eqref{eq:AA-2} follows from Theorem \ref{thm:AiAi}. In what follows, we assume that $u \in A_{\infty}$ and $v \in A_1$.  

By Theorem ${\bf B}$, the hypothesis \eqref{eq:AiAi-1} implies that for every $p \in(0, \infty)$ and for every $w \in A_{\infty}$, 
\begin{equation}\label{eq:AA-3}
\|f\|_{L^p(w)} \leq C \|g\|_{L^p(w)}, \quad(f, g) \in \F.
\end{equation}
It suffices to show that for every pair $(f, g) \in \F$ with $\|fu\|_{\X_v}<\infty$ and $\|gu\|_{\X_v}<\infty$, there exist some $r \in (1, \infty)$ and a weight $w=w(f, g) \in A_{\infty}$ such that  
\begin{align}\label{eq:AA-4}
\|f u\|_{\X_v}  \le 2^r \|f\|_{L^{\frac1r}(w)} \quad\text{and}\quad \|g\|_{L^{\frac1r}(w)} \le 2^r \|gu\|_{\X_v}. 
\end{align}

Denote $\Y=\X^r$ for some $r \in (1, \infty)$ chosen later. For any non-negative function $h \in \Y'_v$ such that $h$ is non-zero on a set of positive measure, we define the Rubio de Francia iteration algorithm as: 
\begin{equation}
\mathcal{R}_{v} h :=\sum_{j=0}^{\infty}\frac{M_v^k h}{2^k \|M_v\|_{\Y'_v \to \Y'_v}^k}.
\end{equation}
Then we have 
\begin{align}\label{eq:RdF-AA}
h \leq\mathcal{R}_{v} h, \quad \norm{\mathcal{R}_{v}h}_{\Y'_v}\leq 2\|h\|_{\Y'_v} 
\quad\text{and}\quad [\mathcal{R}_vh]_{A_1(v)} \leq 2 \|M_v\|_{\Y'_v}
\end{align}
By \eqref{eq:ass-X}, there exists a nonnegative function $h \in \Y'_v$ with $\|h\|_{\Y'_v} \le 1$ such that $h$ is non-zero on a set of positive measure, and 
\begin{align*}
\Big\|\frac{f}{u}\Big\|_{\X_v}^{\frac1r} &= \Big\|\Big(\frac{f}{u}\Big)^{\frac1r}\Big\|_{\Y_v}
\le 2 \int_{\Rn} f^{\frac1r} u^{-\frac{1}{r}} h v  \, dx.  
\end{align*}
Recall that $u \in A_{\infty}$ and $v \in A_1$. It follows from \cite[Theorem~1.3]{HaP} that there exists a dimensional constant $c_n>0$ such that 
\begin{align}\label{eq:eecn}
[u]_{A_p} \leq e^{e^{c_n[u]_{A_{\infty}}}}, \quad\forall  p>e^{c_n[u]_{A_{\infty}}}. 
\end{align} 
Pick $p=1+e^{c_n[u]_{A_{\infty}}}$ such that $u \in A_p$. Then by the $A_p$ factorization theorem, one has $u=u_1 u_2^{1-p}$ for some $u_1, u_2 \in A_1$. We claim that there exists $C_n>0$ large enough such that 
\begin{align}\label{eq:Ruv-uvr}
w := u^{-\frac{1}{r}} \mathcal{R}_{v} h \cdot v 
=\Big[(\mathcal{R}_{v} h \cdot v)   u_2^{\frac{p-1}{r}} \Big] u_1^{-\frac{1}{r}}  \in  A_{\infty}, 
\quad\forall r>C_n [v]_{A_1}^2 e^{C_n [u]_{A_{\infty}}}. 
\end{align}
Indeed, a straightforward calculation gives that $\sigma:=\mathcal{R}_{v}h \cdot v \in A_1$ with 
\begin{align}\label{eq:w-A1-uv-A1}
[\sigma]_{A_1} & \leq [v]_{A_1} [\mathcal{R}_{v}h]_{A_1(v)} 
\leq 2[v]_{A_1} \|M_{v}\|_{\Y'_v \to \Y'_v}  \nonumber\\ 
&\leq 2C [v]_{A_1} \sup_{Q \subset \Rn} \frac{v(3Q)}{v(Q)}
\leq 2C \cdot 3^n [v]_{A_1}^2, 
\end{align}
where we used Lemma \ref{lem:A1Ap} \eqref{eq:Apu}, \eqref{eq:RdF-AA}, \eqref{eq:3QQ} and the doubling property: for any $\omega \in A_s$, $1 \leq s < \infty$, 
\begin{align*}
\omega(\lambda Q) \leq \lambda^{ns} [\omega]_{A_s} \omega(Q),\quad\forall \lambda>1 \text{ and }\forall Q \subset \Rn.
\end{align*}
It follows from \eqref{eq:RHp} below that for every $\sigma \in A_1$ and for every cube $Q \subset \Rn$, 
\begin{align*}
\bigg(\fint_{Q} \sigma(x)^{r_\sigma} dx \bigg)^{\frac{1}{r_{\sigma}}} 
\leq 2\fint_{Q} \sigma(x) dx, 
\quad\text{where } 
r_{\sigma} :=1+\frac{1}{2^{n+1}[\sigma]_{A_1}}. 
\end{align*}
Then by \eqref{eq:w-A1-uv-A1}, we have 
\begin{align*}
r'_{\sigma}(p-1)=(1+2^{n+1}[\sigma]_{A_1}) e^{c_n[u]_{A_{\infty}}}
\leq \widetilde{c}_n [v]_{A_1}^2 e^{\widetilde{c}_n[u]_{A_{\infty}}}. 
\end{align*}
Let $r>C_n [v]_{A_1}^2 e^{C_n [u]_{A_{\infty}}}$ with $C_n \geq \widetilde{c}_n$. Then $r'_{\sigma}(p-1) \leq r$ and
\begin{align*}
\fint_{Q} \sigma u_2^{\frac{p-1}{r}} dx 
&\leq \bigg(\fint_{Q} \sigma^{r_{\sigma}} dx\bigg)^{\frac{1}{r_{\sigma}}} 
\bigg(\fint_{Q} u_2^{\frac{r'_{\sigma}(p-1)}{r}} dx\bigg)^{\frac{1}{r'_{\sigma}}}
\leq 2 \bigg(\fint_{Q} \sigma \, dx\bigg) \bigg(\fint_{Q} u_2 \ dx\bigg)^{\frac{p-1}{r}}
\\
&\leq 2\Big([\sigma]_{A_1} \inf_{Q} \sigma\Big) \Big([u_2]_{A_1} \inf_{Q} u_2 \Big)^{\frac{p-1}{r}}
\leq 2 [\sigma]_{A_1} [u_2]_{A_1}^{\frac{p-1}{r}} \inf_{Q} \Big(\sigma u_2^{\frac{p-1}{r}}\Big). 
\end{align*}
That is, $\sigma u_2^{\frac{p-1}{r}} \in A_1$. 
On the other hand, $u_1 \in A_1$ implies that $u_1^{\frac1r} \in A_1$ with $[u_1^{\frac1r}]_{A_1} \leq [u_1]_{A_1}^{\frac1r}$. In view of Lemma \ref{lem:A1Ap} \eqref{eq:A1RH}, one has $u_1^{-\frac1r} \in RH_{\infty}$. Therefore, these and Lemma \ref{lem:A1Ap} \eqref{eq:uu} immediately yield \eqref{eq:Ruv-uvr}.  

The remaining argument is almost the same in the proof of Theorem \ref{thm:AiAi}. We omit the details. 
\end{proof}

\begin{proof}[\bf Proof of Theorem \ref{thm:Mvr}] 
Let $h$ be a non-negative function such that $ h \in \X'_v$ and $h$ is non-zero on a set of positive measure. In view of \eqref{eq:Mvr-1}, we define the Rubio de Francia algortihm as:
\begin{align*}
\mathcal{R}h := \sum_{j=0}^{\infty} \frac{(M'_{v^{1/r}})^k h}{2^k \|M'_{v^{1/r}}\|_{\X'_v \to \X'_v}^k}. 
\end{align*}
Then we see that $\mathcal{R}h \cdot v^{\frac1r} \in A_1$. Combining this with $u^{\frac1s} v^{-\frac{1}{sr'}} \in A_1$ and  \eqref{eq:rev-fac}, we obtain 
\begin{align}
w:= u^{-1} (\mathcal{R}h)v 
= (u^{\frac1s} v^{-\frac{1}{sr'}})^{1-(s+1)} (\mathcal{R}h \cdot v^{\frac1r}) \in A_{s+1} \subset A_{\infty}. 
\end{align}
The remaining argument is the same as we did before. The only difference is that we do not use the rescaling argument in this case. 
\end{proof}

\begin{proof}[\bf Proof of Theorem \ref{thm:Saw}]
For every $U \in A_1$ with $[U] \le 2K_0$, by Lemma \ref{lem:A1Ap} \eqref{eq:A1A1}, there exists $\varepsilon_0 := \varepsilon_0 (K_0) \in (0, 1)$ such that 
\begin{equation}\label{eq:UV-A1}
UV^{\varepsilon} \in A_1 \quad\text{for every }V \in A_1 \text{ and } \varepsilon \in (0, \varepsilon_0).
\end{equation}
Pick $r>1$ such that $0<\frac{1}{r'}<\min\{\varepsilon_0, \frac{1}{q_0}\}$. Denote $\Y_{uv}=\X^r_{uv}$. Since $r'>q_0$, the hypothesis \eqref{eq:Saw-1} implies that $M'_u$ is bounded on $\Y'_{uv}$ with bound $K_0$. 

For any non-negative function $h \in\Y'_{uv}$ with $\|h\|_{\Y'_{uv}} \leq 1$ such that $h$ is non-zero on a set of positive measure, we define the Rubio de Francia iteration algorithm as: 
\begin{equation*}
\mathcal{R} h:=\sum_{j=0}^{\infty}\frac{(M'_u)^j h}{2^j K_0^j}.
\end{equation*}
Then one can check that 
\begin{align}\label{eq:RdF-Saw}
h \leq\mathcal{R} h, \quad  
\|\mathcal{R}h\|_{\Y'_{uv}}\leq 2\|h\|_{\Y'_{uv}} \quad \text{and} \quad 
[\mathcal{R}h \cdot u]_{A_1} \leq 2 K_0.  
\end{align} 
On the other hand, since $v \in A_{\infty}$, it follows from Lemma \ref{lem:A1Ap} \eqref{eq:uu} that $v=v_1v_2$ for some  $v_1 \in A_1$ and $v_2 \in RH_{\infty}$.  Invoking \eqref{eq:UV-A1} and \eqref{eq:RdF-Saw}, we get 
\begin{equation}\label{eq:Rhuvr}
(\mathcal{R} h \cdot u) v_1^{\frac{1}{r'}} \in A_1.
\end{equation} 
Note that $v_2^{\frac{1}{r'}} \in RH_{\infty}$ by Lemma \ref{lem:A1Ap} \eqref{eq:RHs}. Accordingly, combining \eqref{eq:Rhuvr} with  Lemma \ref{lem:A1Ap} \eqref{eq:uu}, this gives that 
\begin{align}\label{eq:Rhuvv}
w:=\mathcal{R} h \cdot u v^{\frac{1}{r'}}
=[(\mathcal{R} h \cdot u) v_1^{\frac{1}{r'}}] v_2^{\frac{1}{r'}}  \in A_{\infty}. 
\end{align}
The rest of the proof is similar to the argument in the proof of Theorem \ref{thm:AA}.
\end{proof}

\begin{proof}[\bf Proof of Theorem \ref{thm:M3w}]
We first note that the hypothesis \eqref{eq:M3w-1} and Theorem $\bold{B}$ give that for every $p \in (0, \infty)$ and every  $w \in A_{\infty}$, 
\begin{align}\label{eq:fMg}
\|f\|_{L^p(w)} \le C \|Mg\|_{L^p(w)},\quad (f, g) \in \F. 
\end{align}

It suffices to show that for every (or for some) $s>1$ and for every pair $(f, g) \in \F$ with $\norm{\frac{f}{M^3w}}_{\X(Mw)}<\infty$ and $\norm{\frac{g}{Mw}}_{\X(Mw)}<\infty$, there exists a weight $w_0=w_0(f, g) \in A_{\infty}$ such that 
\begin{align}\label{eq:M3w-3}
\norm{\frac{f}{M^3w}}_{\X(Mw)} \lesssim \|f\|_{L^{\frac1s}(w_0)} \quad\text{and}\quad 
\|Mg\|_{L^{\frac1s}(w_0)} \lesssim \norm{\frac{g}{Mw}}_{\X(Mw)}.  
\end{align}
Fix $s>1$ and set $\Y=\X^s$. For each $r>1$, $(Mw)^{1-r(1-\frac{1}{2s})} = [(Mw)^{1-\frac{1}{2s}}]^{1-r} (Mw)^{\frac{1}{2s}} \in A_r$. If we denote $v:=(Mw)^{1-\frac{1}{2s}}$, then 
\begin{align}\label{eq:SfMr}
\|M'_v f\|_{L^r(Mw)} &= \big\|M(fv)\big\|_{L^r((Mw)^{1-r(1-\frac{1}{2s})})} 
\lesssim \big\|fv\big\|_{L^r((Mw)^{1-r(1-\frac{1}{2s})})} = \|f\|_{L^r(Mw)}. 
\end{align}
Note that 
\begin{align*}
p_{\Y'}=(q_{\Y})'=(sq_{\X})'>1 \quad\text{and}\quad q_{\Y'}=(p_{\Y})'=(sp_{\X})'<\infty.
\end{align*} 
Taking $r_1, r_2 \in (1, \infty)$ such that $1<r_1<p_{\Y'} \le q_{\Y'}<r_2<\infty$, we obtain from \eqref{eq:SfMr} that $M'_v$ is bounded on $L^{r_1}(Mw)$ and $L^{r_2}(Mw)$. Hence, Lemma \ref{lem:inter} gives that  
\begin{align*}
M'_v: \Y'(Mw) \to \Y'(Mw) \quad\text{boundedly}. 
\end{align*}
Let $h \in \Y'(Mw)$ be a nonnegative function such that $h$ is non-zero on a set of positive measure. Now we define the Rubio de Francia iteration algorithm as 
\begin{align*}
\mathcal{R}h := \sum_{j=0}^{\infty} \frac{(M'_v)^k h}{2^k \|M'_v\|_{\Y'(Mw)}}. 
\end{align*}
As before, it is easy to verify that 
\begin{align}\label{eq:RdF-YMw}
h \le \mathcal{R}h, \quad \|\mathcal{R}h\|_{\Y'(Mw)} \leq 2 \|h\|_{\Y'(Mw)} \quad\text{and}\quad 
[\mathcal{R}h \cdot v]_{A_1} \le 2 \|M'_v\|_{\Y'(Mw) \to \Y'(Mw)}. 
\end{align}

To proceed, we present several facts for the maximal operator $M$. First, we have 
\begin{align}\label{eq:MMM}
Mf(x) &= \sup_{Q \ni x} \fint_{Q} |f(y)| dy
\le \sup_{Q \ni x} 2^n \fint_{Q(x, \ell(Q))} |f(y)| dy 
\nonumber \\
&=\sup_{Q \ni x} 2^n \fint_{Q(x, \ell(Q))} |f| dw \fint_{Q(x, \ell(Q))} w\, dy
\lesssim M_w^c(f w^{-1})(x) Mw(x), 
\end{align}
where $Q(x,\ell(Q))$ denotes the cube centered at $x$ with side-length $2\ell(Q)$. Gathering  \eqref{eq:MMM} and Lemma \ref{lem:MwcXw} \eqref{list-3}, we obtain 
\begin{equation}\label{eq:MfMw}
\bigg\|\frac{Mf}{Mw} \bigg\|_{\X(w)} 
\lesssim \|M_w^c(fw^{-1})\|_{\X(w)} 
\lesssim \|fw^{-1}\|_{\X(w)}. 
\end{equation}
On the other hand, as we see in \cite[eq.(2.2)]{LOP2}, 
\begin{align}\label{eq:M3M2}
\bigg(\frac{Mf}{M^3f} \bigg)^{\frac12} \le C_n \frac{Mf}{M^2f}.
\end{align} 

By \eqref{eq:ass-X}, there exists a nonnegative function $h \in \Y'(Mw)$ with $\|h\|_{\Y'(Mw)} \le 1$ such that $h$ is non-zero on a set of positive measure, and 
\begin{align}\label{eq:fM3w}
\bigg\|\frac{f}{M^3w}\bigg\|_{\X(Mw)}^{\frac1s} 
=\bigg\|\bigg(\frac{f}{M^3w} \bigg)^{\frac1s}\bigg\|_{\Y(Mw)}
\le 2\int_{\Rn} f^{\frac1s} \frac{h \ Mw}{(M^3w)^{\frac1s}} dx. 
\end{align}
Since $(M^3w)^{\frac{1}{2s}} \in A_1$ by \eqref{eq:CR}, it follows from \eqref{eq:rev-fac} and \eqref{eq:RdF-YMw} that 
\begin{align}\label{eq:M3wA2}
w_0 := (\mathcal{R}h \cdot v) (M^3w)^{-\frac{1}{2s}} 
=(\mathcal{R}h \cdot v) [(M^3w)^{\frac{1}{2s}}]^{1-2} \in A_2. 
\end{align} 
Thanks to \eqref{eq:fM3w} and \eqref{eq:RdF-YMw}, one has 
\begin{align*}
\norm{\frac{f}{M^3w}}_{\X(Mw)}^{\frac1s} 
&\lesssim \int_{\Rn} f^{\frac1s} \frac{(\mathcal{R}h) Mw}{(M^3w)^{\frac1s}} dx
\le \int_{\Rn} f^{\frac1s} \frac{(\mathcal{R}h) (Mw)^{1-\frac{1}{2s}}}{(M^3w)^{\frac{1}{2s}}} dx
=\|f\|_{L^{\frac1s}(w_0)}^{\frac1s}. 
\end{align*} 
Sequently,  applying \eqref{eq:M3M2} and the H\"{o}lder's inequality \eqref{eq:Holder} for $\Y(Mw)$, we deduce that 
\begin{align*}
&\|Mg\|_{L^{\frac1s}(w_0)}^{\frac1s}
=\int_{\Rn} (Mg)^{\frac1s} \frac{(\mathcal{R}h) (Mw)^{1-\frac{1}{2s}}}{(M^3w)^{\frac{1}{2s}}} dx 
\lesssim \int_{\Rn} \bigg(\frac{Mg}{M^2w}\bigg)^{\frac1s} (\mathcal{R}h) Mw \, dx
\\
&\leq \bigg\|\bigg(\frac{Mg}{M^2w}\bigg)^{\frac1s}\bigg\|_{\Y(Mw)} \|\mathcal{R}h\|_{\Y'(Mw)} 
\lesssim \bigg\|\frac{Mg}{M^2w}\bigg\|_{\X(Mw)}^{\frac1s} \|h\|_{\Y'(Mw)} 
\lesssim \bigg\|\frac{g}{Mw}\bigg\|_{\X(Mw)}^{\frac1s}, 
\end{align*}
where \eqref{eq:MfMw} was used in the last step. This proves \eqref{eq:M3w-3}.  
\end{proof}

\begin{proof}[\bf Proof of Theorem \ref{thm:wMw}]
Let $\Y=\X^{\frac{1}{p_0}}$. Then $\Y$ is a BFS and 
\begin{align}\label{eq:p01}
\|f\|_{\X_w}^{p_0} = \|f^{p_0}\|_{\Y_w} 
=\sup_{\substack{0 \leq h \in \Y'_w \\ \|h\|_{\Y'_w \leq 1}}} \int_{\Rn} f^{p_0} h w \ dx. 
\end{align}
Fix a nonnegative function $h \in \Y'_w$ with $\|h\|_{\Y'_w \leq 1}$. Note that $q_{\X}<\infty$ implies that $p_{\Y'}=(q_{\Y})'=(q_{\X}/p_0)'>1$. Invoking \eqref{eq:hyp-M}, \eqref{eq:MMM} and Lemma \ref{lem:MwcXw} \eqref{list-3}, we conclude that 
\begin{align}\label{eq:p02}
\int_{\Rn} f^{p_0} h w\ dx 
\lesssim \int_{\Rn} g^{p_0} M(h w) dx 
&\leq \int_{\Rn} g^{p_0} Mw \ M^c_w h \ dx 
\nonumber \\
= \int_{\Rn} g^{p_0} (Mw/w) \ M^c_w h \ w \ dx 
&\leq \|g^{p_0} (Mw/w)\|_{\Y_w} \|M^c_w h\|_{\Y'_w}
\nonumber \\
&\lesssim \|g (Mw/w)^{\frac{1}{p_0}}\|_{\X_w}^{p_0} \|h\|_{\Y'_w}. 
\end{align}
Therefore, \eqref{eq:con-M} as desired follows from \eqref{eq:p01} and \eqref{eq:p02}. The inequality \eqref{eq:vec-M} is a consequence of \eqref{eq:con-M}. Indeed, given $q \in (0,\infty)$, we define a new family $\F_q$ by 
\begin{align*}
\F_q=\bigg\{(F, G):= \bigg(\Big(\sum_{j=1}^{\infty} |f_j|^q \Big)^{\frac1q},  
\Big(\sum_{j=1}^{\infty} |g_j|^q \Big)^{\frac1q} \bigg): \{(f_j, g_j)\}_{j} \subset \F \bigg\}.
\end{align*} 
Then by \eqref{eq:hyp-M}, one has for every weight $w_0$ and every pair $(F, G) \in \F_q$,  
\begin{align}\label{eq:FG}
\|F\|_{L^q(w_0)} = \bigg(\sum_{j=0}^{\infty}\|f_j\|_{L^q(w_0)}^q\bigg)^{\frac1q}
\leq C \bigg(\sum_{j=0}^{\infty}\|g_j\|_{L^q(Mw_0)}^{q}\bigg)^{\frac1q} = \|G\|_{L^q(Mw_0)}. 
\end{align}
This shows \eqref{eq:hyp-M} holds for the exponent $p_0=q$ and the pair $(f,g)=(F, G)$. Thus,  \eqref{eq:con-M} immediately implies \eqref{eq:vec-M}. 
\end{proof}

\begin{proof}[\bf Proof of Theorem \ref{thm:uvw}]
Recall that $\Y_u$ and $\Y_v$ are Banach function spaces. Then by \eqref{eq:ass-X},  
\begin{align}\label{eq:uvw-4}
\|f w_1\|_{\X_u}^{p_0} = \|f^{p_0} w_1^{p_0}\|_{\Y_u} 
=\sup_{\substack{0 \leq h \in \Y'_u \\ \|h\|_{\Y'_u \leq 1}}} \int_{\Rn} f^{p_0} w_1^{p_0} h u\ dx. 
\end{align}
Now fix a nonnegative function $h \in \Y'_u$ such that $\|h\|_{\Y'_u} \leq 1$. By the hypothesis \eqref{eq:uvw-1}, \eqref{eq:uvw-2} and H\"{o}lder's inequality \eqref{eq:Holder}, we have 
\begin{align*}
\int_{\Rn} f^{p_0} w_1^{p_0} h u \ dx 
&\lesssim \int_{\Rn} g^{p_0} M_{\Phi}(w_1^{p_0} h u) dx 
= \int_{\Rn} g^{p_0} w_2^{p_0} \cdot M_{\Phi}(w_1^{p_0} hu) w_2^{-p_0} dx 
\nonumber \\
&\leq \|g^{p_0} w_2^{p_0}\|_{\Y_v} \|M_{\Phi}(w_1^{p_0} hu) w_2^{-p_0} v^{-1}\|_{\Y'_v}
\lesssim \|g w_2\|_{\X_v}^{p_0} \|h\|_{\Y'_v}. 
\end{align*}
This and \eqref{eq:uvw-4} immediately give \eqref{eq:uvw-3}.  
\end{proof}

\begin{proof}[\bf Proof of Theorem \ref{thm:ABC}]
Note that $\Y_v=\X_v^{\frac{1}{p_0}}$ and $\Y_{M_A v}=\X_{M_A v}^{\frac{1}{p_0}}$ are BFS. Thanks to \eqref{eq:ass-X}, we see that 
\begin{align}\label{eq:ABC-5}
\|f u\|_{\X_v}^{p_0} = \|f^{p_0} u^{p_0}\|_{\Y_v} 
=\sup_{\substack{0 \leq h \in \Y'_v \\ \|h\|_{\Y'_v \leq 1}}} \int_{\Rn} f^{p_0} u^{p_0} hv \ dx. 
\end{align}
Let $0 \le h \in \Y'_v$ with $\|h\|_{\Y'_v} \leq 1$. By our assumption on $\Phi$ and \eqref{eq:Holder-ABC}, one has 
\begin{align}\label{eq:MAMB}
M_{\Phi}(f g)(x) \lesssim M_A f(x) M_Bg(x),\quad\forall x \in \Rn. 
\end{align} 
Thus, applying the hypothesis \eqref{eq:ABC-1}, \eqref{eq:MAMB} and H\"{o}lder's inequality \eqref{eq:Holder},  we obtain 
\begin{align}\label{eq:hYY}
\int_{\Rn} f^{p_0} u^{p_0} h v \ dx 
&\lesssim \int_{\Rn} g^{p_0} M_{\Phi}(u^{p_0} hv) dx 
\lesssim \int_{\Rn} g^{p_0} M_A(u^{p_0}) M_B(hv) \ dx 
\nonumber \\
&\leq \|g^{p_0} M_A(u^{p_0})\|_{\Y_v} \|M_B(hv)v^{-1}\|_{\Y'_v}
\lesssim \|g M_A(u^{p_0})^{\frac{1}{p_0}}\|_{\X_v}^{p_0} \|h\|_{\Y'_v}, 
\end{align}
where \eqref{eq:ABC-2} was used in the last step. On the other hand, by \eqref{eq:ABC-4},  
\begin{align}\label{eq:YvYv}
\int_{\Rn} f^{p_0} u^{p_0} h v \ dx 
&\lesssim \int_{\Rn} g^{p_0} M_{\Phi}(u^{p_0} hv) dx 
\lesssim \int_{\Rn} g^{p_0} M_B(u^{p_0}h) M_Av \ dx 
\nonumber \\
&\leq \|g^{p_0}u^{p_0}\|_{\Y_{M_Av}} \|M_B(u^{p_0}h)u^{-p_0}\|_{\Y'_{M_Av}} 
\lesssim \|gu\|_{\X_{M_Av}}^{p_0} \|h\|_{\Y'_v}. 
\end{align}
Consequently, \eqref{eq:ABC-3} and \eqref{eq:ABC-4} follow at once from \eqref{eq:ABC-5}, \eqref{eq:hYY} and \eqref{eq:YvYv}.  
\end{proof}

The classical Fefferman-Stein inequality in \cite{FS} asserts that for each $p \in (1, \infty)$,   
\begin{equation}\label{eq:FS}
\|Mf\|_{L^p(w)} \leq C_p \|f\|_{L^p(Mw)}\quad\text{ for every weight $w$}. 
\end{equation}
In \cite[Theorem~1.7]{P95}, P\'{e}rez proved that for any $1<p<\infty$, $\Phi \in B_p$ if and only if 
\begin{align}\label{eq:FS-1}
\|M_{\Phi}f\|_{L^p(w)} \leq C \|f\|_{L^p(Mw)} \quad\text{ for every weight $w$}.  
\end{align}
Another result in \cite[Theorem~6.1]{P95} states that $r<p$ if and only if 
\begin{align}\label{eq:FS-2} 
\|M_{r,s}f\|_{L^p(w)} \leq C \|f\|_{L^p(Mw)} \quad\text{ for every weight $w$},  
\end{align}
whenever $1<p,r<\infty$ and $1 \le s < \infty$, where $M_{r, s}$ is the maximal operator associated to Lorentz space $L^{r,s}(\Rn)$ defined by 
\begin{align*}
M_{r,s}f(x) := \sup_{Q \ni x} |Q|^{-\frac1r} \|f\mathbf{1}_{Q}\|_{L^{r,s}(\Rn)}. 
\end{align*}
Moreover, note that for any Young function $A$, 
\begin{align}\label{eq:MAA}
M(fg)(x) \leq 2 M_{A}f(x) M_{\bar{A}}g(x),\quad\forall x \in \Rn. 
\end{align}
As a consequence, by \eqref{eq:FS}--\eqref{eq:MAA} and the approach used in \eqref{eq:FG}, Theorem  \ref{thm:ABC}  implies Theorem \ref{thm:FS-X}.

Now let us see how Theorem \ref{thm:FS-X} recovers the sharp vector-valued inequality obtained in \cite[Theorem~1.4]{CP} and \cite[Theorem~1.1]{P00}. Denote 
\begin{align*}
r=p/q, \quad \varepsilon=[r]+1-r,\quad A_0(t)=A(t^{1/r}) \quad\text{and}\quad A(t)=t^r \log(e+t)^{r-1+\varepsilon}. 
\end{align*}
Then, one has 
\begin{align}\label{eq:AABr}
A_0(t) \simeq t\log(e+t)^{[r]} \quad\text{and}\quad 
\bar{A}(t) \simeq t^{r'} \log(e+t)^{-1-(r'-1)\varepsilon} \in B_{r'}.
\end{align} 
Note that for $\Phi_k(t)=t\log(e+t)^k$, 
\begin{align}\label{eq:MkMk}
M^{k+1} f(x) \simeq M_{\Phi_k}f(x),\quad x \in \Rn \text{ and } k \in \N_{+}. 
\end{align} 
Let $w$ be an arbitrary weight. Thus, the former in \eqref{eq:AABr} and \eqref{eq:MkMk} imply
\begin{align*}
M_A(w^{1/r})(x)^r = M_{A_0}w(x) \simeq M^{[r]+1}w(x).  
\end{align*}
Also, it follows from \eqref{eq:AABr} and Lemma \ref{lem:MBp} that $M_{\bar{A}}$ is bounded on $L^{r'}(\Rn)$. Accordingly, the inequality \eqref{eq:MMA-1} applied to $u=w^{\frac1p}$, $v \equiv 1$ and $\X_v=L^p(\Rn)$ gives the Corollary \ref{cor:FS-Lp}.

\section{Extrapolation on modular spaces}\label{sec:modular}

We turn now our attention to the extrapolation on modular spaces. In this context, for a weight $w$ and a Young function $\Phi$, we define the modular $\rho_w^{\Phi}$ of $f \in \M$ by
\begin{equation*}
\rho_w^{\Phi}(f):=\int_{\Rn} \Phi(|f(x)|)w(x)\, dx. 
\end{equation*}
When $w \equiv 1$, we denote $\rho^{\Phi}$ instead of $\rho_w^{\Phi}$. The collection of functions, 
\begin{align*}
\mathcal{M}_{w}^{\Phi} :=\{f \in \M: \rho_{w}^{\Phi}(f)<\infty\}. 
\end{align*}
is referred to as a {\tt modular space}. 

\begin{theorem}\label{thm:ModAi}
Let $\Phi$ be a Young function with $1<i_{\Phi} \le I_{\Phi}<\infty$. If for some $p_0 \in (0, \infty)$ and for every $w \in A_{\infty}$, 
\begin{equation}\label{eq:ModAi-1}
\|f\|_{L^{p_0}(w)} \leq C \|g\|_{L^{p_0}(w)}, \quad (f,g) \in \F, 
\end{equation}
then for every $u \in RH_{\infty}$ and $v \in A_{\infty}$,   
\begin{align}\label{eq:ModAi-2}
\int_{\Rn} \Phi(fu)v\, dx \le C \int_{\Rn} \Phi(gu)v\, dx, \quad (f,g) \in \F.
\end{align}
Moreover, the following vector-valued inequality holds for every $q \in (0, \infty)$, 
\begin{equation}\label{eq:ModAi-vec}
\int_{\Rn}\Phi \bigg(\Big(\sum_j f_j^q \Big)^{\frac1q}u \bigg)v\,dx 
\leq C\, \int_{\Rn}\Phi \bigg(\Big(\sum_j g_j^q \Big)^{\frac1q}u \bigg)v\,dx,\quad \{(f_j, g_j)\} \subset \F. 
\end{equation}
\end{theorem}

\begin{proof}
We begin with a claim: for every pair $(f, g) \in \F$ with $\rho_v^{\Phi}(fu)<\infty$ and $\rho_v^{\Phi}(gu)<\infty$, there exists a weight $w=w(f, g) \in A_{\infty}$ such that  
\begin{align}\label{eq:ModAi-3} 
\rho_v^{\Phi}(fu) \le \|f\|_{L^1(w)} \quad\text{and}\quad 
\|g\|_{L^1(w)} \le C_0 \big(\varepsilon \rho_v^{\Phi}(fu) + \varepsilon^{-r} \rho_v^{\Phi}(gu) \big),
\end{align}
where $\varepsilon \in (0, 1)$ is an arbitrary number, $C_0>1$ and $r>1$ are independent of $\varepsilon$. Assuming this momentarily, let us see how \eqref{eq:ModAi-2} follows from \eqref{eq:ModAi-3}. Observe that by Theorem \textbf{B}, \eqref{eq:ModAi-1} implies that 
 for every $p \in (0, \infty)$ and for every $w \in A_{\infty}$, 
\begin{equation}\label{eq:ModAi-4}
\|f\|_{L^p(w)} \leq C_1 \|g\|_{L^p(w)}, \quad (f,g) \in \F. 
\end{equation}
Given $(f, g) \in \F$, we may assume that $\rho_v^{\Phi}(gu)<\infty$. This in turn implies $f<\infty$ a.e.. Otherwise, there exists a measurable set $E\subset \Rn$ with $|E|>0$ such that $f=\infty$ on $E$. By means of \eqref{eq:ModAi-4}, it holds 
\begin{align}\label{eq:gi}
\|g\|_{L^1(w)}= \infty \quad\text{ for every } w \in A_{\infty}.  
\end{align}
On the other hand, applying our claim to $f=g$ and $\varepsilon=1/2$, we get a weight $w_0=w_0(g) \in A_{\infty}$ such that 
$\|g\|_{L^1(w_0)} \le C_0 2^{r+1} \rho_v^{\Phi}(gu) < \infty$, which contradicts \eqref{eq:gi}. 

For every $N \ge 1$, we define $f_N := f\, {\bf 1}_{\{x\in B(0, N) : f(x)\leq N, u(x) \le N, v(x) \le N\}}$. This gives that 
\begin{align*}
\rho_v^{\Phi}(f_Nu) \le N \Phi(N^2) |B(0, N)| < \infty. 
\end{align*}
Then \eqref{eq:ModAi-3} applied to the pair $(f_N, g)$ give that $w=w(f_N, g) \in A_{\infty}$ 
\begin{align}
\label{eq:PfNu-1} \rho_v^{\Phi}(f_Nu) \le \|f_N\|_{L^1(w)}\quad\text{and}\quad 
\|g\|_{L^1(w)} \le C_0 \big(\varepsilon \rho_v^{\Phi}(f_Nu) + \varepsilon^{-r} \rho_v^{\Phi}(gu)\big),
\end{align}
where $\varepsilon \in (0, 1)$ is an arbitrary number, $C_0>1$ and $r>1$ are independent of $\varepsilon$ and $N$. Gathering \eqref{eq:ModAi-4} and  \eqref{eq:PfNu-1}, we conclude that 
\begin{align}\label{eq:PfPg}
\rho_v^{\Phi}(f_Nu) & \le \|f_N\|_{L^1(w)} \le \|f\|_{L^1(w)} \le C_1 \|g\|_{L^1(w)} 
\nonumber\\
&\le C_0 C_1 \big(\varepsilon \rho_v^{\Phi}(f_Nu) + \varepsilon^{-r} \rho_v^{\Phi}(gu)\big) 
\le \frac12 \rho_v^{\Phi}(f_Nu) + C_0 C_1 \varepsilon^{-r} \rho_v^{\Phi}(gu), 
\end{align}
provided $0<\varepsilon<\frac{1}{2C_0 C_1}$. Therefore, \eqref{eq:PfPg} gives that 
\begin{align}\label{eq:Nugu}
\rho_v^{\Phi}(f_Nu) &\le 2C_0 C_1 \varepsilon^{-r} \rho_v^{\Phi}(gu). 
\end{align}
Recall that $f<\infty$ a.e., and hence $f_N \nearrow f$ as $N \to \infty$. Consequently, \eqref{eq:ModAi-2} follows at once from \eqref{eq:Nugu} and the monotone convergence theorem. The vector-valued inequality \eqref{eq:ModAi-vec} can be shown as before. 

In order to show \eqref{eq:ModAi-3}, we invoke a result contained in \cite[Proposition~5.1]{CGMP}. That is, for every $\phi \in \Delta_2$ and for every $w \in A_{\infty}$, 
\begin{align}\label{eq:ModAi-w}
\int_{\Rn} \phi(M_wf(x))w(x)\, dx \le K_0 \int_{\Rn} \phi(|f(x)|) w(x)\, dx, 
\end{align}
where $K_0:=K_0(\phi, w) \ge 1$ is independent of $f$. For a non-negative function $h$ with $\rho_v^{\bar{\Phi}}(h)<\infty$ such that $h$ is non-zero on a set of positive measure, we define the Rubio de Francia iteration algorithm  as: 
\begin{equation*}
\mathcal{R}_v h :=\sum_{j=0}^{\infty} \frac{M_v^j h}{2^j K_0^j},
\end{equation*} 
where $K_0=K_0(\bar{\Phi}, v)$ is defined in \eqref{eq:ModAi-w}. The following properties are verified:  
\begin{align}\label{eq:ModAi-RdF}
h \leq \mathcal{R}_v h,\quad \rho_v^{\bar{\Phi}}(\mathcal{R}_v h) \leq 2 \rho_v^{\Phi}(2h)   
\quad\text{and}\quad \left[\mathcal{R}h\right]_{A_1(v)} \leq 2K_0.
\end{align}
The proof of the first and the last estimates in \eqref{eq:ModAi-RdF} is the same as for the analogue properties in the proof of Theorem~\ref{thm:Aq}. Now write $\lambda:=\frac{1}{2K_0} \leq \frac12$ (see \cite[Lemma~4.2]{CMM}). From this and the formula for the sum of a geometric series, we obtain that
\begin{equation*}
\sum_{k=0}^{\infty}(1-\lambda)\lambda^k = 1. 
\end{equation*}
In addition, by the convexity of $\Phi$ and \eqref{eq:ModAi-w} for $\phi=\bar{\Phi}$, one has 
\begin{align*}
\rho_v^{{\bar{\Phi}}}(\mathcal{R}_v h) 
&=\int_{\Rn} \bar{\Phi}\bigg(\sum_{k=0}^{\infty}(1-\lambda)\,\lambda^k\, 
M_v^k \Big(\frac{h}{1-\lambda}\Big)\bigg)v\,dx 
\nonumber\\[4pt]
&\leq (1-\lambda) \sum_{k=0}^{\infty} \lambda^k \int_{\Rn} 
\bar{\Phi}\bigg(M_v^k \Big(\frac{h}{1-\lambda}\Big) \bigg)v\,dx
\nonumber\\[4pt]
&\leq(1-\lambda) \sum_{k=0}^{\infty} \lambda^k K_0^k \int_{\Rn} 
\bar{\Phi}\Big(\frac{h}{1-\lambda}\Big)v\,dx
\nonumber\\[4pt]
&=2(1-\lambda) \int_{\Rn}{\bar{\Phi}}\Big(\frac{h}{1-\lambda}\Big)v\,dx
\leq 2\,\rho_v^{\bar{\Phi}}(2h),
\end{align*}
which proves the second one in \eqref{eq:ModAi-RdF}. To proceed, fix a pair $(f, g) \in \F$ with $\rho_v^{\Phi}(fu)<\infty$ and $\rho_v^{\Phi}(gu)<\infty$. We define 
\begin{equation*}
h(x):=
\left\lbrace
\begin{array}{ll}
\displaystyle\frac{\Phi(f(x)u(x))}{f(x)u(x)}, &\text{ if }f(x)u(x)\neq 0,
\\[8pt]
0, &\text{ if }f(x)u(x)=0,
\end{array}
\right.
\quad x\in\Rn. 
\end{equation*}
We may assume that $h$ is non-zero on a set of positive measure. Otherwise, $h=0$ a.e., and hence, $fu=0$ a.e.. The later implies that $f=0$ a.e. since $0<u<\infty$ a.e.. Thus, \eqref{eq:ModAi-2} holds. From \eqref{eq:ModAi-RdF} and Lemma \ref{lem:A1Ap} \eqref{eq:AiRH} and \eqref{eq:vA1u}, we deduce that $w:=u(\mathcal{R}_vh)v \in A_{\infty}$. Fix $\varepsilon \in (0,1)$ and pick $r>I_{\Phi}$. It follows from \eqref{eq:Ii} that there is a constant $C_{\Phi}>0$ such that 
\begin{equation}\label{eq:Phits} 
\Phi(ts) \leq C_{\Phi}\,t^r\Phi(s) \quad\text{ for every } s>0 \text{ and } t>1.
\end{equation} 
This gives that 
\begin{align}\label{eq:egu-1}
\rho_v^{\Phi}(\varepsilon^{-1} gu) 
\le  C_{\Phi} \varepsilon^{-r} \rho_v^{\Phi}(gu) . 
\end{align}
Additionally, from the convexity of $\bar{\Phi}$, \eqref{eq:ModAi-RdF}, \eqref{eq:Young-2} and the definition of $h$, we obtain that 
\begin{equation}\label{eq:egu-2}
\rho_v^{\bar{\Phi}}(\varepsilon \mathcal{R}_v h) 
\le \varepsilon \rho_v^{\bar{\Phi}}(\mathcal{R}_v h) 
\le 2\varepsilon \rho_v^{\bar{\Phi}}(2h) 
\leq \varepsilon C \rho_v^{\bar{\Phi}}(2h) 
\le \varepsilon C \rho_v^{\Phi}(fu).
\end{equation}
By \eqref{eq:ModAi-RdF} again, one has 
\begin{align*}
\rho_v^{\Phi}(fu) &=\int_{\Rn} f\,u\, h\,v\,dx
\leq\int_{\Rn} f \,u\, (\mathcal{R}_v h)v\,dx = \|f\|_{L^1(w)}. 
\end{align*} 
On the other hand, gathering \eqref{eq:stst}, \eqref{eq:egu-1} and \eqref{eq:egu-2}, we deduce that
\begin{align*}
\|g\|_{L^1(w)} &= \int_{\Rn} gu(\mathcal{R}_v h)v\,dx 
\le \rho_v^{\bar{\Phi}}(\varepsilon \mathcal{R}_v h) + \rho_v^{\Phi}(\varepsilon^{-1} gu) 
\le C_0 \big(\varepsilon\rho_v^{\Phi}(fu) + \varepsilon^{-r} \rho_v^{\Phi}(gu) \big).
\end{align*}
This proves \eqref{eq:ModAi-3} and completes the proof. 
\end{proof}

Theorems \ref{thm:AiAi} and \ref{thm:ModAi} immediately imply the following result. 

\begin{corollary}\label{cor:BFSMod}
Let $\Phi(t)=t$ or $\Phi$ be a Young function with $1<i_{\Phi} \le I_{\Phi}<\infty$. Suppose that $\X$ is a RIBFS over $(\Rn, dx)$ such that $q_{\X}<\infty$. If for some $p_0 \in (0, \infty)$ and for every $w \in A_{\infty}$, 
\begin{equation}
\|f\|_{L^{p_0}(w)} \leq C \|g\|_{L^{p_0}(w)}, \quad (f,g) \in \F,  
\end{equation}
then for every $u \in RH_{\infty}$ and for every $v \in A_{\infty}$,    
\begin{align}
\|\Phi(fu)\|_{\X_v} \le C \|\Phi(gu)\|_{\X_v}, \quad (f,g) \in \F.
\end{align}
\end{corollary}

Recalling \eqref{eq:MRH} and \eqref{eq:MPhiRH}, we have that $u:=(M_A f)^{-\lambda} \in RH_{\infty}$ for any $\lambda>0$, where $A(t)=t$ or $A$ is a Young function. Note that for any invertible function $\Phi$, 
\begin{align*}
\|\Phi(f)\|_{L^{1,\infty}(w)} 
= \sup_{\lambda>0} \lambda w(\{x: \Phi(f(x))>\lambda\}) 
=\sup_{\lambda>0} \Phi(\lambda) w(\{x: f(x)>\lambda\}). 
\end{align*}
From these and Corollary \ref{cor:BFSMod}, we conclude the estimates below. 

\begin{corollary}\label{cor:TM}
Let $T$ be an operator, $A(t)=t$ or  $A$ be a Young function, and $\Phi(t)=t$ or $\Phi$ be Young function such that $1<i_{\Phi} \le I_{\Phi}<\infty$. Let $E \subset \Rn$ be a measurable set. Suppose that $\X$ is a RIBFS over $(\Rn, dx)$ such that $q_{\X}<\infty$. If for some $p_0 \in (0, \infty)$ and for every $w \in A_{\infty}$, 
\begin{equation}
\|Tf\|_{L^{p_0}(E, w)} \leq C \|M_A f\|_{L^{p_0}(E, w)}. 
\end{equation}
then for every $v \in A_{\infty}$, 
\begin{align}
\bigg\|\Phi\bigg(\frac{|Tf|}{M_A f} \bigg) \mathbf{1}_E\bigg\|_{\X_v} 
\le C \|\mathbf{1}_E\|_{\X_v}.
\end{align}
In particular, 
\begin{align*}
\sup_{\lambda>0} \Phi(\lambda) \, v(\{x \in E: |Tf(x)|>\lambda M_Af(x)\}) \le C v(E). 
\end{align*}
\end{corollary}

\begin{theorem}\label{thm:PhiA1}
Let $\Phi$ be a Young function with $1<i_{\Phi} \le I_{\Phi}<\infty$. Let $u$ and $v$ be weights on $\Rn$ such that $v \in A_1$. If for some $p_0 \in (0,\infty)$ and for every $w \in A_1$,  
\begin{equation}\label{eq:PhiA1-2}
\|f\|_{L^{p_0}(w)} \leq C \|g\|_{L^{p_0}(w)}, \quad (f,g) \in \F, 
\end{equation} 
then 
\begin{equation}\label{eq:PhiA1-3}
\int_{\Rn}\Phi(f^{p_0} u^{p_0})v\,dx 
\leq C \int_{\Rn}\Phi(g^{p_0} u^{p_0})v\,dx,\quad (f,g) \in \F, 
\end{equation}
provided that 
\begin{align}\label{eq:PhiA1-1} 
\int_{\Rn} \bar{\Phi}((M_v f) u^{-p_0}) v\, dx 
\le K_0 \int_{\Rn} \bar{\Phi}(f u^{-p_0}) v\, dx, \quad\forall f \in \M. 
\end{align}
\end{theorem}

\begin{proof} 
As before, it suffices to show that: if \eqref{eq:PhiA1-1} holds, then for every pair $(f, g) \in \F$ with $\rho_v^{\Phi}(f^{p_0}u^{p_0})<\infty$ and $\rho_v^{\Phi}(g^{p_0}u^{p_0})<\infty$, there exists a weight $w=w(f, g) \in A_1$ such that 
\begin{align}\label{eq:PhiA1-4}
\rho_v^{\Phi}(f^{p_0}u^{p_0}) \le \|f\|_{L^{p_0}(w)}^{p_0} \quad\text{and}\quad 
\|g\|_{L^{p_0}(w)}^{p_0} \le C_0 \big(\varepsilon \rho_v^{\Phi}(f^{p_0}u^{p_0}) + \varepsilon^{-r} \rho_v^{\Phi}(g^{p_0}u^{p_0}) \big), 
\end{align}
where $\varepsilon \in (0, 1)$ is an arbitrary number, $C_0>1$ and $r>1$ are independent of $\varepsilon$. 

To demonstrate \eqref{eq:PhiA1-4}, we fix a pair $(f, g) \in \F$ with $\rho_v^{\Phi}(f^{p_0}u^{p_0})<\infty$ and $\rho_v^{\Phi}(g^{p_0}u^{p_0})<\infty$. Define 
\begin{equation*}
h(x):=
\left\lbrace
\begin{array}{ll}
\displaystyle\frac{\Phi(f(x)^{p_0}u(x)^{p_0})}{f(x)^{p_0}u(x)^{p_0}}, &\text{ if }f(x)u(x)\neq 0,
\\[8pt]
0, &\text{ if }f(x)u(x)=0,
\end{array}
\right.
\quad x\in\Rn. 
\end{equation*}
For the function $h$, we define 
\begin{equation*}
\mathcal{R} h :=\sum_{j=0}^{\infty}\frac{M_v^j h}{2^j K_0^j} \quad\text{and}\quad 
H := \mathcal{R}(hu^{p_0}) u^{-p_0}.
\end{equation*} 
Together with \eqref{eq:PhiA1-1}, a straigtforward calculation gives that  
\begin{align}\label{eq:PhiA1-RdF}
h \leq H,\quad  
\rho_v^{\bar{\Phi}}(H) \leq 2 \rho_v^{\bar{\Phi}}(2h) \quad\text{and}\quad
[\mathcal{R}h]_{A_1(v)} \leq 2K_0.
\end{align}
In view of Lemma \ref{lem:A1Ap} \eqref{eq:vA1u} and $v \in A_1$, we have $w:=u^{p_0}Hv=\mathcal{R}(hu^{p_0})v \in A_1$. Then the first estimate in \eqref{eq:PhiA1-RdF} gives that 
\begin{align*}
\rho_v^{\Phi}(f^{p_0}u^{p_0}) = \int_{\Rn} f^{p_0}u^{p_0} hv\, dx
\le \int_{\Rn} f^{p_0}u^{p_0} Hv\, d\mu = \|f\|_{L^{p_0}(w)}^{p_0}. 
\end{align*}
In addition, by \eqref{eq:stst}, the convexity of $\Phi$ and \eqref{eq:Phits}, we have  
\begin{align*}
\|g\|_{L^{p_0}(w)}^{p_0} &= \int_{\Rn} g^{p_0} u^{p_0} Hv\,dx 
\le \rho_v^{\bar{\Phi}}(\varepsilon H) + \rho_v^{\Phi}(\varepsilon^{-1} g^{p_0} u^{p_0}) 
\\
&\le \varepsilon \rho_v^{\bar{\Phi}}(H) + \rho_v^{\Phi}(\varepsilon^{-1} g^{p_0} u^{p_0})  
\le 2\varepsilon \rho_v^{\bar{\Phi}}(2h) + \varepsilon^{-r} C_{\Phi} \rho_v^{\Phi}(g^{p_0} u^{p_0}) 
\\
&\le C_0 \big(\varepsilon \rho_v^{\Phi}(f^{p_0} u^{p_0})  + \varepsilon^{-r} \rho_v^{\Phi}(g^{p_0} u^{p_0}) \big). 
\end{align*}
This proves \eqref{eq:PhiA1-4}.  
\end{proof}

\begin{theorem}\label{thm:PhiAi}
Let $\Phi$ be a Young function with $1<i_{\Phi} \le I_{\Phi}<\infty$. Let $u$ and $v$ be weights on $\Rn$ such that $v \in A_{\infty}$. If for some $p_0 \in (0,\infty)$ and for every $w \in A_{\infty}$,  
\begin{equation}\label{eq:PhiAi-2}
\|f\|_{L^{p_0}(w)} \leq C \|g\|_{L^{p_0}(w)}, \quad (f,g) \in \F, 
\end{equation} 
then for every $p \in (0, \infty)$, 
\begin{equation}\label{eq:PhiAi-3}
\int_{\Rn}\Phi(f^p u)v\,dx \leq C \int_{\Rn}\Phi(g^p u)v\,dx,\quad (f,g) \in \F, 
\end{equation}
provided that 
\begin{align}\label{eq:PhiAi-1} 
\int_{\Rn} \bar{\Phi}((M_v f) u^{-1}) v\, dx \le K_0 \int_{\Rn} \bar{\Phi}(f u^{-1}) v\, dx,\quad\forall f \in \M. 
\end{align}
Moreover, under the hypothesis \eqref{eq:PhiAi-1}, \eqref{eq:PhiAi-2} implies that for every $p, q\in (0, \infty)$, 
\begin{equation}\label{eq:PhiAi-4}
\int_{\Rn}\Phi \bigg(\Big(\sum_j f_j^q \Big)^{\frac1q}u \bigg)v\,dx 
\leq C\, \int_{\Rn}\Phi \bigg(\Big(\sum_j g_j^q \Big)^{\frac1q}u \bigg)v\,dx,\quad \{(f_j, g_j)\} \subset \F. 
\end{equation}
\end{theorem}

\begin{proof} 
By Theorem \textbf{B}, we have that for every $p \in (0,\infty)$ and for every $w \in A_{\infty}$, 
\begin{align}\label{eq:PhiAi-2p}
\|f\|_{L^p(w)} \leq C \|g\|_{L^p(w)}, \quad (f,g) \in \F.  
\end{align} 
This can be rewritten as 
\begin{align}\label{eq:PhiAi-p2}
\|f^p\|_{L^1(w)} \le C^p \|g^p\|_{L^1(w)}, \quad (f,g) \in \F.  
\end{align} 

It is enough to show that if \eqref{eq:PhiAi-1} holds, then for every pair $(f, g) \in \F$ with $\rho_v^{\Phi}(fu)<\infty$ and $\rho_v^{\Phi}(gu)<\infty$, there exists a weight $w=w(f, g) \in A_{\infty}$ such that $[w]_{A_{\infty}} \le 2 K_0[v]_{A_{\infty}}$, 
\begin{align}\label{eq:PhiAi-4}
\rho_v^{\Phi}(fu) \le \|f\|_{L^1(w)} \text{ and }
\|g\|_{L^1(w)} \le C_0 \big(\varepsilon \rho_v^{\Phi}(fu) + \varepsilon^{-r} \rho_v^{\Phi}(gu)\big), 
\end{align}
where $\varepsilon \in (0, 1)$ is an arbitrary number, $C_0>1$ and $r>1$ are independent of $\varepsilon$. Then \eqref{eq:PhiAi-3} will follow from \eqref{eq:PhiAi-3} and \eqref{eq:PhiAi-p2}.

Indeed, the proof of \eqref{eq:PhiAi-4} is similar to that of \eqref{eq:PhiA1-4}. We here only present the main difference. Define 
\begin{equation*}
\mathcal{R} h :=\sum_{j=0}^{\infty}\frac{M_v^j h}{2^j K_0^j},\quad\text{and}\quad 
H := \mathcal{R}(hu) u^{-1}.
\end{equation*} 
Together with \eqref{eq:PhiAi-1}, a straigtforward calculation gives that  
\begin{align}\label{eq:PhiAi-5}
h \leq H,\quad  
\rho_v^{\bar{\Phi}}(H) \leq 2 \rho_v^{\bar{\Phi}}(2h) \quad\text{and}\quad
[\mathcal{R}h]_{A_1(v)} \leq 2K_0.
\end{align}
Since $v \in A_{\infty}$, we have $w:=uHv=\mathcal{R}(hu)v \in A_{\infty}$ with $[w]_{A_{\infty}} \le [\mathcal{R}(hu)]_{A_1(v)} [v]_{A_{\infty}} \le 2 K_0 [v]_{A_{\infty}}$.  Using the same argument as in Theorem \ref{thm:PhiA1}, one can conclude \eqref{eq:PhiAi-4}. 

Finally, \eqref{eq:PhiAi-4} is a consequence of \eqref{eq:PhiAi-3} and \eqref{eq:PhiAi-2p}.  
\end{proof}

\section{Extrapolation for commutators}\label{sec:comm}

Our next goal is to establish the extrapolation for commutators in the case of two-weight and Banach function spaces. Let us first recall the notion of commutators.  Given an operator $T$ and measurable functions $\b=(b_1,\ldots,b_k)$, we define, whenever it makes sense, the commutator by 
\begin{align*}
C_{\b}(T)(f)(x) := T\bigg(\prod_{j=1}^k (b_j(x)-b_j(\cdot))f(\cdot)\bigg)(x). 
\end{align*}
When $\b=(b, \ldots, b)$, we denote $C_{b}^k(T)$ instead of $C_{\b}(T)$. 

We say that a sublinear operator $T$ is {\tt linearizable} if there exists a Banach space $\mathcal{B}$ and a $\mathcal{B}$-valued linear operator $\mathcal{T}$ such that $Tf(x) = \|\T f(x)\|_{\mathcal{B}}$. In this way, we set 
\begin{align*}
C_{\b}(T)f(x) := \|C_{\b}(T)f(x)\|_{\mathcal{B}} 
= \bigg\|\T \Big(\prod_{j=1}^k \big(b_j(x)-b_j(\cdot)\big) f(\cdot)\Big)(x)\bigg\|_{\mathcal{B}}. 
\end{align*}

Given a weight $\nu$, we say that a locally integrable function $f \in \BMO(\nu)$ if 
\begin{equation*}
\|f\|_{\BMO(\nu)} :=\sup_{Q} \frac{1}{\nu(Q)} \int_{Q} |f(x)-f_Q| \, dx < \infty. 
\end{equation*}
If $\nu \equiv 1$, we simply write $\BMO:=\BMO(\nu)$. Define a new Orlicz-type space $\mathcal{BMO}:=\mathcal{BMO}(\Rn)$ via the norm
\begin{equation*}
\|f\|_{\mathcal{BMO}} :=\sup_{Q} \|f-f_Q\|_{\exp L, Q}.  
\end{equation*}

The John-Nirenberg inequality says that there exists a dimensional constant $C_n$ such that for every $f \in \BMO$, 
\begin{align}\label{eq:BMO}
\|f\|_{\BMO} \le \|f\|_{\mathcal{BMO}} \le C_n \|f\|_{\BMO}. 
\end{align}
Thus, \eqref{eq:BMO} implies that $(\BMO, \|\cdot\|_{\BMO})$ and $(\mathcal{BMO}, \|\cdot\|_{\mathcal{BMO}})$ are equivalent quasi-normed spaces; nevertheless, the appearance of $\mathcal{BMO}$ in the statements of our results below is to emphasize that $\|\cdot\|_{\mathcal{BMO}}$ is used.

It is well-known that if $w \in A_{\infty}$, then $\log w \in \mathcal{BMO}$. The converse can be formulated as follows. Let $b \in \mathcal{BMO}$. Then for every $1 \leq p<\infty$, 
\begin{align}\label{eq:ebAp}
\lambda \in \R: |\lambda| \leq \min\{1, p-1\} \|b\|_{\mathcal{BMO}}^{-1} 
\quad\Longrightarrow\quad [e^{\lambda b}]_{A_p} \leq 4^{|\lambda| \|b\|_{\mathcal{BMO}}}.  
\end{align}
Let us recall the sharp reverse H\"{o}lder's inequality from \cite{CGPSZ, HP12, LOP}. For every $w \in A_p$ with $1 \le p \le \infty$,  
\begin{align}\label{eq:RHp}
\bigg(\fint_{Q} w^{r_w} dx \bigg)^{\frac{1}{r_w}} \le 2 \fint_Q w \, dx, 
\end{align}
for every cube $Q$, where 
\begin{equation*}
r_w=
\begin{cases}
1+\frac{1}{2^{n+1}[w]_{A_1}}, & p=1, \\
1+\frac{1}{2^{n+1+2p}[w]_{A_p}}, &p \in (1, \infty), \\
1+\frac{1}{2^{n+11}[w]_{A_{\infty}}}, &p=\infty. 
 \end{cases}
 \end{equation*}

Now we state the main theorem of this section.  
\begin{theorem}\label{thm:TTb}
Fix $1<s_i, \theta_i<\infty$ for $i=1,2,3,4$. Let $\sigma$ and $\nu$ be weights on $\Rn$, $\X_{\sigma}$ and $\Y_{\sigma}$ be Banach function spaces over $(\Rn, \sigma)$. Let $T$ be either a linear or a linearizable operator.  Assume that for every pair of weights $(\mu, \lambda)$ with $\nu=\mu/\lambda$ such that $(\mu^{\theta_1}, \lambda^{\theta_2}, \mu^{-\theta_3}, \lambda^{-\theta_4}) \in A_{s_1} \times A_{s_2} \times A_{s_3} \times A_{s_4}$,   
\begin{align}\label{eq:TTb-1}
\|(Tf)\lambda\|_{\Y_{\sigma}} 
\le \Psi \big([\mu^{\theta_1}]_{A_{s_1}}, [\lambda^{\theta_2}]_{A_{s_2}}, [\mu^{-\theta_3}]_{A_{s_3}}, 
[\lambda^{-\theta_4}]_{A_{s_4}} \big) \|f\mu\|_{\X_{\sigma}}, 
\end{align}
where $\Psi:[1, \infty) \times [1, \infty) \times [1, \infty) \times [1, \infty) \to [1, \infty)$ is an increasing function of each variable. Then for every $k \ge 1$, for every $\b=(b_1, \ldots, b_k) \in \mathcal{BMO}^k$ and for every pair of weights $(\mu, \lambda)$ with $\nu=\mu/\lambda$ such that $(\mu^{\theta_1}, \lambda^{\theta_2}, \mu^{-\theta_3}, \lambda^{-\theta_4}) \in A_{s_1} \times A_{s_2} \times A_{s_3} \times A_{s_4}$,   
\begin{multline*}
\|(C_{\b}(T)f) \lambda\|_{\Y_{\sigma}} 
\leq \bigg(\frac{2^{n+2+2\max\{s_i, s'_i\}}\theta k}{\min\{1, s_i-1\}}\bigg)^k 
\max\big\{[\mu^{\theta_1}]_{A_{s_1}}^{r_1}, [\lambda^{\theta_2}]_{A_{s_2}}^{r_2}, 
[\mu^{-\theta_3}]_{A_{s_3}}^{r_3}, [\lambda^{-\theta_4}]_{A_{s_4}}^{r_4} \big\}^k  
\\
\times \Psi \big(2^{t_1} [\mu^{\theta_1}]_{A_{s_1}}, 2^{t_2} [\lambda^{\theta_2}]_{A_{s_2}}, 
2^{t_3} [\mu^{-\theta_3}]_{A_{s_3}}, 2^{t_4} [\lambda^{-\theta_4}]_{A_{s_4}}\big)
\prod_{j=1}^k \|b_j\|_{\mathcal{BMO}} \|f \mu\|_{\X_{\sigma}},  
\end{multline*} 
where $\theta=\max\{\theta_1,\theta_2, \theta_3, \theta_4\}$, $r_i=\max\{1, \frac{1}{s_i-1}\}$ and $t_i:=s_i+2\min\{1, s_i-1\}$. 
\end{theorem}

\begin{proof}
We will borrow the idea from \cite{BMMST}, where the extrapolation for commutators was obtained in the one-weight case and in Lebesgue spaces. Let $u:=\mu_1^{\theta} \in A_{s_1}$. Then $v:=u^{1-s'_1} \in A_{s'_1}$. Due to \eqref{eq:RHp}, one has 
\begin{align}\label{eq:ru}
\bigg(\fint_{Q} u^{r_u} dx \bigg)^{\frac{1}{r_u}} \le 2 \fint_Q u \, dx \quad\text{and}\quad 
\bigg(\fint_{Q} v^{r_v} dx \bigg)^{\frac{1}{r_v}} \le 2 \fint_Q v \, dx 
\end{align}
for every cube $Q$, where 
\begin{align}\label{eq:ruru}
r_u=1+\frac{1}{2^{n+1+2s_1}[u]_{A_{s_1}}} \quad\text{and}\quad r_v=1+\frac{1}{2^{n+1+2s'_1}[v]_{A_{s'_1}}}
\end{align}
Set $\eta_1:=\min\{r_u, r_v\}$. Taking into account \eqref{eq:ru} and \eqref{eq:ruru}, we obtain 
\begin{align}\label{eq:uAs}
[\mu^{\theta_1 \eta_1}]_{A_{s_1}}^{\frac{1}{\eta_1}}
\le 2^{s_1} [\mu^{\theta_1}]_{A_{s_1}} \quad\text{with}\quad 
\eta'_1 = 1+2^{n+1+2\max\{s_1,s'_1\}}[\mu^{\theta_1}]_{A_{s_1}}^{\max\{1, \frac{1}{s_1-1}\}}. 
\end{align}
Analogously, 
\begin{align*}
{}[\lambda^{\theta_2\eta_2}]_{A_{s_2}}^{\frac{1}{\eta_2}} \le 2^{s_2} [\lambda^{\theta_2}]_{A_{s_2}} 
&\quad\text{with}\quad \eta'_2 = 1+2^{n+1+2\max\{s_2,s'_2\}}[\lambda^{\theta_2}]_{A_{s_2}}^{\max\{1, \frac{1}{s_2-1}\}}, 
\\
{}[\mu^{-\theta_3\eta_3}]_{A_{s_3}}^{\frac{1}{\eta_3}} \le 2^{s_3} [\mu^{-\theta_3}]_{A_{s_3}} 
&\quad\text{with}\quad \eta'_3=1+2^{n+1+2\max\{s_3,s'_3\}}[\lambda^{-\theta_3}]_{A_{s_3}}^{\max\{1, \frac{1}{s_3-1}\}},  
\\
{}[\lambda^{-\theta_4\eta_4}]_{A_{s_4}}^{\frac{1}{\eta_4}} \le 2^{s_4} [\lambda^{-\theta_4}]_{A_{s_4}} 
&\quad\text{with}\quad \eta'_4=1+2^{n+1+2\max\{s_4,s'_4\}}[\lambda^{-\theta_4}]_{A_{s_4}}^{\max\{1, \frac{1}{s_4-1}\}}. 
\end{align*}

To continue, we write 
\begin{align*}
\Phi_z f(x):= e^{z_1 b_1(x)+\cdots+z_k b_k(x)} T \big(e^{-z_1 b_1-\cdots-z_kb_k} f \big)(x),\quad z=(z_1\ldots,z_k) \in \C^k. 
\end{align*}
By Cauchy integral formula adapted to several complex variables, we get 
\begin{align}\label{eq:CIF}
C_{\b}(T)f(x) = \partial_{z} \Phi_z f(x) \big|_{z=0} = \frac{1}{(2\pi i)^k} \int_{\partial P_{\delta}} 
\frac{\Phi_z f(x)}{z_1^2 \cdots z_k^2} dz_1 \cdots dz_k,  
\end{align}
where $P_{\delta}=\{z=(z_1,\ldots,z_k) \in \C^k: |z_j|<\delta_j, j=1,\ldots,k\}$. 
Pick 
\begin{equation}\label{eq:del}
\delta_j := \frac{\min\{1, s_1-1, s_2-1, s_3-1, s_4-1,\}}{k \max\{\theta_1 \eta'_1, \theta_2 \eta'_2, \theta_3 \eta'_3, \theta_4 \eta'_4  \} \|b_j\|_{\mathcal{BMO}}}, \quad j=1,\ldots,k. 
\end{equation}
Fix $z \in \overline{P_{\delta}}$, denote 
\begin{align*}
U(x) :=\mu(x) W(x),\quad V(x) :=\lambda(x) W(x) \quad\text{and}\quad W(x):=e^{\Re(z_1) b_1(x)+\cdots+\Re(z_k)b_k(x)}. 
\end{align*}
Then it follows from H\"{o}lder's inequality and \eqref{eq:ebAp} that 
\begin{align*}
&\bigg(\fint_Q W^{\theta_1\eta'_1} \, dx\bigg) \bigg(\fint_Q W^{\theta_1 \eta'_1(1-s'_1)} \, dx\bigg)^{s_1-1} 
\\
&\leq \prod_{j=1}^k \bigg(\fint_Q e^{k\theta_1 \eta'_1 \Re(z_j) b_j} dx\bigg)^{\frac{1}{k}}
\bigg(\fint_Q e^{k\theta_1 \eta'_1 \Re(z_j) b_j(1-s'_1)} dx \bigg)^{\frac{s_1-1}{k}}
\\
&\leq \prod_{j=1}^k 4^{\theta_1 \eta'_1 |\Re(z_j)| \|b_j\|_{\mathcal{BMO}}} 
\leq \prod_{j=1}^k 4^{\frac{\min\{1, s_1-1\}}{k}} =4^{\min\{1, s_1-1\}}. 
\end{align*}
Together with H\"{o}lder's inequality and \eqref{eq:uAs}, this in turn gives 
\begin{align*}
&\bigg(\fint_Q U^{\theta_1} dx \bigg) \bigg(\fint_Q U^{\theta_1(1-s'_1)} \, dx\bigg)^{s_1-1} 
\\
&\qquad\qquad\leq \bigg(\fint_Q \mu^{\theta_1 \eta_1} dx\bigg)^{\frac{1}{\eta_1}} 
\bigg(\fint_{Q} W^{\theta_1 \eta'_1} dx\bigg)^{\frac{1}{\eta'_1}}  
\\
&\qquad\qquad\qquad\times \bigg(\fint_Q \mu^{\theta_1 \eta_1(1-s'_1)} dx\bigg)^{\frac{s_1-1}{\eta_1}} 
\bigg(\fint_{Q} W^{\theta_1\eta'_{\mu}(1-s'_1)} dx\bigg)^{\frac{s_1-1}{\eta'_1}} 
\\
&\qquad\qquad\leq 4^{\min\{1, s_1-1\}} [\mu^{\theta_1 \eta_1}]_{A_{s_1}}^{\frac{1}{\eta_1}} 
\leq 4^{\min\{1, s_1-1\}} 2^{s_1} [\mu^{\theta_1}]_{A_{s_1}}.   
\end{align*}
That is, 
\begin{align}\label{eq:UAs} 
[U^{\theta_1}]_{A_{s_1}} \le 4^{\min\{1, s_1-1\}} 2^{s_1} [\mu^{\theta_1}]_{A_{s_1}}. 
\end{align}
Similarly, 
\begin{align}
\label{eq:VAs} [V^{\theta_2}]_{A_{s_2}} &\le 4^{\min\{1, s_2-1\}} 2^{s_2} [\lambda^{\theta_2}]_{A_{s_2}}, 
\\
\label{eq:UAs3} [U^{-\theta_3}]_{A_{s_3}} &\le 4^{\min\{1, s_3-1\}} 2^{s_3} [\mu^{-\theta_3}]_{A_{s_3}}, 
\\
\label{eq:VAs4} [V^{-\theta_4}]_{A_{s_4}} &\le 4^{\min\{1, s_4-1\}} 2^{s_4} [\lambda^{-\theta_4}]_{A_{s_4}}. 
\end{align}
Note that $U/V=\mu/\lambda=\nu$. This together with \eqref{eq:UAs}--\eqref{eq:VAs4} and \eqref{eq:TTb-1} gives that  
\begin{align}\label{eq:fpq}
\|(\Phi_z f)\lambda\|_{\Y_{\sigma}} 
&= \|T(e^{-z_1b_1-\cdots-z_mb_m}f) V\|_{\Y_{\sigma}} 
\nonumber \\
&\le \Psi \big([U^{\theta_1}]_{A_{s_1}}, [V^{\theta_2}]_{A_{s_2}}, [U^{-\theta_3}]_{A_{s_3}}, 
[V^{-\theta_4}]_{A_{s_4}} \big) \|(W^{-1}f) U\|_{\X_{\sigma}} 
\nonumber \\
&\le \Psi \big(2^{t_1} [\mu^{\theta_1}]_{A_{s_1}}, 2^{t_2} [\lambda^{\theta_2}]_{A_{s_2}}, 
2^{t_3}[\mu^{-\theta_3}]_{A_{s_3}}, 2^{t_4} [\lambda^{-\theta_4}]_{A_{s_4}}\big) \|f \mu\|_{\X_{\sigma}}.  
\end{align}

Now fix a non-negative function $h \in \Y'_{\sigma}$ with $\|h\|_{\Y'_{\sigma}} \le 1$. Invoking \eqref{eq:CIF} and H\"{o}lder's inequality \eqref{eq:Holder}, we obtain 
\begin{align}\label{eq:CbYv}
&\int_{\Rn} |C_{\b}(T)f(x)| \lambda(x) h(x) v(x) dx 
\nonumber\\
&\le \int_{\Rn} \bigg(\frac{1}{(2\pi)^k} \int_{\partial P_{\delta}} 
\frac{|\Phi_zf(x)|}{|z_1|^2 \cdots |z_k|^2} |dz_1| \cdots |dz_k| \bigg) \lambda(x) h(x) v(x) dx 
\nonumber\\
&\le \prod_{j=1}^k \delta_j^{-2} \frac{1}{(2\pi)^k} \int_{\partial P_{\delta}} 
\bigg(\int_{\Rn} |\Phi_zf(x)| \lambda(x) h(x) v(x) \, dx \bigg)  |dz_1| \cdots |dz_k| 
\nonumber\\
&\le \prod_{j=1}^k \delta_j^{-2} \frac{1}{(2\pi)^k} \int_{\partial P_{\delta}} 
\|(\Phi_z f)\lambda\|_{\Y_{\sigma}} \|h\|_{\Y'_{\sigma}}  |dz_1| \cdots |dz_k|.  
\end{align}
Then, combining \eqref{eq:ass-X}, \eqref{eq:fpq} and \eqref{eq:CbYv}, we conclude that 
\begin{multline*}
\|(C_{\b}(T)f) \lambda\|_{\Y_{\sigma}} 
\le \prod_{j=1}^k \delta_j^{-2} \frac{1}{(2\pi)^k} \int_{\partial P_{\delta}} 
\|(\Phi_z f)\lambda\|_{\Y_{\sigma}}  |dz_1| \cdots |dz_k|.  
\\
\le \prod_{j=1}^k \delta_j^{-1} \Psi \big(2^{t_1} [\mu^{\theta_1}]_{A_{s_1}}, 2^{t_2} [\lambda^{\theta_2}]_{A_{s_2}}, 
2^{t_3}[\mu^{-\theta_3}]_{A_{s_3}}, 2^{t_4} [\lambda^{-\theta_4}]_{A_{s_4}}\big) \|f \mu\|_{\X_{\sigma}}.  
\end{multline*}
Therefore, the desired estimate follows at once this and \eqref{eq:del}. 
\end{proof}

Recall that a family $\S$ of cubes is called {\tt sparse} if for every cube $Q \in \S$, there exists $E_Q \subset Q$ such that $|E_Q | \geq \eta |Q|$ for some $0<\eta<1$ and the collection $\{E_Q\}_{Q \in \S}$ is pairwise disjoint.  Given a sparse family $\S$ and $\gamma \ge 1$, we define a sparse operator as 
\begin{align*}
\A_{\S}^{\gamma} (f)(x) := \bigg(\sum_{Q \in \S} 
\langle f \rangle_{Q}^{\gamma} \mathbf{1}_{Q}(x) \bigg)^{\frac{1}{\gamma}},\quad x \in \Rn 
\end{align*}
where $\langle f \rangle_{Q}=\frac{1}{|Q|}\int_{Q} f \, dx$. If $\gamma=1$, we denote $\A_{\S}=\A^{\gamma}_{\S}$.

\begin{corollary}\label{cor:TTb}
If for every $f, g \in C_c^{\infty}(\Rn)$ and for every $b \in L^1_{\loc}(\Rn)$, 
\begin{align}\label{eq:CbT} 
|\langle C_b(T)f, g\rangle| \lesssim \sup_{\S \text{ is sparse}} \big(\T_{\S}(b, f, g) +\T^{*}_{\S}(b, f, g) \big), 
\end{align}
where 
\begin{align*}
\T_{\S}(b, f, g) &:= \sum\limits_{Q \in \S} \langle |f| \rangle_{Q} \langle|(b-b_Q)g|\rangle_{Q} |Q|, 
\\
\T^{*}_{\S}(b, f, g)&: = \sum\limits_{Q \in \S} \langle |(b-b_Q)f| \rangle_{Q} \langle |g| \rangle_{Q} |Q|,
\end{align*} 
then for every $p \in (1, \infty)$, for every $\mu^p,\lambda^p \in A_p$ with $\nu=\mu/\lambda$ and for every $\b=(b_1, b_2, \ldots,b_k) \in \BMO(\nu) \times \BMO \times \cdots \times \BMO$, 
\begin{align}\label{eq:Bloom}
\|C_{\b}(T)\|_{L^p(\mu^p) \to L^p(\lambda^p)} 
&\lesssim \|b_1\|_{\BMO(\nu)} \prod_{j=2}^k\|b_j\|_{\BMO}  
\big([\mu^p]_{A_p} [\lambda^p]_{A_p} \big)^{\max\{1, \frac{1}{p-1}\}} 
\nonumber \\
&\qquad\qquad \times \max\big\{ [\mu^p]_{A_p}, [\lambda^p]_{A_p} \big\}^{(k-1)\max\{1, \frac{1}{p-1}\}}.  
\end{align}
\end{corollary}

\begin{proof}
Fix $p \in (1, \infty)$ and $\mu^p,\lambda^p \in A_p$ with $\nu=\mu/\lambda$. Let $\S$ be a sparse family, $b:=b_1\in \BMO(\nu)$ and $b_j \in \BMO$, $j=2,\ldots,k$. It is well-known that 
\begin{align}\label{eq:AS-sharp}
\|\A_{\S}\|_{L^p(w) \to L^p(w)} \le c_{n,p} [w]_{A_p}^{\max\{1, \frac{1}{p-1}\}},\quad\forall p \in (1, \infty). 
\end{align}
Recently, Lerner et al. \cite[Lemma~5.1]{LOR} proved that for a sparse family $\S$, there exists another sparse family $\widetilde{\S} \subset \D$ containing $\S$ and such that for every $Q \in \widetilde{S}$, 
\begin{align*}
|b(x)-b_Q| \leq 2^{n+2} \sum_{Q' \in \widetilde{\S}: Q' \subset Q} 
\langle |b-b_{Q'}| \rangle_{Q'} \mathbf{1}_{Q'}(x),\quad\text{a.e. } x \in Q. 
\end{align*}
This immediately gives that 
\begin{align}\label{eq:bQAS}
\langle |(b-b_Q)g|\rangle_Q \lesssim \|b\|_{\BMO(\nu)} \langle \A_{\widetilde{\S}}(|g|) \nu \rangle_Q. 
\end{align}
Thus, using \eqref{eq:bQAS}, H\"{o}lder's inequality, \eqref{eq:M-sharp} and \eqref{eq:AS-sharp}, we deduce that 
\begin{align*}
\T_{\S}(b, f, g) &\lesssim \|b\|_{\BMO(\nu)} \sum_{Q \in \S} 
\langle |f| \rangle_{Q} \langle \A_{\widetilde{S}}(|g|) \nu \rangle_{Q}  |Q|
\\
&\lesssim \|b\|_{\BMO(\nu)} \sum_{Q \in \S} 
\Big(\inf_{Q} Mf \Big) \Big(\inf_{Q}M(\A_{\widetilde{\S}}(|g|) \nu)\Big) |E_Q| 
\\
&\le \|b\|_{\BMO(\nu)} \|Mf \cdot M(\A_{\widetilde{\S}}(|g|) \nu)\|_{L^1(\Rn)}
\\
&\le \|b\|_{\BMO(\nu)} \|Mf\|_{L^p(\mu^p)} \|M(\A_{\widetilde{\S}}(|g|) \nu)\|_{L^{p'}(\mu^{-p'})} 
\\
&\lesssim \|b\|_{\BMO(\nu)} [\mu^p]_{A_p}^{\frac{1}{p-1}} [\mu^{-p'}]_{A_{p'}}^{\frac{1}{p'-1}} 
\|f\|_{L^p(\mu^p)} \|\A_{\widetilde{\S}}(|g|)\|_{L^{p'}(\lambda^{-p'})} 
\\
&\lesssim \|b\|_{\BMO(\nu)} \big([\mu^p]_{A_p} [\lambda^p]_{A_p}\big)^{\frac{1}{p-1}} 
\big([\mu^{-p'}]_{A_{p'}} [\lambda^{-p'}]_{A_{p'}}\big)^{\frac{1}{p'-1}} \|f\|_{L^p(\mu^p)} \|g\|_{L^{p'}(\lambda^{-p'})}.  
\end{align*}
Similarly, one has 
\begin{align*}
\T^*_{\S}(b, f, g) &\lesssim \|b\|_{\BMO(\nu)} \|M(\A_{\widetilde{\S}}(|f|) \nu) \cdot Mg\|_{L^1(\Rn)}
\\
&\lesssim \|b\|_{\BMO(\nu)} [\lambda^p]_{A_p}^{\frac{1}{p-1}} 
[\mu^p]_{A_p}^{\max\{1, \frac{1}{p-1}\}} [\lambda^{-p'}]_{A_{p'}}^{\frac{1}{p'-1}} 
\|f\|_{L^p(\mu^p)} \|g\|_{L^{p'}(\lambda^{-p'})} 
\\
&\lesssim \|b\|_{\BMO(\nu)} \big([\mu^p]_{A_p} [\lambda^p]_{A_p}\big)^{\frac{1}{p-1}}   
\big([\mu^{-p'}]_{A_{p'}} [\lambda^{-p'}]_{A_{p'}}\big)^{\frac{1}{p'-1}} \|f\|_{L^p(\mu^p)} \|g\|_{L^{p'}(\lambda^{-p'})}.  
\end{align*}
Collecting the above estimates and \eqref{eq:CbT}, we conclude that 
\begin{align}
\|C_b(T)\|_{L^p(\mu^p) \to L^p(\lambda^p)} 
\lesssim \|b\|_{\BMO(\nu)} \big([\mu^p]_{A_p} [\lambda^p]_{A_p} \big)^{\frac{1}{p-1}} 
\big([\mu^{-p'}]_{A_{p'}} [\lambda^{-p'}]_{A_{p'}}\big)^{\frac{1}{p'-1}}. 
\end{align}

Consequently, \eqref{eq:Bloom} follows at once from \eqref{eq:BMO} and Theorem \ref{thm:TTb} applied to the case $\X_{\sigma}=\Y_{\sigma}=L^p(\Rn)$, $s_1=s_2=\theta_1=\theta_2=p$, $s_3=s_4=\theta_3=\theta_4=p'$, and 
\begin{align*}
\Psi(t_1, t_2, t_3, t_4) = \|b\|_{\BMO(\nu)} (t_1 t_2)^{\frac{1}{p-1}} (t_3 t_4)^{\frac{1}{p'-1}}.  
\end{align*}
The proof is complete. 
\end{proof}

\begin{theorem}\label{thm:TaTb}
For every $p,q \in (1, \infty)$, for every $\mu, \lambda \in A_p$, for every $\b=(b_1, b_2, \ldots,b_k) \in \BMO(\nu) \times \BMO \times \cdots \times \BMO$ with $\nu=(\mu/\lambda)^{1/p}$,  
\begin{align*}
\|\mathbb{T}\|_{L^p(\mu) \to L^p(\lambda)} 
&\lesssim \|b_1\|_{\BMO(\nu)} \prod_{j=2}^k\|b_j\|_{\BMO}  
\big([\mu]_{A_p} [\lambda]_{A_p} \big)^{\max\{1, \frac{1}{p-1}\}} 
\nonumber \\
&\qquad\qquad \times \max\big\{ [\mu]_{A_p}, [\lambda]_{A_p} \big\}^{(k-1)\max\{1, \frac{1}{p-1}\}},  
\end{align*}
for every operator $\mathbb{T} \in \big\{C_{\b}(T),\, |C_{\b}(T)|_q,\, \V_{\rho} \circ C_{\b}(\T),\, C_{\b}(T_{\sigma})\big\}$, where $T \in \w\text{-CZO}$ with $\w \in {\rm Dini}$. 
\end{theorem}

The definitions of operators above can be found in Section \ref{sec:app}. Let us see how Corollary \ref{cor:TTb} implies Theorem \ref{thm:TaTb}. Indeed, it is enough to show that each operator $\mathbb{T}$ above verifies \eqref{eq:CbT}. This will follow from the point sparse domination obtained in \cite[Theorem~1.4]{LOR}, \cite[Theorem~1.11]{CLPR}, \cite[Proposition~3.3]{WWZ} and \cite[Proposition~4.1]{CXY} respectively.  

Next, we apply Theorem \ref{thm:TTb} to establish Bloom type inequality for the fractional integral operator $I_{\alpha}$ ($0<\alpha<n$) given by 
\begin{align*}
I_{\alpha}f(x) := \int_{\Rn} \frac{f(y)}{|x-y|^{n-\alpha}} dy,\quad x \in \Rn.
\end{align*}
Let us see the first order commutator. It was proved in \cite{AMR} that 
\begin{align*}
\|C_b(I_{\alpha})\|_{L^p(\mu^p) \to L^q(\lambda^q)} \lesssim \Xi(\mu, \lambda) \|b\|_{\BMO(\nu)}, 
\end{align*}
where 
\begin{align*}
\Xi(\mu, \lambda) 
= [\mu]_{A_{p,q}}^{(1-\frac{\alpha}{n}) \max\{1, \frac{p'}{q}\}} [\lambda^q]_{A_q}^{\max\{1, \frac{1}{q-1}\}}
+ [\lambda]_{A_{p,q}}^{(1-\frac{\alpha}{n}) \max\{1, \frac{p'}{q}\}} [\mu^p]_{A_p}^{\max\{1, \frac{1}{p-1}\}}. 
\end{align*}
In view of \eqref{eq:wApq} and \eqref{eq:ppqq}, we rewrite $\Xi(\mu, \lambda)$ as 
\begin{align*}
\Xi(\mu, \lambda) = \Psi \Big([\mu^p]_{A_p}, [\lambda^q]_{A_{1+\frac{q}{p'}}}, 
[\mu^{-p'}]_{A_{1+\frac{p'}{q}}}, \lambda^{-q'}]_{A_{q'}}\Big), 
\end{align*}
where 
\begin{align}\label{eq:tttt}
\Psi(t_1, t_2, t_3, t_4) := t_1^{\max\{1, \frac{1}{p-1}\}} t_2^{(1-\frac{\alpha}{n}) \max\{1, \frac{p'}{q}\}} 
+ t_3^{(1-\frac{\alpha}{n}) \max\{1, \frac{q}{p'}\}} t_4^{\max\{1, \frac{1}{q'-1}\}}. 
\end{align}
Therefore, the result below immediately follows from Theorem \ref{thm:TTb} applied to the case $\X_{\sigma}=L^p(\Rn)$, $\Y_{\sigma}=L^q(\Rn)$, $(s_1, s_2, s_3, s_4)=(p, 1+\frac{q}{p'}, 1+\frac{p'}{q}, q')$, $(\theta_1, \theta_2, \theta_3, \theta_4)=(p, q, p', q')$, and the function $\Psi$ defined in \eqref{eq:tttt}.  

\begin{theorem}\label{thm:CbIa}
Let $0<\alpha<n$ and $1<p<q<\infty$ with $\frac{1}{p}-\frac{1}{q}=\frac{\alpha}{n}$. Then for every $\mu, \lambda \in A_{p,q}$ and for every $\b=(b_1, b_2, \ldots,b_k) \in \BMO(\nu) \times \BMO \times \cdots \times \BMO$ with $\nu=\mu/\lambda$, 
\begin{multline*}
\|C_{\b}(I_{\alpha})\|_{L^p(\mu^p) \to L^q(\lambda^q)} 
\lesssim \|b_1\|_{\BMO(\nu)} \prod_{j=2}^k\|b_j\|_{\BMO} \,  \Xi(\mu, \lambda) 
\\
\times \max\Big\{[\mu]_{A_{p, q}}^{\max\{1, \frac{p'}{q}\}}, 
[\mu^p]_{A_p}^{\max\{1, \frac{1}{p-1}\}}, 
[\lambda]_{A_{p, q}}^{\max\{1, \frac{p'}{q}\}}, 
[\lambda^{q}]_{A_q}^{\max\{1, \frac{1}{q-1}\}} \Big\}^{k-1}. 
\end{multline*}
\end{theorem}

\section{Applications}\label{sec:app}
The purpose of this section is to present some applications of the extrapolation theorems established above. We will see that by means of extrapolation, many known results can be extended to the general Banach function spaces. Before doing that, we recall the definitions and notation of some operators. 

Let $\w: [0, 1] \to [0, \infty)$ be a modulus of continuity, that is, $\w$ is increasing, subadditive and $\w(0)=0$. We say that $T$ is an {\tt $\w$-Calder\'{o}n-Zygmund operator} (or simply, $T \in \w\text{-CZO}$), if $T$ is $L^2$ bounded and represented as 
\begin{align*}
Tf(x) = \int_{\Rn} K(x, y) f(y) dx, \quad\forall x \not\in \supp(f), 
\end{align*}
where the kernel $K:\Rn \times \Rn \setminus \{(x, y) \in \Rn \times \Rn: x=y\} \to \C$ is a function satisfying the following conditions: 
\begin{itemize}
\item Size condition: $\displaystyle{ |K(x, y)| \lesssim \frac{1}{|x-y|^n}, \quad\forall x \neq y}$,    
\item Smoothness conditions:
\begin{align*}
|K(x, y)-K(x',y)| &\lesssim \w \bigg(\frac{|x-x'|}{|x-y|}\bigg) \frac{1}{|x-y|^n}, 
\quad\text{whenever } |x-x'| \le \frac12 |x-y|, 
\\
|K(x, y)-K(x,y')| &\lesssim \w \bigg(\frac{|y-y'|}{|x-y|}\bigg) \frac{1}{|x-y|^n}, 
\quad\text{whenever } |y-y'| \le \frac12 |x-y|. 
\end{align*} 
\end{itemize}
Throughout this section, whenever $T \in \w\text{-CZO}$, we always assume that $\w$ satisfies the Dini condition (or, $\w \in \text{Dini}$), which means that $\|\w\|_{{\rm Dini}}:=\int_{0}^{1}\w(t) \frac{dt}{t}<\infty$. An example of Dini condition is $\w(t)=t^{\alpha}$ with $\alpha>0$. In this case, we say that $T$ is a {\tt standard Calder\'{o}n-Zygmund operator}, or $T \in \alpha\text{-CZO}$. 

Given an $\w$-Calder\'{o}n-Zygmund operator $T$, we define its truncated singular integral by 
\begin{align*}
T_{\epsilon}f(x) := \int_{|x-y|>\epsilon} K(x, y) f(y) dy, 
\end{align*}
Then for $\rho>2$, the {\tt $\rho$-variation operator} for the families of operators $\T:=\{T_{\epsilon}\}_{\epsilon>0}$ and $\T_b:=C_b(\T)=\{T_{b, \epsilon}=C_b(T_{\epsilon})\}_{\epsilon>0}$ are defined as 
\begin{align}
\label{eq:VT} \V_{\rho}(\T f)(x) &:=\sup_{\{\epsilon_j\}\downarrow 0} \bigg(\sum_{j=1}^{\infty} 
|T_{\epsilon_{j+1}}f(x) - T_{\epsilon_j}f(x)|^{\rho} \bigg)^{\frac{1}{\rho}}, 
\\
\label{eq:VTb} \V_{\rho}(\T_b f)(x) &:=\sup_{\{\epsilon_j\}\downarrow 0} \bigg(\sum_{j=1}^{\infty} 
|T_{b, \epsilon_{j+1}}f(x) - T_{b, \epsilon_j}f(x)|^{\rho} \bigg)^{\frac{1}{\rho}}, 
\end{align}
where the supremum is taken over all sequences $\{\epsilon_j\}$ decreasing to zero. 

Given a symbol $\sigma$, the {\tt pseudo-differential operator} $T_{\sigma}$ is defined by
\begin{align}\label{eq:def-Ta}
T_{\sigma}f(x) = \int_{\Rn} \sigma(x,\xi) e^{2\pi i x \cdot \xi} \widehat{f}(\xi) d\xi, 
\end{align}
where the Fourier transform $\widehat{f}$ of the function $f$. Given $m \in \R$ and $\varrho,\delta \in [0, 1]$, a smooth function $\sigma$ on $\Rn \times \Rn$ belongs to {\tt H\"{o}rmander class} $S_{\varrho,\delta}^m$ if for each triple of multi-indices $\alpha$ and $\beta$ there exists a constant $C_{\alpha,\beta}$ such that 
\begin{equation*}
\big| \partial_{x}^{\alpha} \partial_{\xi}^{\beta} \sigma(x, \xi) \big| 
\leq C_{\alpha,\beta} (1+|\xi|)^{m-\rho |\beta| + \delta|\alpha|}.
\end{equation*}

Let us introduce two kinds of more singular operators. Recall that the {\tt rough singular integral} $T_{\Omega}$ is defined by 
\begin{equation}\label{def:TO}
T_{\Omega}f(x):={\rm p.v. } \int_{\Rn} \frac{\Omega(y')}{|y|^n} f(x-y) dy, 
\end{equation}
where $\Omega \in L^{\infty}(\Sn)$ and $\int_{\Sn}\Omega(\xi) d\sigma(\xi)=0$. On the other hand, the {\tt Bochner-Riesz multiplier} is defined by 
\begin{align}\label{def:BR}
\widehat{B_{\delta}f}(\xi) := (1-|\xi|^2)_{+}^{\delta} \widehat{f}(\xi).   
\end{align}

Next, we introduce a class of square functions. A function $K$ defined on $\Rn \times \Rn$ is said to be in ${\rm LP}$ if there exist $\beta>0$ and $\gamma>0$ such that the following conditions holds:
\begin{itemize}
\item Size condition:
$\displaystyle{|K(x, y)| \lesssim \frac{1}{(1+|x-y|)^{n+\beta}}, \quad\forall x,y \in \Rn}$, 
\item Smoothness conditions:
\begin{align*}
|K(x, y)-K(x',y)| &\lesssim \frac{|x-x'|^{\gamma}}{(1+|x-y|)^{n+\beta+\gamma}}, 
\quad\text{whenever } |x-x'| \le \frac12 |x-y|, 
\\
|K(x, y)-K(x,y')| &\lesssim \frac{|y-y'|^{\gamma}}{(1+|x-y|)^{n+\beta+\gamma}}, 
\quad\text{whenever } |y-y'| \le \frac12 |x-y|. 
\end{align*} 
\end{itemize}
Given a function $K \in {\rm LP}$, we always denote $\displaystyle{K_t f(x) := \frac{1}{t^n} \int_{\Rn} K \Big(\frac{x}{t}, \frac{y}{t}\Big) f(y) dy}$. Then for $\alpha \ge 0$ and $\lambda>2$, we define the square functions as 
\begin{align}
\label{def:S} S_{\alpha}(f)(x) &:= \bigg(\iint_{\Gamma_{\alpha}(x)} |K_t f(x)|^2 \frac{dydt}{t^{n+1}}\bigg)^{\frac12}, 
\\
\label{def:g} g_{\lambda}^{*}(f)(x) &:= \bigg(\iint_{\R^{n+1}_{+}} \Big(\frac{t}{t+|x-y|}\Big)^{n\lambda} |K_t f(y)|^2 \frac{dydt}{t^{n+1}}\bigg)^{\frac12}, 
\end{align}
where $\Gamma_{\alpha}(x)=\{(y, t) \in \R^{n+1}_{+}: |x-y|<\alpha t\}$.

\subsection{Local decay estimates}\label{sec:local} 
In this subsection, let us see how to use Theorem \ref{thm:Aq} to establish local decay estimates: 
\begin{align*}
\psi_t(\mathfrak{T}, \mathfrak{M}) 
:= \sup_{Q:\text{ cube } \subset \Rn} \sup_{\substack{f \in L_c^{\infty}(\Rn) \\ \supp(f) \subset Q}} 
|Q|^{-1} |\{x \in Q: |\mathfrak{T}f(x)| > t\, \mathfrak{M}f(x)\}|, \quad t>0, 
\end{align*} 
where $\mathfrak{T}$ is a singular operator and $\mathfrak{M}$ is an appropriate maximal operator. 

Let $T$ be the $\w$-Calder\'{o}n-Zygmund operator with $\w \in {\rm Dini}$. Recall that the $\rho$-variation operators $\V_{\rho} \circ \T$ and $\V_{\rho} \circ \T_b$ are defined in \eqref{eq:VT} and \eqref{eq:VTb}. Fix a cube $Q$ and $f \in L_c^{\infty}(\Rn)$ with $\supp(f) \subset Q$. With these notation in hand, we have for every $1\le q<\infty$ and for every $w \in A_q$, 
\begin{align}
\label{eq:TL1} \|Tf\|_{L^1(Q, w)} &\leq c_{n,q} [w]_{A_q} \|Mf\|_{L^1(Q, w)}, 
\\
\label{eq:VTL1} \|\V_{\rho}(\T f)\|_{L^1(Q, w)} &\leq c_{n,q} [w]_{A_q} \|Mf\|_{L^1(Q, w)}, 
\\
\label{eq:TaL1} \|T_{\sigma}f\|_{L^1(Q, w)} &\leq c_{n,q} [w]_{A_q} \|Mf\|_{L^1(Q, w)}, 
\\
\label{eq:SL1} \|S_{\alpha}f\|_{L^2(Q, w)} &\leq c_{n,q} [w]_{A_q}^{\frac12} \|Mf\|_{L^2(Q, w)}, 
\\
\label{eq:gL1} \|g^{*}_{\lambda}f\|_{L^2(Q, w)} &\leq c_{n,q} [w]_{A_q}^{\frac12} \|Mf\|_{L^2(Q, w)}, 
\\
\label{eq:TOL1} \|T_{\Omega}f\|_{L^1(Q, w)} &\leq c_{n,q} [w]_{A_q}^2 \norm{Mf}_{L^1(Q, w)}, 
\\
\label{eq:Bn} \|B_{(n-1)/2}f\|_{L^1(Q, w)} &\leq c_{n,q} [w]_{A_q}^2 \norm{Mf}_{L^1(Q, w)}, 
\\
\label{eq:VTb} \|\V_{\rho}(\T_b f)\|_{L^1(Q, w)} &\leq c_{n,q} \|b\|_{\BMO} [w]_{A_q}^2 \|M^2f\|_{L^1(Q, w)}, 
\\
\label{eq:CbTa} \|C_b(T_{\sigma})f\|_{L^1(Q, w)} &\leq c_{n,q} \|b\|_{\BMO} [w]_{A_q}^2 \|M^2f\|_{L^1(Q, w)}, 
\\
\label{eq:CbTO} \|C_b(T_{\Omega})f\|_{L^1(Q, w)} &\leq c_{n,q} \|b\|_{\BMO} [w]_{A_q}^3 \|M^2f\|_{L^1(Q, w)}, 
\\
\label{eq:TbkL1} \|C_b^k(T) f\|_{L^1(Q, w)} &\leq c_{n,q} \|b\|_{\BMO}^k [w]_{A_q}^{k+1} \|M^{k+1}f\|_{L^1(Q, w)}, 
\end{align}
where $c_{n,q}$ is independent of $Q$, $f$ and $w$. Indeed, using the techniques in \cite{CXY, OPR}, one can show \eqref{eq:TL1}--\eqref{eq:gL1}. But it needs the sharp maximal function control for these operators. This strategy is invalid for $T_{\Omega}$ and $B_{(n-1)/2}$. To circumvent this problem, we present a local version of sparse domination for all the operators above. We only give the proof of \eqref{eq:CbTO} since the other proofs are similar and simpler.  

We only consider the case $w \in A_q$ with $1<q<\infty$. Recall that $f \in L_c^{\infty}(\Rn)$ with $\supp(f) \subset Q$. Modifying the proof of \cite[Theorem~3.1]{Riv}, we obtain that for every $g \in L_c^{\infty}(\Rn)$, for every $b \in L_{\loc}^1(\Rn)$ and for every $s \in (1,\infty)$, there exists a sparse family $\S(Q) \subset \D(Q)$ such that 
\begin{align}\label{eq:bTO} 
|\langle C_b(T_{\Omega})f, g \mathbf{1}_Q\rangle| 
\lesssim s' \big(\Lambda_{Q, s}(b, f, g) +\Lambda^{*}_{Q, s}(b, f, g) \big), 
\end{align}
where 
\begin{align*}
\Lambda_{Q, s}(b, f, g) &:= \sum_{Q' \in \S(Q)} \langle |f| \rangle_{3Q'} \langle |(b-b_{3Q'})g|\rangle_{Q', s} |Q'|,  
\\
\Lambda^*_{Q, s}(b, f, g) &:= \sum_{Q' \in \S(Q)} \langle |(b-b_{3Q'})f| \rangle_{3Q'} \langle |g| \rangle_{Q', s} |Q'|.
\end{align*}  
By $w\in A_q$ and \eqref{eq:RHp}, one has $w \in RH_{r_w}$, where $r_w=1+\frac{1}{2^{n+1+2q}[w]_{A_q}}$. If we set 
$r=1+\frac{1}{2^{n+2+2q}[w]_{A_q}}$ and $s=\frac{r_w}{r}$, then $1<s<2$ and $s' \simeq [w]_{A_q} \simeq r'$. A useful inequality from \cite[Corollary~3.1.8]{G2} is that $\|b-b_{Q'}\|_{p, Q'} \le c_n p \|b\|_{\BMO}$ for all $p>1$. Using these estimates, we have 
\begin{align}\label{eq:TS-Mf}
\Lambda_{Q, s}(b, f, gw) &= \sum_{Q' \in \S(Q)} \langle |f| \rangle_{3Q'} \langle |(b-b_{3Q'})g w|\rangle_{Q',s} |Q'| 
\nonumber\\
&\le \sum_{Q' \in \S(Q)} \langle |f| \rangle_{3Q'} \langle |b-b_{3Q'}| \rangle_{Q', sr'} \langle w\rangle_{Q', sr} \|g\|_{L^{\infty}(\Rn)} |Q'| 
\nonumber\\
&\lesssim \sum_{Q' \in \S(Q)} \langle |f| \rangle_{3Q'} \langle |b-b_{3Q'}| \rangle_{3Q', sr'} \langle w\rangle_{Q', sr} \|g\|_{L^{\infty}(\Rn)} |Q'| 
\nonumber\\
&\lesssim sr' \|b\|_{\BMO} \|g\|_{L^{\infty}(\Rn)} \sum_{Q' \in \S(Q)} \langle |f| \rangle_{3Q'} w(Q') 
\nonumber\\
&\lesssim [w]_{A_q} \|b\|_{\BMO} \|g\|_{L^{\infty}(\Rn)} \sum_{Q' \in \S(Q)} 
\bigg(\fint_{Q'}(Mf)^{\frac12} \mathbf{1}_Q \, dw\bigg)^2 w(Q') 
\nonumber\\
&\lesssim [w]_{A_q}^2 \|b\|_{\BMO} \|Mf\|_{L^1(Q, w)} \|g\|_{L^{\infty}(\Rn)},   
\end{align}
where we used the Carleson embedding theorem from \cite[Theorem~4.5]{HP} and that the collection $\{w(Q')\}_{Q' \in \S(Q)}$ satisfies the Carleson packing condition with the constant $c_n [w]_{A_q}$. Likewise, one has 
\begin{align}\label{eq:TS-ML}
\Lambda^*_{Q, s}(b, f, gw) &= \sum_{Q' \in \S(Q)} \langle |(b-b_{3Q'})f| \rangle_{3Q'} \langle |gw| \rangle_{Q',s} |Q| 
\nonumber\\
&\le \sum_{Q' \in \S(Q)} \|b-b_{3Q'}\|_{\exp L, 3Q'} \|f\|_{L(\log L), 3Q'}  \langle w\rangle_{s} \|g\|_{L^{\infty}(\Rn)} |Q'| 
\nonumber\\
&\lesssim \|b\|_{\BMO} \|g\|_{L^{\infty}(\Rn)} \sum_{Q' \in \S(Q)} \|f\|_{L(\log L), 3Q'} w(Q') 
\nonumber\\
&\lesssim [w]_{A_q} \|b\|_{\BMO} \|M_{L(\log L)}f\|_{L^1(Q, w)} \|g\|_{L^{\infty}(\Rn)}.  
\end{align}
Hence, \eqref{eq:CbTO} follows at once from \eqref{eq:bTO}, \eqref{eq:TS-Mf}, \eqref{eq:TS-ML} and that $s' \simeq [w]_{A_q}$. 

Let us turn our attention to the local decay estimates. 
\begin{theorem}\label{thm:local}
Let $A(t)=t$ or $A$ be a Young function and let $E \subset \Rn$ be a measurable set. Suppose that $\X$ is a RIBFS over $(\Rn, dx)$ with $q_{\X}<\infty$. If for some $q \in (2, \infty)$ and for every $w \in A_q$, 
\begin{align}\label{eq:local-1}
\|Tf\|_{L^1(E, w)} \leq \Psi([w]_{A_q}) \|M_A f\|_{L^1(E, w)}, 
\end{align}
where $\Psi:[1, \infty) \to [1, \infty)$ is an increasing function, then for every $v \in A_1$, 
\begin{align}\label{eq:local-2}
\bigg\| \frac{Tf}{M_Af} \mathbf{1}_E \bigg\|_{\X_v} 
\leq 2 \Psi \big(c_{n,q} \|M_v\|_{\X'_v \to \X'_v} [v]_{A_1}\big) \|\mathbf{1}_E\|_{\X_v}. 
\end{align}
In particular, for every $p \in (1, \infty)$, 
\begin{align}\label{eq:local-3}
\sup_{t>0} \, t\, |\{x \in E: |Tf(x)| > t M_Af(x)\}|^{\frac1p} \le 2 \Psi(C_{n,q} \, p) |E|^{\frac1p}. 
\end{align}
\end{theorem}

\begin{proof}
Fix $q>2$. Let $v \in A_1$ and $E \subset \Rn$ be a measurable set. Set $u:=(M_A f)^{-1}$. By \eqref{eq:CR} and \eqref{eq:CR-Phi}, one has $[u^{1-q'}]_{A_1}=[(M_A f)^{\frac{1}{q-1}}]_{A_1} \le C_{n, q}$. Then it follows from \eqref{eq:local-1} that for every $w \in A_q$, 
\begin{align*}
\|Tf\cdot \mathbf{1}_E\|_{L^1(w)} 
= \|Tf\|_{L^1(E, w)} 
\leq \Psi([w]_{A_q}) \|M_A f\|_{L^1(E, w)} 
= \|M_A f \cdot \mathbf{1}_E\|_{L^1(w)}. 
\end{align*}
This verifies \eqref{eq:Aq-1}. Hence, Theorem \ref{thm:Aq} applied to the pair $(Tf\cdot \mathbf{1}_E, M_Af\cdot \mathbf{1}_E)$ gives \eqref{eq:local-2} as desired. Furthermore, observe that for every $p \in (1, \infty)$, $\|M\|_{L^{p', \infty}(\Rn) \to L^{p', \infty}(\Rn)} \lesssim p$ (see \cite[Exercise~2.1.13]{G1} or \eqref{eq:weakM} below). Therefore, \eqref{eq:local-3} is a consequence of \eqref{eq:local-2} for the case $v \equiv 1$ and $\X_v=L^{p, \infty}(\Rn)$. 
\end{proof}

To proceed, fix $p \in (1, \infty)$ chosen later and $t>0$. In view of \eqref{eq:TbkL1},  the hypothesis \eqref{eq:local-1} with $\Psi(t)=c_{n} \|b\|_{\BMO}^k t^{k+1}$ is verified. Thus, the inequality \eqref{eq:local-3} gives that  
\begin{align}\label{eq:TbkM}
\psi_t(C_b^k(T), M^{k+1})  
\leq 2 \big(c_{n,k} \|b\|_{\BMO}^k t^{-1} p^{k+1} \big)^p. 
\end{align} 
If $t>t_0:=c_{n,k} e\|b\|_{\BMO}^k$, pick $p \in (1, \infty)$ such that $t=c_{n,k} e\|b\|_{\BMO}^k p^{k+1}$. If we denote $\alpha=(c_{n,k}e)^{-1}$, then it follows from \eqref{eq:TbkM} that 
\begin{align}\label{eq:TbkM-1}
\psi_t(C_b^k(T), M^{k+1})  \le 2 e^{-p} = 2e^{-(\alpha t/\|b\|_{\BMO}^k)^{\frac{1}{k+1}}}.  
\end{align} 
If $0<t<t_0$, then $p \in (1, \infty)$ can be chosen as an arbitrary number and by definition, 
\begin{align}\label{eq:TbkM-2}
\psi_t(C_b^k(T), M^{k+1})  \le 1=e \cdot e^{-(\alpha t_0/\|b\|_{\BMO}^k)^{\frac{1}{k+1}}}
\le e \cdot e^{-(\alpha t/\|b\|_{\BMO}^k)^{\frac{1}{k+1}}}.  
\end{align}  
Summing \eqref{eq:TbkM-1} and \eqref{eq:TbkM-2} up, we obtain 
\begin{align*}
\psi_t(C_b^k(T), M^{k+1}) \lesssim e^{-(\alpha t/\|b\|_{\BMO}^k)^{\frac{1}{k+1}}},\quad\forall k \ge 1.  
\end{align*}  
A similar argument yields that 
\begin{align*}
&\psi_t(T, M) \lesssim e^{-\alpha t},\qquad\quad  
\psi_t(S_{\alpha}, M) \lesssim e^{-\alpha t^2}, \quad 
\psi_t(B_{(n-1)/2}, M) \lesssim e^{-\alpha t^{\frac12}}, 
\\
&\psi_t(\V_{\rho}\circ T, M) \lesssim e^{-\alpha t}, \quad 
\psi_t(g_{\lambda}^{*}, M) \lesssim e^{-\alpha t^2}, \quad \, 
\psi_t(\V_{\rho}\circ \T_b, M^2) \lesssim e^{(-\alpha t/\|b\|_{\BMO})^{\frac12}}, 
\\
&\psi_t(T_{\sigma}, M) \lesssim e^{-\alpha t}, \qquad \,  \, \, 
\psi_t(T_{\Omega}, M) \lesssim e^{-\alpha t^{\frac12}},\quad 
\psi_t(C_b(T_{\sigma}), M^2) \lesssim e^{(-\alpha t/\|b\|_{\BMO})^{\frac12}}, 
\\
&\psi_t(C_b(T_{\Omega}), M^2) \lesssim e^{(-\alpha t/\|b\|_{\BMO})^{\frac13}}. 
\end{align*}

\subsection{Coifman-Fefferman inequalities}\label{sec:CF}
To simplify notation, we set 
\begin{align}\label{eq:T1T2} 
\mathbb{T}_1 \in\{T, \V_{\rho}\circ \T, T_{\sigma}, T_{\Omega}, B_{(n-1)/2}\}, \quad
\mathbb{T}_2 \in\{C_b(T), \V_{\rho}\circ \T_b, C_b(T_{\sigma}), C_b(T_{\Omega})\}, 
\end{align}
where $T$ is an $\w$-Calder\'{o}n-Zygmund operator with $\w \in \text{Dini}$.

\begin{theorem}\label{thm:TTT}
Let $\Phi(t)=t$ or $\Phi$ be a Young function with $1<i_{\Phi} \le I_{\Phi}<\infty$. Suppose that $\X$ is a RIBFS over $(\Rn, dx)$ such that $q_{\X}<\infty$. Then for every $u \in RH_{\infty}$ and for every $v \in A_{\infty}$,  
\begin{align}\label{eq:TTT} 
\|\Phi(\mathbb{T}_i f \cdot u)\|_{\X_v} \lesssim \|\Phi(M^i f \cdot u)\|_{\X_v}, \quad i=1,2. 
\end{align} 
\end{theorem}

\begin{proof}
By Corollary \ref{cor:BFSMod}, the estimate \eqref{eq:TTT} follows from the following 
\begin{align}\label{eq:TMAi} 
\|\mathbb{T}_i f\|_{L^p(w)} \lesssim \|M^if\|_{L^p(w)}, \quad\forall p \in (0, \infty),\, \forall w \in A_{\infty}, \quad i=1,2, 
\end{align}
where the implicit constants only depend on $n$, $p$ and $[w]_{A_{\infty}}$. Hence, it suffices to verify \eqref{eq:TMAi}. By means of sparse dominations aforementioned for $\mathbb{T}_1$ and $\mathbb{T}_2$, a standard argument will derive \eqref{eq:TMAi} for $p=1$. Then using Theorem {\bf B} in Section \ref{sec:BFS}, we obtain \eqref{eq:TMAi} for all $p \in (0, \infty)$.  
\end{proof}

\begin{remark}
By \eqref{eq:MRH}, we see that $u:=(Mw)^{-\mu/p} \in RH_{\infty}$ for all $1<p<\infty$ and $\mu>0$. Then taking $v \equiv 1$ and $\X=L^p(\Rn)$, we derive the result for the $\alpha$-CZO in \cite[Theorem~1.7]{L10}. If $u \equiv 1$ and $v \in A_{\infty}$,  then Theorem \ref{thm:TTT} applied to $\X=L^p(\Rn)$ and $\X=L^{p,\infty}(\Rn)$ yields the weighted inequalities in \cite[Corollary~1.2]{LPRR}. By the same reason, Theorem \ref{thm:TTT} gives the estimates that coincide with \cite[Theorems~1.9, 1.10]{CLPR} for $\w$-CZO $T$ and its commutator in the case $w \in A_{\infty}$. 
\end{remark}

We are going to present another type of Coifman-Fefferman inequalities. Note that by Coifmann and Rochberg theorem \eqref{eq:CR}, we get $v:=M_r w \in A_1$ for every weight $w$ and every $r>1$. Thus, invoking Theorem \ref{thm:AA} applied to $u=v=M_r w$, Theorem \ref{thm:Mvr} applied to $u=v=M_{A_p} w$, \eqref{eq:MPhiRH} and \eqref{eq:TMAi}, we obtain the weighted estimates below.  
\begin{theorem}
Let $\X$ be a RIBFS over $(\Rn, dx)$ with $q_{\X}<\infty$. Then for every weight $w$ and for every $r>1$,   
\begin{align}\label{eq:TfMw-1} 
\norm{\frac{\mathbb{T}_i f}{M_r w}}_{\X(M_r w)} 
\lesssim \norm{\frac{M^i f}{M_r w}}_{\X(M_rw)}, \quad i=1,2. 
\end{align} 
\end{theorem}

\begin{theorem}
Let $A$ be a Young function and $w$ be a weight. Write $A_p(t)=A(t^{1/p})$ for $0<p<\infty$. Suppose that $\X_v$ be a BFS over $(\Rn, v\,dx)$ with $v=M_{A_p} w$. If there exists $r>1$ such that $M'_{v^{1/r}}$ is bounded on $\X'_v$, then 
\begin{align}\label{eq:TfMw-2} 
\norm{\frac{\mathbb{T}_i f}{v}}_{\X_v} \lesssim \norm{\frac{M^if}{v}}_{\X_v}, \quad i=1,2. 
\end{align} 
\end{theorem}

We here mention that \eqref{eq:TfMw-1} and \eqref{eq:TfMw-2} are key ingredients leading to the sharp $A_1$ inequalities as follows: 
\begin{align}
\label{eq:A1-sharp1} \|Tf\|_{L^p(w)} &\le c_{n,p,T} [w]_{A_1}^{\beta} \|f\|_{L^p(w)}, \quad 1<p<\infty, 
\\
\label{eq:A1-sharp2} \|Tf\|_{L^{1,\infty}(w)} &\le c_{n,T} [w]_{A_1}^{\gamma} \log(e+[w]_{A_{\infty}}) \|f\|_{L^1(w)}. 
\end{align}
Such estimates originated from \cite{LOP1} and were extensively extended to other singular operators and commutators. 

\begin{list}{\textup{(\theenumi)}}{\usecounter{enumi}\leftmargin=1cm \labelwidth=1cm \itemsep=0.2cm 
			\topsep=.2cm \renewcommand{\theenumi}{\arabic{enumi}}}
\item If $\X=L^p(\Rn)$ with $p>1$ and $\mathbb{T}_1$ is a $\alpha$-CZO, \eqref{eq:TfMw-1} was given in \cite[Lemma~2.1]{LOP1}. In this case, the estimate \eqref{eq:A1-sharp1} with $\beta=1$ is sharp with respect to $[w]_{A_1}$. The inequality \eqref{eq:A1-sharp2} holds for $\gamma=1$ as well. 
\item If $\X=L^{p'}(\Rn)$ with $p>1$ and $\mathbb{T}_1$ is a $\alpha$-CZO, \eqref{eq:TfMw-1} for $T^{*}$ was obtained in the proof of \cite[Theorem~1.3]{Ort}. Also, \eqref{eq:A1-sharp1} holds for $C_b(T)$ and $\beta=2$. 
\item If $\X=L^{p'}(\Rn)$ with $p>1$ and $\mathbb{T}_1=T_{\Omega}$, \eqref{eq:TfMw-1} was established in the proof of \cite[Theorem~1.1]{Riv}. A more accurate bound is obtained and leads to \eqref{eq:A1-sharp1} with $\beta=2$.
\item If $\X_v=L^{p'}(v)$ with $p>1$ and $\mathbb{T}_1 \in \{T_{\Omega}, B_{(n-1)/2}\}$, \eqref{eq:TfMw-2} was got in \cite[~p.2546]{LPRR}. A refined endpoint inequality implies \eqref{eq:A1-sharp2} with $\gamma=2$. 
\item The inequality \eqref{eq:TfMw-2} also holds for the sparse dyadic operator $\A_{\S}$. A particular case $\X_v=L^{p'}(v)$ was shown in \cite[Lemma~4.3]{HP}, which can be used to show \eqref{eq:A1-sharp2} with $\gamma=1$ for the maximal singular integral $T_{\#}$. 
\end{list}

\subsection{Muckenhoupt-Wheeden conjecture} 
We will present some estimates concerning \eqref{eq:mwdd}. Recall the operators $\mathbb{T}_1$ and $\mathbb{T}_2$ in \eqref{eq:T1T2}. Thanks to \eqref{eq:TMAi}, Theorem \ref{thm:M3w} applied to $\X=L^{1,\infty}(\Rn)$ and the pair $(\mathbb{T}_1f, f)$ gives that 
\begin{align*}
\norm{\frac{\mathbb{T}_1f}{M^3 w}}_{L^{1, \infty}(Mw)} 
\lesssim \norm{\frac{f}{Mw}}_{L^{1, \infty}(Mw)} 
\le \norm{\frac{f}{Mw}}_{L^1(Mw)} =\|f\|_{L^1(\Rn)}.  
\end{align*}
Likewise, treating $\mathbb{T}_2$, we conclude the following weighted weak-type inequalities. 

\begin{theorem}\label{thm:MW} 
For every weight $w$, we have 
\begin{align}\label{eq:MW-T} 
\norm{\frac{\mathbb{T}_i f}{M^3w}}_{L^{1,\infty}(Mw)}  \lesssim \|M^{i-1}f\|_{L^1(\Rn)}, \quad i=1,2. 
\end{align} 
In particular, for every $w \in A_1$, 
\begin{align}
\norm{\frac{\mathbb{T}_i (fw)}{w}}_{L^{1,\infty}(w)}  \lesssim \|M^{i-1}f\|_{L^1(w)},\quad i=1,2.  
\end{align} 
\end{theorem}

The first estimate in \eqref{eq:MW-T} for the standard Calder\'{o}n-Zygmund operator $T$ was proved in \cite{LOP2}. 
Let us next pay attention to the mixed weak-type estimates: 
\begin{align}\label{eq:MWS}
\norm{\frac{T(fv)}{v}}_{L^{1,\infty}(uv)} \lesssim \|f\|_{L^1(uv)}.  
\end{align}

Now let $\X_{uv}$ be a BFS satisfying the hypotheses of Theorem \ref{thm:Saw}. Considering \eqref{eq:TMAi}, we utilize Theorem \ref{thm:Saw} to obtain 
\begin{align}\label{eq:TiMiX}
\norm{\frac{\mathbb{T}_i (fv)}{v}}_{\X_{uv}}  \lesssim \norm{\frac{M^i (fv)}{v}}_{\X_{uv}}, \quad i=1,2. 
\end{align} 
Let us go back to the particular case $\X_{uv}=L^{1,\infty}(uv)$. Let $u \in A_1$ and $v \in A_{\infty}$. The condition $v \in A_{\infty}$ gives that $v \in A_r$ for some $r>1$, which in turn yields that $v=v_1v_2^{1-r}$ for some $v_1, v_2 \in A_1$. By Lemma \ref{lem:A1Ap} \eqref{eq:A1A1}, there exists $\varepsilon_0=\varepsilon_0([u]_{A_1}) \in (0, 1)$ such that $uv_2^{\varepsilon} \in A_1$ for every $\varepsilon \in (0, \varepsilon_0)$. Then, picking $p_0=2(r-1)/\varepsilon_0+1$, we have $u^{1-p_0} v =v_1(uv_2^{\frac{r-1}{p_0-1}})^{1-p_0} \in A_{p_0}$. By \eqref{eq:M-sharp},  $M$ is bounded on $L^{p_0}(u^{1-p_0}v)$, which is equivalent to that $M'_u$ is bounded on $L^{p_0}(uv)$ with a constant $C_0$. On the other hand, the condition $u \in A_1$ implies that $M'_u$ is bounded on $L^{\infty}(uv)$ with the constant $C_1=[u]_{A_1}$. Thus, the interpolation theorem \cite[Proposition~A.1]{CMP05} gives that 
\begin{align*}
\|M'_u f\|_{L^{q,1}(uv)} \le 2^{1/q} (C_0(1/p_0-1/q)^{-1}+C_1) \|f\|_{L^{q,1}(uv)}, \quad \forall q>p_0.  
\end{align*}
If we set $q_0:=2p_0$, then 
\begin{align}\label{eq:MK}
\|M'_u f\|_{L^{q,1}(uv)} \le 4p_0 (C_0+C_1) \|f\|_{L^{q,1}(uv)} =:K_0 \|f\|_{L^{q,1}(uv)}, \quad \forall q \ge q_0,   
\end{align}
which verifies \eqref{eq:Saw-1}. Furthermore, as aforementioned, $M'_v$ is bounded from $L^1(uv)$ to $L^{1, \infty}(uv)$. Hence, this and \eqref{eq:TiMiX} imply the following. 

\begin{theorem}\label{thm:SawL1} 
For every $u \in A_1$ and $v \in A_{\infty}$ with $uv \in L^1_{\loc}(\Rn)$, 
\begin{align}
\norm{\frac{\mathbb{T}_1 (fv)}{v}}_{L^{1,\infty}(uv)}  \lesssim \|f\|_{L^1(uv)}. 
\end{align} 
\end{theorem}

Next, we consider Sawyer conjecture in the case $\X_{uv}=L^{p,\infty}(uv)$ with $p>1$.  
\begin{theorem}\label{thm:SawLp} 
For every $u \in A_1$ and $v \in A_{\infty}$ with $uv \in L^1_{\loc}(\Rn)$, there exists $\epsilon>0$ small enough such that 
\begin{align}\label{eq:SawLp} 
\norm{\frac{\mathbb{T}_i (fv)}{v}}_{L^{1+\epsilon,\infty}(uv)}  \lesssim \|f\|_{L^{1+\epsilon, \infty}(uv)}, \quad i=1,2.  
\end{align} 
\end{theorem}

\begin{proof}
Let $u \in A_1$ and $v \in A_{\infty}$. We begin with showing that 
\begin{align}\label{eq:weakM}
\|M'_v f\|_{L^{p, \infty}(uv)} \lesssim p' \|f\|_{L^{p, \infty}(uv)}, \quad \forall p \in(1, \infty). 
\end{align} 
For each $N \ge 1$, denote $u_N:=u \mathbf{1}_{B(0, N)}$. Then $[u_N]_{A_1} \le [u]_{A_1}$ for every $N \ge 1$. For any $\lambda>0$, we set 
\begin{equation*}
f_\lambda(x)=\begin{cases}
f(x), &M'_v f(x)>\lambda, \\
0,    &M'_v f(x)\le \lambda
\end{cases}.
\end{equation*}
We claim that 
\begin{equation}\label{eq:MtMt}
\{x\in\Rn: M'_vf(x)>\lambda \} \subset \{x\in\Rn: M'_v(f_\lambda)(x)>\lambda\}. 
\end{equation}
Indeed, fix $x \in \Rn$ such that $M'_vf(x)>\lambda$. Then by definition, there exists a cube $Q_x \ni x$ such that $\fint_{Q_x} |f|v\, dy>\lambda v(x)$. This implies $Q_x \subset \{x\in\Rn: M'_vf(x)>\lambda \}$, from which we have $f_\lambda=f$ on $Q_x$ and hence $M'_v(f_\lambda)(x)>\lambda$. Thus, \eqref{eq:MtMt} holds. Recall that $M'_v$ is bounded from $L^1(uv)$ to $L^{1, \infty}(uv)$ for all $u \in A_1$ and $v \in A_{\infty}$. Hence, this and \eqref{eq:MtMt} imply 
\begin{align}\label{eq:fuv}
&u_N v(\{x\in\Rn: M'_vf(x)>\lambda\})
\le u_N v(\{x\in\Rn: M'_v(f_\lambda)(x)>\lambda\}) 
\nonumber \\
&\qquad\lesssim \frac{1}{\lambda} \int_{\Rn} |f_\lambda| u_N v\, dx 
=\frac{1}{\lambda} \int_{\{x\in\Rn: M'_vf(x)>\lambda\}} |f| u_N v\, dx.
\end{align}
Recall that for any $0<r<p$ and for any measurable set $E$ with $w(E)<\infty$, 
\begin{equation}\label{eq:Kolm}
\int_{E}|f|^r w\, dx \le \frac{p}{p-r} w(E)^{1-\frac{r}{p}} \|f\|_{L^{p,\infty}(w)}^r. 
\end{equation}
The fact $uv \in L^1_{\loc}(\Rn)$ implies that $u_N v(\{x\in\Rn: M'_vf(x)>\lambda\}) \le uv(B(0, N))<\infty$. Thus, by \eqref{eq:fuv} and \eqref{eq:Kolm}, 
\begin{align*}
u_N v(\{x\in\Rn: M'_vf(x)>\lambda\})
\lesssim p' \lambda^{-1} \|f\|_{L^{p,\infty}(u_N v)} \, u_N v(\{x\in\Rn: M'_vf(x)>\lambda\})^{1-\frac1p},
\end{align*}
and so we get
\begin{equation*}
\lambda \, u_N v(\{x\in\Rn: M'_vf(x)>\lambda\})^{\frac1p} \lesssim p' \|f\|_{L^{p,\infty}(uv)}. 
\end{equation*}
Accordingly, it follows from the monotone convergence theorem that \eqref{eq:weakM} holds. Applying \eqref{eq:Saw-4}, \eqref{eq:TMAi} and \eqref{eq:weakM}, we deduce that 
\begin{align*}
\norm{\frac{\mathbb{T}_i (fv)}{v}}_{L^{1+\epsilon, \infty}(uv)} 
\lesssim \norm{\frac{M^i (fv)}{v}}_{L^{1+\epsilon, \infty}(uv)}
=\norm{(M'_v)^i f}_{L^{1+\epsilon, \infty}(uv)} 
\lesssim \|f\|_{L^{1+\epsilon, \infty}(uv)}.  
\end{align*} 
This shows \eqref{eq:SawLp} and completes the proof.  
\end{proof}

\subsection{Littlewood-Paley operators}\label{sec:LPO}
Given $\rho \in \C$, $\Omega \in L^1(\Sn)$, a radial function $h$ on $\Rn$, and $\lambda>1$, we define the parametric Marcinkiewicz integrals 
\begin{align*}
\mu^{\rho}_{\Omega,h,S}(f)(x) &:= \bigg(\iint_{\Gamma(x)} 
|\Theta^{\rho}_{\Omega,h}f(y, t)|^2 \frac{dy dt}{t^{n+1}}\bigg)^{\frac12}, 
\\
\mu^{*,\rho}_{\Omega,h,\lambda}(f)(x) &:= \bigg(\iint_{\R^{n+1}_{+}} 
\Big(\frac{t}{t+|x-y|}\Big)^{n\lambda} |\Theta^{\rho}_{\Omega,h}f(y, t)|^2 \frac{dy dt}{t^{n+1}}\bigg)^{\frac12}, 
\end{align*}
where $\Gamma(x) := \{(y, t) \in \R^{n+1}_{+}: |x-y|<t\}$ and 
\begin{align*}
\Theta^{\rho}_{\Omega,h}f(x, t) = \frac{1}{t^{\rho}} \int_{|x-y| \leq t} 
\frac{\Omega(x-y)h(x-y)}{|x-y|^{n-\rho}}  f(y) dy.  
\end{align*}

If $h \equiv 1$, the operator $\mu^{\rho}_{\Omega,h,S}$ was introduced by H\"{o}rmander \cite{H1} in the higher dimension. If $\rho=1$ and $h \equiv 1$, it is the usual Marcinkiewicz integral corresponding to the Littlewood-Paley $g$-function introduced by Stein \cite{S1}, where $L^p$-boundedness ($1<p\leq 2$) of $\mu^{1}_{\Omega,1,S}$ was established for $\Omega \in \Lip_{\alpha}(\Sn)$ ($0<\alpha:=\Re(\rho) \leq 1$). This result was extended to the case $\rho>0$ and $1<p<\infty$ in \cite{H1}. 

For $1 \leq q \leq \infty$, denote 
\begin{align*}
\ell^{\infty}(L^q)(\R_{+}) := \bigg\{h \in L_{\loc}^1(\R_{+}): \sup_{j \in \Z} 
\bigg(\int_{2^j}^{2^{j+1}} |h(r)| \frac{dr}{r} \bigg)^{\frac1q}<\infty \bigg\}. 
\end{align*}
When $q=\infty$, $\ell^{\infty}(L^q)(\R_{+})$ is understood as $L^{\infty}(\R_{+})$.

In \cite[Lemma~3]{DLY}, Ding et al. showed that if $\lambda>1$, $\rho \in \C$ with $\Re(\rho)>0$, $h \in \ell^{\infty}(L^q)(\R_{+})$ with $q \in(1,\infty]$ and $\Omega \in L\log^{+}L(\Sn)$ is homogeneous of degree zero on $\Rn$ and $\int_{\Sn} \Omega \ d\sigma=0$, then for every weight $w$,  
\begin{align}\label{eq:Mar}
\|\mu^{*,\rho}_{\Omega,h,\lambda}(f)\|_{L^2(w)} \leq C_{n,\rho,\lambda} \|f\|_{L^2(Mw)}. 
\end{align}
Observe that for any $\lambda>1$, 
\begin{align}\label{eq:S-lambda}
\mu^{\rho}_{\Omega,h,S}(f)(x) \leq 2^{n\lambda} \mu^{*,\rho}_{\Omega,h,\lambda}(f)(x),\quad x \in \Rn. 
\end{align}

In view of Theorem \ref{thm:wMw}, \eqref{eq:Mar} and \eqref{eq:S-lambda} give the following conclusion.  

\begin{theorem}\label{thm:Mar} 
Let $\lambda>1$, $\rho \in \C$ with $\Re(\rho)>0$, a radial function $h \in \ell^{\infty}(L^q)(\R_{+})$ with $q \in(1,\infty]$. Let $\Omega \in L\log^{+}L(\Sn)$ be homogeneous of degree zero on $\Rn$ and $\int_{\Sn} \Omega \ d\sigma=0$. Suppose that $\X$ be a RIBFS over $(\Rn, dx)$ with $q_{\X}<\infty$ such that $\X^{\frac12}$ is also a BFS. Then for every weight $w$,  
\begin{align}\label{eq:Mar-1} 
\|\mu^{\rho}_{\Omega,h,S}(f)\|_{\X_w} &\leq C \|f (Mw/w)^{\frac12}\|_{\X_w}, 
\\
\label{eq:Mar-2} \|\mu^{*,\rho}_{\Omega,h,\lambda}(f)\|_{\X_w} &\leq C \|f (Mw/w)^{\frac12}\|_{\X_w}. 
\end{align}
\end{theorem}

\begin{remark}
Theorem $\ref{thm:Mar}$ covers several known results in Lebesgue space. Let $\X=L^p(\Rn)$ and $2\leq p<\infty$. Then we see that both $\X$ and $\X^{\frac12}$ are BFS and $q_{\X}=p$. Hence, Theorem $\ref{thm:Mar}$ contains the unweighted inequalities in Theorems $1$ and $2$ and the weighted $L^2$ inequality for $A_1$ weight in Corollary 1 in \cite{DLY}. Moreover, \eqref{eq:Mar-1} and \eqref{eq:Mar-2} extend the Fefferman-Stein inequalities in Theorem $2$ and Corollary $2$ in \cite{XDY} to the general Banach function spaces. 
\end{remark}

Now let us turn to another type of square functions. Given $\alpha, \beta \in \R$ and $\lambda>0$, we define the square functions by  
\begin{align*}
g_{\alpha,\beta}(f)(x) &:= \bigg(\iint_{\Gamma_{\alpha}(x)} 
|f*\phi_t(y)|^2 \frac{dy}{t^{(1-\alpha)n+2\beta}} \frac{dt}{t} \bigg)^{\frac12}, 
\\
g^{*}_{\alpha,\beta,\lambda}(f)(x) &:= \bigg(\int_{t^{\alpha} \leq 1} \int_{\Rn} 
\Big(\frac{t^{1-\alpha}}{t^{1-\alpha}+|x-y|} \Big)^{n\lambda}
|f*\phi_t(y)|^2 \frac{dy}{t^{(1-\alpha)n+2\beta}} \frac{dt}{t} \bigg)^{\frac12}, 
\end{align*}
where $\phi_t(x)=\frac{1}{t^n}\phi(\frac{x}{t})$, $\phi$ is a smooth function with suitable compact Fourier support away from the origin, and $\Gamma_{\alpha}(x) := \{(y, t) \in \Rn \times \R_{+}: 0<t^{\alpha} \leq 1, |y-x| \leq t^{1-\alpha}\}$. A new type of fractional maximal operator is given by 
\begin{align}\label{eq:Mab-def}
\mathcal{M}_{\alpha, \beta}f(x) := \sup_{(y, r) \in \Gamma_{\alpha}(x)} 
\frac{1}{|B(y, r)|^{1-2\beta/n}} \int_{B(y, r)} |f(z)| dz. 
\end{align}
Observe that $g^{*}_{\alpha,\beta,\lambda}$ is a pointwise majorant of $g_{\alpha,\beta}$. If $\alpha=\beta=0$, then 
$\mathcal{M}_{\alpha, \beta}$ is the Hardy-Littlewood maximal function, $g_{\alpha,\beta}$ and $g^{*}_{\alpha,\beta,\lambda}$ are the generalization of Littlewood-Paley operators with Poisson kernels studied in \cite{S2}. 

Recently, Beltran and Bennett \cite{BB} proved that for every $\alpha, \beta \in \R$ and every weight $w$, 
\begin{align}\label{eq:f-g(f)}
\int_{\Rn} |f(x)|^2 w(x) dx \leq C \int_{\Rn} g_{\alpha,\beta}f(x)^2 \mathcal{M}_{\alpha, \beta} M^4w(x) dx, 
\end{align}
for all functions $f$ such that $\supp(\widehat{f}) \subset \{\xi \in \Rn: |\xi|^{\alpha} \geq 1\}$. Conversely, they obtained that 
for all $\alpha \in \R$, $\lambda>1$, and for every weight $w$, 
\begin{align}\label{eq:g(f)-f}
\int_{\Rn} g^{*}_{\alpha,0,\lambda}f(x)^2 w(x) dx \leq C \int_{\Rn} |f(x)|^2 M^2 w(x) dx. 
\end{align}

By Theorem \ref{thm:uvw}, \eqref{eq:MkMk} and \eqref{eq:g(f)-f}, we get the two-weight inequalities on the weighted Banach function spaces as follows. 

\begin{theorem}
Let $u$, $v$, $w_1$ and $w_2$ be weights on $\Rn$. Suppose that $\X_u$ and $\X_v$ are respectively BFS over $(\Rn, u\, dx)$ and $(\Rn, v\,dx)$ such that $\Y_u=\X_u^{\frac12}$ and $\Y_v=\X_v^{\frac12}$ are BFS.  Assume that 
\begin{align}
\|(M_{L\log L}f) w_2^{-2}v^{-1}\|_{\Y'_v} \le C \|fw_1^{-2}u^{-1}\|_{\Y_u'}, \quad\forall f \in \M.  
\end{align}
Then for all $\alpha \in \R$ and $\lambda>1$, 
\begin{align}
\|(g^{*}_{\alpha,0,\lambda}f) w_1\|_{\X_u} \leq C \|f w_2\|_{\X_v}. 
\end{align} 
\end{theorem}

\subsection{Fourier integral operators}\label{sec:FIO}
For a function $\m$ on $\Rn$, the Fourier multiplier $T_{\m}$ is defined by 
\begin{align*}
T_{\m} f(x) := \int_{\Rn} \m(\xi) \widehat{f}(\xi) e^{-2\pi i x \cdot \xi} d\xi, 
\end{align*}
for all functions $f \in \S(\Rn)$. 
For each $\alpha, \beta \in \R$, let $\mathscr{C}(\alpha, \beta)$ be the class of functions $\m:\R \to \C$ for which 
$\supp(\m) \subset \{\xi: |\xi|^{\alpha} \geq 1 \}$,  $\sup_{\xi \in \Rn} |\xi|^{\beta} |\m(\xi)| < \infty$,  and 
\begin{align*}
\sup_{r^{\alpha} \geq 1} \sup_{\substack{I \subset [r, 2r] \\ \ell(I)=r^{1-\alpha}}} 
r^{\beta} \int_{\pm I} |\m'(\xi)| d\xi < \infty. 
\end{align*}
Let $\mathscr{D}(\alpha, \beta)$ be the collection of all functions $\m:\Rn \to \C$ such that 
\begin{align*}
\sup_{B} \dist(0, B)^{\beta+(1-\alpha)\theta} |B|^{-\frac12} \|\m\Psi_B\|_{\dot{H}^{\theta}} < \infty,  
\end{align*}
for all $0 \leq \theta \leq \sigma$ and some $\sigma>n/2$, uniformly over normalized bump functions $\Psi_B$ adapted to an $\alpha$-subdyadic ball $B$. Here, $\dot{H}^{\theta}$ denotes the usual homogeneous Sobolev spaces of order $\theta$, and $\Psi$ is a suitable smooth function with compact support away from the origin. By the $\alpha$-subdyadic ball, we mean that a Euclidean ball $B \subset \Rn$ satisfies that $\dist(0, B)^{\alpha} \geq 1$ and $r(B) \simeq \dist(0, B)^{1-\alpha}$. 

Bennett \cite{Ben} showed that for every $\m \in \mathscr{C}(\alpha, \beta)$ with $\alpha, \beta \in \R$ and for every weight $w$, 
\begin{align}\label{eq:Tm}
\int_{\R} |T_{\m} f|^2 w\, dx \leq C \int_{\R} |f|^2 M^6 \mathcal{M}_{\alpha, \beta}M^4w\, dx, 
\end{align}
where $\mathcal{M}_{\alpha, \beta}$ is defined in \eqref{eq:Mab-def}. By H\"{o}lder's inequality, one has 
\begin{align}\label{eq:Mab-M}
\mathcal{M}_{\alpha, \beta}w(x) &\leq \sup_{x \in B(y, r^{1-\alpha})} r^{2\beta} 
\bigg(\fint_{B(y, r)}w^s dz \bigg)^{\frac1s}
\nonumber \\
&\leq \sup_{x \in B(y, r^{1-\alpha})} r^{2\beta-\frac{\alpha}{s}} 
\bigg(\fint_{B(y, r^{1-\alpha})} w^s dz \bigg)^{\frac1s}
\leq M(w^s)(x)^{\frac1s}, 
\end{align}
provided $\alpha=2s\beta$ and $s \geq 1$. Therefore, from \eqref{eq:Tm} and \eqref{eq:Mab-M}, we obtain that for every $\m \in \mathscr{C}(\alpha, \alpha/2)$ and for every weight $w$
\begin{align}\label{eq:Tm-1}
\int_{\R} |T_{\m} f|^2 w\, dx \leq C \int_{\R} |f|^2 M^{11}w\, dx. 
\end{align}

In addition, it was shown in \cite[Theorem~1]{BB} that for every $\m \in \mathscr{D}(\alpha, \beta)$ with $\alpha, \beta \in \R$,  
\begin{align}\label{eq:g(Tmf)-g(f)}
g_{\alpha, \beta}(T_{\m} f)(x) \lesssim g^{*}_{\alpha,0,\lambda}(f)(x), \quad \lambda=2\sigma/n>1. 
\end{align}
Then invoking \eqref{eq:f-g(f)}, \eqref{eq:g(Tmf)-g(f)}, \eqref{eq:g(f)-f} and \eqref{eq:Mab-M}, we conclude that for every weight $w$, 
\begin{align}\label{eq:Tm-2}
\int_{\Rn} |T_{\m} f|^2 w\, dx &\leq C \int_{\Rn} |f|^2 M^2 \mathcal{M}_{\alpha, \alpha/2}M^4w\, dx
\leq C \int_{\Rn} |f|^2 M^7 w\, dx, 
\end{align}
for any $\m \in \mathscr{D}(\alpha, \alpha/2)$ supported in $\{\xi \in \Rn: |\xi|^{\alpha} \geq 1\}$.  

Recall the pseudo-differential operator $T_{\sigma}$ defined in \eqref{eq:def-Ta}. It was proved in \cite{Bel} that if $\sigma \in S^m_{\varrho,\delta}$ with $m \in \R$, $0 \leq \delta \leq \varrho \leq 1$ and $\delta < 1$, then for any weight $w$, 
\begin{equation}\label{eq:Ta-M}
\int_{\Rn} |T_{\sigma} f|^2 w \ dx 
\lesssim \int_{\Rn}|f|^2 M^2 \mathcal{M}_{\varrho, m} M^5 w\, dx, 
\end{equation}
where 
\begin{equation*}
\mathcal{M}_{\varrho, m} w(x) := \sup _{(y, r) \in \Lambda_{\varrho}(x)} \frac{w(B(y,r))}{|B(y, r)|^{1+2m/n}},  
\end{equation*}
and $\Lambda_{\varrho}(x) := \{(y, r) \in \Rn \times(0,1) :|y-x| \leq r^{\varrho} \}$. Observe that $B(y,r) \subset B(x,2r^{\varrho})$ for any $(y, r) \in \Lambda_{\varrho}(x)$. Thus, picking $m \in \R$ and $\varrho \in [0, 1]$ such that $\varrho=1+2m/n$, we get 
\begin{equation}\label{eq:MwMw}
\mathcal{M}_{\varrho, m} w(x)  
\leq \sup _{(y, r) \in \Lambda_{\varrho}(x)} \frac{|B(x,2r^{\varrho})|}{|B(y, r)|^{1+2m/n}} 
\frac{w(B(x,2r^{\varrho}))}{|B(x,2r^{\varrho})|} \leq C Mw(x). 
\end{equation}
Gathering \eqref{eq:Ta-M} and \eqref{eq:MwMw}, we conclude that 
\begin{equation}\label{eq:Ta}
\int_{\Rn} |T_{\sigma} f|^2 w\, dx \leq C \int_{\Rn} |f|^2 M^8 w\, dx. 
\end{equation}
As a consequence, combining \eqref{eq:Tm-1}, \eqref{eq:Tm-2}, \eqref{eq:Ta} and Theorem \ref{thm:ABC}, we conclude the following estimates.  

\begin{theorem}
Let $v$ be a weight on $\Rn$, $\X_v$ be a BFS over $(\Rn, v\,dx)$ such that $\X_v^{\frac12}$ is also a BFS. Assume that there exist Young functions $A$ and $B$ such that $A^{-1}(t) B^{-1}(t) \lesssim \Phi^{-1}(t)$, and that $M'_{B,v}$ is bounded on $(\X_v^{\frac12})'$. Then for every weight $u$,  
\begin{equation}\label{eq:Tfw-fMAw}
\|(\mathbf{T} f) u\|_{\X_v}  \leq C \|f M_A(u^2)^{\frac12}\|_{\X_v}, 
\end{equation}  
provided that the pair $(\mathbf{T}, \Phi)$ satisfies one of the following:  
\begin{enumerate}
\item $(\mathbf{T}, \Phi)=(T_{\m}, t \log(e+t)^{10})$, where $\m \in \mathscr{C}(\alpha, \alpha/2)$ with $\alpha \in \R$; 
\item $(\mathbf{T}, \Phi)=(T_{\m}, t \log(e+t)^{6})$, where $\m \in \mathscr{D}(\alpha, \alpha/2)$ supported in $\{\xi: |\xi|^{\alpha} \geq 1\}$ with $\alpha \in \R$; 
\item $(\mathbf{T}, \Phi)=(T_{\sigma}, t \log(e+t)^{7})$, where $\sigma \in S^m_{\varrho,\delta}$ with $m=-n(1-\varrho)/2$, $0 \leq \delta \leq \varrho \leq 1$ and $\delta < 1$.  
\end{enumerate}
\end{theorem}



\begin{thebibliography}{00}



\bibitem{AMR}N. Accomazzo, J.C. Mart\'{i}nez-Perales, and I.P. Rivera-R\'{i}os, 
\emph{On Bloom type estimates for iterated commutators of fractional integrals}, 
Indiana Univ. Math. J. \textbf{69} (2020), no. 4, 1207--1230.


\bibitem{AF}G. Anatriello and M.R. Formica, 
\emph{Weighted fully measurable grand Lebesgue spaces and the maximal theorem}, 
Ricerche Mat. \textbf{65} (2016), 221--233. 


\bibitem{ACM}T. Anderson, D. Cruz-Uribe, and K. Moen, 
\emph{Logarithmic bump conditions for Calder\'{o}n-Zygmund operators on spaces of homogeneous type}, 
Publ. Mat. \textbf{59} (2015), 17--43. 


\bibitem{AKMP}I.U. Asekritova, N.Y. Krugljak, L. Maligranda, and L.E. Persson, 
\emph{Distribution and rearrangement estimates of the maximal functions and interpolation}, 
Studia Math. \textbf{124} (1997), 107--132.


\bibitem{Bel}D. Beltran, 
\emph{Control of pseudodifferential operators by maximal functions via weighted inequalities},  
Trans. Amer. Math. Soc. \textbf{37} (2019), 3117--3143.


\bibitem{BB}D. Beltran and J. Bennett, 
\emph{Subdyadic square functions and applications to weighted harmonic analysis}, 
Adv. Math. \textbf{307} (2017), 72--99.  


\bibitem{Ben}J. Bennett, 
\emph{Optimal control of singular Fourier multipliers by maximal operators}, 
Anal. PDE \textbf{7} (2014), 1317--1338.


\bibitem{BR}C. Bennett and K. Rudnick, 
\emph{On Lorentz-Zygmund spaces}, 
Dissertationes Math. \textbf{175} (1980), 1--67.


\bibitem{BS}C. Bennett and R. Sharpley,
\emph{Interpolation of operators}, Pure and Applied
Mathematics, Vol. \textbf{129}, Academic Press Inc., Boston, MA, 1988.


\bibitem{BMMST}\'{A}. B\'{e}nyi, J.M. Martell, K. Moen, E. Stachura and R.H. Torres, 
\emph{Boundedness results for commutators with $\BMO$ functions via weighted estimates: a comprehensive approach}, 
Math. Ann. \textbf{376} (2020), 61--102. 


\bibitem{Blo85}S. Bloom, 
\emph{A commutator theorem and weighted $\BMO$}, 
Trans. Amer. Math. Soc. \textbf{292} (1985), 103--122.


\bibitem{Buc}S.M. Buckley, 
\emph{Estimates for operator norms on weighted spaces and reverse Jensen inequalities},  
Trans. Am. Math. Soc. \textbf{340} (1993), 253--272. 


\bibitem{CMM}M. Cao, J.J. Mar\'{i}n, and J.M. Martell, 
\emph{Extrapolation on function and modular spaces, and applications}, 
Adv. Math. \textbf{406} (2022), No. 108520.


\bibitem{COY}M. Cao, A. Olivo, and K. Yabuta, 
\emph{Extrapolation for multilinear compact operators and applications}, 
Trans. Amer. Math. Soc. \textbf{375} (2022), 5011--5070.


\bibitem{CXY}M. Cao, Q. Xue, and K. Yabuta,  
\emph{Weak and strong type estimates for the multilinear pseudo-differential operators},  
J. Funct. Anal. \textbf{278} (2020), 108454.

 
\bibitem{CRS}M.J. Carro, J.A. Raposo, and J. Soria, 
\emph{Recent developments in the theory of Lorentz spaces and weighted inequalities},  
Mem. Amer. Math. Soc. \textbf{187} (2007). 


\bibitem{CLPR}M.E. Cejas, K. Li, C. Perez, and I.P. Rivera-Rios, 
\emph{Vector-valued operators, optimal weighted estimates and the $C_p$ condition}, 
Sci. China Math. (2020), to appear. 


\bibitem{CGPSZ}G. Citti, L. Grafakos, C. P\'{e}rez, A. Sarti, and X. Zhong, 
\emph{Harmonic and geometric analysis}, 
Advanced Courses in Mathematics. CRM Barcelona. Birkh\"{a}user, Basel (2014). 


\bibitem{CF09}D. Cruz-Uribe and A. Fiorenza, 
\emph{$L\log L$ results for the maximal operator in variable $L^p$ spaces}, 
Trans. Amer. Math. Soc. \textbf{361}(2009), 2631--2647. 


\bibitem{CMP05}D. Cruz-Uribe, J.M. Martell, and C. P\'{e}rez, 
\emph{Weighted weak type inequalities and a conjecture of Sawyer}, 
Int. Math. Res. Not. \textbf{30} (2005), 1849--1871. 


\bibitem{CMP11}D. Cruz-Uribe, J.M. Martell, and C. P\'erez,
\emph{Weights, extrapolation and the theory of Rubio de Francia}, 
Operator Theory: Advances and Applications, Vol.\,\textbf{215}, Birkh\"auser/Springer Basel AG, Basel, 2011. 


\bibitem{CP99}D. Cruz-Uribe and C. P\'erez, 
\emph{Sharp two-weight, weak-type norm inequalities for singular integral operators},  
Math. Res. Lett. \textbf{6} (1999), 1--11. 


\bibitem{CP}D. Cruz-Uribe and C. P\'erez, 
\emph{Two weight extrapolation via the maximal operator}, 
J. Funct. Anal. \textbf{174} (2000), 1--17. 


\bibitem{CGMP}G. P. Curbera, J. Garc\'{i}a-Cuerva, J.M. Martell, and C. P\'{e}rez, 
\emph{Extrapolation with weights, rearrangement-invariant function spaces, 
modular inequalities and applications to singular integrals}, 
Adv. Math. \textbf{203} (2006), 256--318. 


\bibitem{DHHR}L. Diening, P. Harjulehto, P. H\"{a}st\"{o}, and M. Ru$\check{z}$i$\check{c}$ka, 
\emph{Lebesgue and Sobolev spaces with variable exponents}, 
Lecture Notes in Mathematics, \textbf{2017}, Springer, Heidelberg, 2011.


\bibitem{DLY}Y. Ding, S. Lu, and K. Yabuta,  
\emph{A problem on rough parametric Marcinkiewicz functions},  
J. Austral. Math. Soc. \textbf{72} (2002), 13--21.


\bibitem{EOP}W.D. Evans, B. Opic, and L. Pick, 
\emph{Interpolation of operators on scales of generalized Lorentz-Zygmund spaces}, 
Math. Nachr. \textbf{182} (1996), 127--181.


\bibitem{FS}C. Fefferman and E.M. Stein, 
\emph{Some maximal inequalities},  
Amer. J. Math. \textbf{93} (1971), 107--115.


\bibitem{FGJ}A. Fiorenza, B. Gupta, and P. Jain,   
\emph{The maximal theorem for weighted grand Lebesgue spaces}, 
Studia Math. \textbf{188} (2008), 123--133.


\bibitem{FG}M.R. Formica and R. Giova, 
\emph{Boyd indices in generalized grand Lebesgue spaces and applications}, 
Mediterr. J. Math. \textbf{12} (2015), 987--995. 


\bibitem{GR}J. Garc{\'\i}a-Cuerva and J. Rubio de Francia,
\emph{Weighted norm inequalities and related topics}, 
North Holland, Amsterdam, 1985.


\bibitem{G1}L. Grafakos, 
\emph{Classical Fourier analysis}, 
Third edition, Springer, New York, 2014.


\bibitem{G2}L. Grafakos, 
\emph{Modern Fourier analysis}, 
Third edition, Springer, New York, 2014.


\bibitem{HaP}P.A. Hagelstein and I. Parissis, 
\emph{Weighted Solyanik estimates for the Hardy-Littlewood maximal operator 
and embedding of $A_{\infty}$ into $A_p$}, 
J. Geom. Anal. \textbf{26}  (2016), 924--946. 


\bibitem{Ha}P.A. H\"{a}st\"{o},  
\emph{The maximal operator on generalized Orlicz spaces}, 
J. Funct. Anal. \textbf{269} (2015), 4038--4048. 


\bibitem{HM12}S. Hofmann and J.M. Martell, 
\emph{$A_{\infty}$ estimates via extrapolation of Carleson measures and applications to divergence form elliptic operators}, 
Trans. Amer. Math. Soc. \textbf{364} (2012), 65--101. 


\bibitem{HM14}S. Hofmann and J.M. Martell, 
\emph{Uniform rectifiability and harmonic measure I: uniform rectifiability implies Poisson kernels in $L^p$}, 
Ann. Sci. Ecole Norm. Sup. \textbf{47} (2014), 577--654. 


\bibitem{H1}L. H\"{o}rmander, 
\emph{Estimates for translation invariant operators in $L^p$ spaces}, 
Acta Math. \textbf{104} (1960), 93--140.


\bibitem{Hy12}T. Hyt\"{o}nen, 
\emph{The sharp weighted bound for general Calder\'{o}n-Zygmund operators}, 
Ann. of Math. (2) \textbf{175} (2012), 1473--1506.


\bibitem{HL-1}T. Hyt\"{o}nen and S. Lappas, 
\emph{Extrapolation of compactness on weighted spaces},  
 Revista Mat. Iberoam. (2022), to appear. 
	
	
\bibitem{HL-2}T. Hyt\"{o}nen and S. Lappas, 
\emph{Extrapolation of compactness on weighted spaces II: Off-diagonal and limited range estimates}, 
\url{https://arxiv.org/abs/2006.15858}. 


\bibitem{HP12}T. Hyt\"onen and C. P\'erez, 
\emph{Sharp weighted bounds involving $A_{\infty}$}, 
Anal. PDE \textbf{6} (2013), 777--818. 


\bibitem{HP}T. Hyt\"onen and C. P\'erez, 
\emph{The $L(\log L)^{\epsilon}$ endpoint estimate for maximal singular integral operators},  
J. Math. Anal. Appl. \textbf{428} (2015), 605--626.


\bibitem{Kar}G.A. Karagulyan, 
\emph{Exponential estimates for the Calder\'{o}n-Zygmund operator and related problems of Fourier series},  
Mat. Zametki \textbf{71} (2002), 398--411. 


\bibitem{LMPT}M.T. Lacey, K. Moen, C. P\'{e}rez, and R.H. Torres, 
\emph{Sharp weighted bounds for fractional integral operators}, 
J. Funct. Anal.  \textbf{259} (2010), 1073--1097. 


\bibitem{L10}A.K. Lerner, 
\emph{A pointwise estimate for the local sharp maximal function with applications to singular integrals}, 
Bull. London Math. Soc. \textbf{42} (2010), 843--856. 


\bibitem{L14}A.K. Lerner,  
\emph{On sharp aperture-weighted estimates for square functions}, 
J. Fourier Anal. Appl. \textbf{20} (2014), 784--800. 



\bibitem{LNO}A. K. Lerner,  F. Nazarov, and S. Ombrosi , 
\emph{On the sharp upper bound related to the weak Muckenhoupt--Wheeden conjecture}, 
Anal. PDE \textbf{13} (2020), 1939--1954. 


\bibitem{LOP}A. K. Lerner, S. Ombrosi and C. P\'{e}rez, 
\emph{Sharp $A_1$ bounds for Calder\'{o}n-Zygmund operators and the relationship 
with a problem of Muckenhoupt and Wheeden}, 
Int. Math. Res. Not.  \textbf{161} (2008). 


\bibitem{LOP1}A.K. Lerner, S. Ombrosi, and C. P\'{e}rez, 
\emph{$A_1$ bounds for Calder\'{o}n-Zygmund operators related to a problem of Muckenhoupt and Wheeden}, 
Math. Res. Lett. \textbf{16} (2009), 149--156. 


\bibitem{LOP2}A.K. Lerner, S. Ombrosi, and C. P\'{e}rez, 
\emph{Weak type estimates for singular integrals related to a dual problem of Muckenhoupt-Wheeden},  
J. Fourier Anal. Appl. \textbf{15} (2009), 394--403.


\bibitem{LOR}A.K. Lerner, S. Ombrosi, and I.P. Rivera-R\'{i}os, 
\emph{On pointwise and weighted estimates for commutators of Calder\'{o}n-Zygmund operators}, 
Adv. Math. \textbf{319} (2017), 153--181. 


\bibitem{LP}A. K. Lerner and C. P\'{e}rez,  
\emph{A new characterization of the Muckenhoupt $A_p$ weights through an extension of the Lorentz-Shimogaki theorem},  
Indiana Univ. Math. J. \textbf{56} (2007), 2697--2722. 


\bibitem{LiOP}K. Li, S. Ombrosi and C. P\'{e}rez, 
\emph{Proof of an extension of E. Sawyer's conjecture about weighted mixed weak-type estimates}, 
Math. Ann. \textbf{374} (2019), 907--929. 


\bibitem{LPRR}K. Li, C. P\'{e}rez, I.P. Rivera-R\'{i}os, and L. Roncal, 
\emph{Weighted norm inequalities for rough singular integral operators}, 
J. Geom. Anal. \textbf{29} (2019), 2526--2564. 


\bibitem{MMMMM}J.J. Mar\'{i}n, J.M. Martell, D. Mitrea, I. Mitrea, and M. Mitrea,
\emph{Singular integral operators, quantitative flatness, and boundary problems}, 
book manuscript, 2020. 


\bibitem{Mon}S.J. Montgomery-Smith, 
\emph{The Hardy operator and Boyd indices}, 
Interaction between Functional Analysis, Harmonic Analysis, and Probability (Columbia, MO, 1994), 
Lecture Notes in Pure and Applied Mathematics, vol. \textbf{175}, Dekker, New York, 1996, pp. 359--364.


\bibitem{Muc} B. Muckenhoupt, 
\emph{Weighted norm inequalities for the Hardy maximal function}, 
Trans. Amer. Math. Soc. \textbf{165} (1972), 207--226.


\bibitem{MW}B. Muckenhoupt and R. Wheeden,  
\emph{Some weighted weak-type inequalities for the Hardy-Littlewood maximal function and the Hilbert transform}, 
Indiana Math. J. \textbf{26} (1977), 801--816. 


\bibitem{NRVV}F. Nazarov, A. Reznikov, V. Vasyunin, and A. Volberg,  
\emph{On weak weighted estimates of the martingale transform and a dyadic shift}, 
Anal. PDE \textbf{11} (2018), 2089--2109. 


\bibitem{ON}R. O'Neil, 
\emph{Fractional integration in Orlicz spaces},  
Trans. Amer. Math. Soc. \textbf{115} (1965), 300--328. 


\bibitem{Ort}C. Ortiz-Caraballo, 
\emph{Quadratic $A_1$ bounds for commutators of singular integrals with $\BMO$ functions}, 
Indiana Univ. Math. J. \textbf{60} (2011), 2107--2130.


\bibitem{OPR}C. Ortiz-Caraballo, C. P\'{e}rez, and E. Rela,
\emph{Exponential decay estimates for singular integral operators},
Math. Ann. \textbf{357} (2018), 1217--1243.


\bibitem{P95}C. P\'{e}rez, 
\emph{On sufficient conditions for the boundedness of the Hardy-Littlewood maximal 
            operator between weighted $L^p$-spaces with different weights}, 
Proc. Lond. Math. Soc.  \textbf{71} (1995), 135--157. 


\bibitem{P00}C. P\'{e}rez, 
\emph{Sharp weighted inequalities for the vector-valued maximal function}, 
Trans. Amer. Math. Soc. \textbf{352} (2000), 3265--3288.


\bibitem{RT}M.C. Reguera and C. Thiele, 
\emph{The Hilbert transform does not map $L^1(Mw)$ to $L^{1, \infty}(w)$},  
Math. Res. Lett. \textbf{19} (2012), 1--7. 


\bibitem{Riv}I.P. Rivera-R\'{i}os, 
\emph{Improved $A_1$-$A_{\infty}$ and related estimates for commutators of rough singular integrals},  
Proc. Edinb. Math. Soc.  \textbf{61} (2018), 1069--1086.


\bibitem{Rub}J.L. Rubio de Francia, 
\emph{Factorization theory and $A_p$ weights},  
Am. J. Math. \textbf{106} (1984), 533--547.


\bibitem{Saw}E.T. Sawyer,  
\emph{A weighted weak type inequality for the maximal function},  
Proc. Am. Math. Soc. \textbf{93} (1985), 610--614.


\bibitem{Shen}Z. Shen,  
\emph{Extrapolation for the $L^p$ Dirichlet problem in Lipschitz domains}, 
Acta Math. Sin. \textbf{35} (2019), 1074--1084.


\bibitem{S1}E.M. Stein, 
\emph{On the functions of Littlewood-Paley, Lusin and Marcinkiewicz}, 
Trans. Amer. Math. Soc. \textbf{88} (1958), 430--466. 


\bibitem{S2}E.M. Stein, 
\emph{Singular integrals and differentiability properties of functions},  
Princeton Math. Ser., vol. 30, Princeton University Press, Princeton, NJ, 1970.


\bibitem{WWZ}Y. Wen, H. Wu, and J. Zhang, 
\emph{Weighted variation inequalities for singular integrals and commutators}, 
J. Math. Anal. Appl. \textbf{485} (2020), 123825. 


\bibitem{XDY}Q. Xue, Y. Ding, and K. Yabuta, 
\emph{Weighted estimate for a class of Littlewood-Paley operators},  
Taiwanese J. Math. \textbf{11} (2007), 339--365.



\end{thebibliography}
\end{document}